\documentclass[11pt]{article}
\usepackage{amsfonts,amsmath,amsthm,amssymb,enumerate,color}
\usepackage[makeroom]{cancel}
\usepackage{authblk}

\numberwithin{equation}{section}
\newcommand{\M}{\mathbb{M}}

\newcommand{\beq}{\begin{equation}}
\newcommand{\eeq}{\end{equation}}
\newcommand{\bea}{\begin{eqnarray}}
\newcommand{\eea}{\end{eqnarray}}
\newcommand{\beas}{\begin{eqnarray*}}
\newcommand{\eeas}{\end{eqnarray*}}

\newtheorem{theorem}{Theorem}[section]

\newtheorem{definition}[theorem]{Definition}
\newtheorem{proposition}[theorem]{Proposition}

\newtheorem{corollary}[theorem]{Corollary}
\newtheorem{lemma}[theorem]{Lemma}
\newtheorem{remark}[theorem]{Remark}
\newtheorem{example}[theorem]{Example}
\newtheorem{examples}[theorem]{Examples}
\newtheorem{foo}[theorem]{Remarks}

\newcommand{\p}{\partial}

\newcommand{\bM}{\mathbb M}

\newcommand{\Rn}{\mathbb R^n}
\newcommand{\Om}{\Omega}

\newcommand{\Ho}{\mathcal H}

\newcommand{\R}{\mathbb R}

\newcommand{\ve}{\varepsilon}

\title{Geometric Inequalities  on Riemannian and sub-Riemannian manifolds by heat semigroups techniques \\ Levico Summer School 2017}

\author{Fabrice Baudoin\footnote{ Research was supported in part by NSF Grant DMS 1660031}
}
\affil{Department of Mathematics, University of Connecticut,\par
   341 Mansfield Road, Storrs, CT  06269-1009, USA\par
   \texttt{fabrice.baudoin@uconn.edu}\vspace{1em}}

\begin{document}
\maketitle

\begin{abstract}
In those lecture notes, we review some applications of heat semigroups methods in Riemannian and sub-Riemannian geometry. 
\end{abstract}

\tableofcontents

\section{Subelliptic diffusion operators}

\subsection{Diffusion operators}

\begin{definition}\label{D:do}
A differential operator $L$ on $\mathbb{R}^n$,  is called a diffusion operator if it can be written
\[
L=\sum_{i,j=1}^n \sigma_{ij} (x) \frac{\partial^2}{ \partial x_i \partial x_j} +\sum_{i=1}^n b_i (x)\frac{\partial}{\partial x_i},
\]
where $b_i$ and $\sigma_{ij}$ are continuous functions on $\mathbb{R}^n$, and if for every $x \in \mathbb{R}^n$ the matrix $\sigma(x) \overset{def}{=} (\sigma_{ij}(x))_{1\le i,j\le n}$ is symmetric and nonnegative. 
\end{definition}
If for every $x \in \mathbb{R}^n$ the matrix $(\sigma_{ij}(x))_{1\le i,j\le n}$ is positive definite, then the operator $L$ is said to be elliptic.
The canonical example of an elliptic diffusion operator is the Laplace operator in $\mathbb{R}^n$:
\[
\Delta=\sum_{i=1}^n \frac{\partial^2}{\partial x_i^2}.
\]
However, Definition \ref{D:do} includes also non-elliptic operators such as for instance the heat operator
\[
H = \Delta - \frac{\partial}{\partial t},
\]
or very degenerate operators such as the so-called Kolmogorov operator in $\R^3$,
\[
\mathcal K = \frac{\partial^2}{\partial x^2} + x\frac{\partial }{\partial y} - \frac{\partial}{\partial t}.
\]
One of the most basic properties of diffusion operators is that they satisfy a maximum principle. Before we state it let us recall a simple  result from linear algebra.

\begin{lemma}\label{L:nonneg}
Let $A$ and $B$ be two symmetric and nonnegative matrices, then
\[
\operatorname{tr} (AB) \ge 0.
\]
\end{lemma}
\begin{proof}
Since $A$ is symmetric and non negative, there exists a symmetric and nonnegative matrix $S$ such that $S^2=A$. We have then
\[
\operatorname{tr} ( AB)=\operatorname{tr}(S^2 B)= \operatorname{tr} (S B S)=\operatorname{tr} (^t S BS).
\]
The matrix $^t S BS$ is seen to be symmetric and nonnegative and therefore tr$(^t S BS)\ge 0$.

\end{proof}

\begin{proposition}[Maximum principle for diffusion operators]\label{P:maxpr}
Let $f:\mathbb{R}^n \rightarrow \mathbb{R}$ be a $C^2$ function that attains a local minimum at $x$. If $L$ is a diffusion operator, then $Lf (x) \ge 0$.
\end{proposition}

\begin{proof}
Let 
\[
L=\sum_{i,j=1}^n \sigma_{ij} (x) \frac{\partial^2}{ \partial x_i \partial x_j} +\sum_{i=1}^n b_i (x)\frac{\partial}{\partial x_i},
\]
and let  $f:\mathbb{R}^n \rightarrow \mathbb{R}$ be a $C^2$ function that attains a local minimum at $x$. We have
\begin{align*}
Lf(x) & =\sum_{i,j=1}^n \sigma_{ij} (x) \frac{\partial^2 f}{ \partial x_i \partial x_j} (x) \\
 & =\operatorname{tr} \left( \sigma (x) \operatorname{Hess} f (x) \right),
\end{align*}
where $\operatorname{Hess} f(x)  = (\frac{\partial^2 f}{ \partial x_i \partial x_j} (x))$ is the Hessian matrix of $f$. Since, by the assumption, we have  $\operatorname{Hess } f (x) \ge 0$, the desired conclusion immediately follows from Lemma \ref{L:nonneg}.

\end{proof}

Combined with the linearity,  Proposition \ref{P:maxpr} actually characterizes the diffusion operators.

\begin{theorem}\label{caracterisation}
Let $L: C^{\infty} (\mathbb{R}^n) \rightarrow C^0 (\mathbb{M})$
be an operator such that:
\begin{itemize}
\item[1)] $L$ is linear;
\item[2)] for any $f \in C^{\infty} (\mathbb{R}^n)$ such that $f$ has a local minimum at $x$, one has $Lf (x) \ge 0$.
\end{itemize}
Then, $L$ is a diffusion operator.
\end{theorem}

\begin{proof}
Suppose $L$ satisfy the above properties. It is readily seen from 2) that $L1 = 0$. Fix now $y \in \mathbb{R}^n$. Our first observation is that if $g\in C^\infty(U)$, where $U\subset \Rn$ is a neighborhood of $y$, then 
\begin{equation}\label{Lcube}
L(\|\cdot- y\|^3 g)(y) =0.
\end{equation}
To see this, for $\varepsilon >0$ consider the function
\[
x \rightarrow \| x-y \|^3 g(x) +\varepsilon \| x -y \|^2.
\]
Since such function admits a local minimum at $y$, by 2) we have
\[
L( \|\cdot -y \|^3 g) (y)  \ge - \varepsilon L ( \|\cdot -y \|^2)(y).
\]
Letting $\varepsilon \to 0$, we thus obtain
\[
L( \|\cdot-y \|^3 g) (y)  \ge 0.
\]
By considering now the function
\[
x \rightarrow \| x-y \|^3 g(x) -\varepsilon \| x -y \|^2,
\]
which has a local maximum at $y$, we obtain similarly that 
\[
L( \| x-y \|^3 g) (y)  \le 0.
\]
In conclusion, we have proved \eqref{Lcube}.

In order to show that $L$ is a diffusion operator we now consider $f\in C^\infty(\Rn)$, and would like to show that
\begin{equation}\label{do}
Lf(y)=\sum_{i,j=1}^n \sigma_{ij} (y) \frac{\partial^2 f}{ \partial x_i \partial x_j}(y) +\sum_{i=1}^n b_i (y)\frac{\partial f}{\partial x_i}(y).
\end{equation}
for continuous functions $\sigma_{ij}$ and $b_i $, with $(\sigma_{ij})_{1\le i,j\le n} \ge 0$. By the Taylor formula there exists  a neighborhood $U$ of $y$, and a function $g\in C^\infty(U)$, such that for any $x\in U$ one has
\[
f(x)=f(y)+\sum_{i=1}^n (x_i -y_i) \frac{\partial f}{\partial x_i}(y)+\frac{1}{2} \sum_{i,j=1}^n (x_i-y_i)(x_j-y_j)  \frac{\partial^2 f}{\partial x_i \partial x_j} (y) +  \| x-y \|^3 g(x).
\]
By applying the operator $L$ to the previous identity, and by taking 1), the fact that $L1 = 0$, and \eqref{Lcube} into account, we obtain
\[
Lf(y)=\sum_{i=1}^n L(x_i -y_i)(y)  \frac{\partial f}{\partial x_i}(y)+\frac{1}{2} \sum_{i,j=1}^n L((x_i-y_i)(x_j-y_j))(y) \frac{\partial^2 f}{\partial x_i \partial x_j} (y).
\]
By denoting
\[
b_i(y)= L(x_i -y_i)(y),\ \ \ \ \ \ \text{and}\ \   
\sigma_{ij} (y)=\frac{1}{2} L((x_i-y_i)(x_j-y_j))(y),
\]
we conclude that \eqref{do} holds. Furthermore, since $L$ transforms smooth into continuous functions, the functions $b_i$'s and $\sigma_{ij}$'s are continuous.
To complete the proof it would thus suffice to show that the matrix 
$(\sigma_{ij}(y))_{1\le i,j \le n}$ be nonnegative. i.e., that for every $\xi\in \Rn$ one has $\sum_{i,j=1}^n  \sigma_{ij} (y) \xi_i \xi_j \ge 0$.
Indeed, by 1) again we have 
\[
\sum_{i,j=1}^n \sigma_{ij} (y)  \xi_i \xi_j =\frac{1}{2} L( \langle \xi, x-y \rangle^2)(y).
\]
Since the function $x \rightarrow \langle \xi, x-y \rangle^2$ attains a local minimum at $y$, we obtain from 2)
\[
L( \langle \xi, x-y \rangle^2)(y) \ge 0.
\]
From the arbitrariness of $y\in \Rn$, the proof is completed.

\end{proof}

The previous characterization of diffusion operators is intrinsic and has the advantage of being coordinate free. This suggests to adopt the following natural definition on manifolds.
 
 \begin{definition}\label{D:dm}
 Let $\mathbb{M}$ be a smooth manifold. A diffusion operator $L$ on $\mathbb{M}$ is an operator $L: C^{\infty} (\mathbb{M}) \rightarrow C^0 (\mathbb{M})$
such that:
\begin{enumerate}
\item $L$ is linear;
\item for any $f \in C^{\infty} (\mathbb{M})$ which has a local minimum at $x$, one has $Lf (x) \ge 0$.
\end{enumerate}
 \end{definition}
 
Of course, in any local chart, a diffusion operator $L$ reads as
\[
L=\sum_{i,j=1}^n \sigma_{ij} (x) \frac{\partial^2}{ \partial x_i \partial x_j} +\sum_{i=1}^n b_i (x)\frac{\partial}{\partial x_i}.
\]
where $\sigma$ is symmetric and nonnegative.

\subsection{Subelliptic diffusion operators}

In this Section, we give without proofs some results of regularity theory for subelliptic diffusion operators. For an introduction to those topics in the case of elliptic operators, we refer to Folland \cite{Fol} in the Euclidean case and \cite{Gri} in the manifold case. Regularity theory for hypoelliptic diffusion operators was developed by L. H\"ormander \cite{Hor} and we refer to \cite{Bra} for a recent presentation.

\begin{definition} Let $L$ be a diffusion operator with smooth coefficients which is defined on an open set $\Omega \subset \mathbb{R}^n$. We say that $L$ is subelliptic on $\Omega$, if for every compact set $K \subset \Omega$, there exist a constant $C$ and $\varepsilon >0$ such that for every $u \in C^\infty_0(K)$,
\begin{align}\label{subelliptic estimate}
\| u \|^2_{(2\varepsilon)} \le C \left( \| Lu \|_2^2+\|u\|^2_2\right).
\end{align}
\end{definition}

In the above definition, we denoted for $s\in \mathbb{R}$,  the Sobolev norm $$\|f \|^2_{(s)}=\int_{\mathbb{R}^n} | \hat{f} (\xi) |^2 (1+\| \xi \|^2)^s d\xi <+\infty ,$$ where $ \hat{f} (\xi)$ is the Fourier transform of $f$, and $\| \cdot \|_2$ is the classical $L^2$ norm. It is well-known that elliptic operators are subelliptic in the sense of the previous definition with $\varepsilon=1$.  There are many interesting examples of diffusion operators which are subelliptic but not elliptic.  Let, for instance,
\[
L=\sum_{i=1}^d V_i^2 +V_0
\]
where $V_0,V_1,\cdots,V_d$ are smooth vector fields defined on an open set $\Omega$. We denote by $\mathfrak{V}$ the Lie algebra generated by the $V_i$'s, $1 \le i \le d$, and for $x \in \Omega$,
\[
\mathfrak{V}(x)=\{ V(x), V \in \mathfrak{V} \}.
\]
The celebrated H\"ormander's theorem states that if  for every $x \in \Omega$, $\mathfrak{V}(x)=\mathbb{R}^n$, then $L$ is a subelliptic operator. In that case $\varepsilon$ is $1/d$, where $d$ is the maximal length of the brackets that are needed to generate $\mathbb{R}^n$.

If $L$ is a subelliptic diffusion operator, using the theory of pseudo-differential operators, it can be proved that the inequality \eqref{subelliptic estimate} self-improves into a family of inequalities of the type 
\[
\| u \|^2_{(2\varepsilon+s)} \le C \left( \| Lu \|_{(s)}^2+\|u\|^2_{(s)}\right), \quad u \in C^\infty_0(K),
\]
where $s\in \mathbb{R}$ and the constant $C$ only depends on $K$ and $s$. This implies, in particular, by a usual bootstrap argument and Sobolev lemma that subelliptic operators are hypoelliptic. Iterating the latter inequality also leads to
\[
\| u \|^2_{(2k\varepsilon)} \le C \sum_{j=0}^k \| L^j u\|^2_2, \quad u \in C^\infty_0(K), 
\]
where $k \ge 0$. This may be used to bound derivatives of $u$ in terms of $L^2$ norms to iterated powers of $u$. Indeed, if $\alpha$ is a multi-index and $k$ is such that $4k\varepsilon> 2 | \alpha | +n$, then we get $\sup_{x \in K} | \partial^\alpha u (x) |^2 \le C\| u \|^2_{(2k\varepsilon)} $ and therefore
\[
\sup_{x \in K} | \partial^\alpha u (x) |^2 \le C'   \sum_{j=0}^k \| L^j u\|^2_2.
\]
Along the same lines, we also get the following result.

\begin{proposition}\label{regula}
Let $L$ be a subelliptic diffusion operator with smooth coefficients on an open set $\Omega \subset \mathbb{R}^n$. Let $u \in L^2(\Omega)$ such that, in the sense of distributions,
\[
Lu,L^2u,\cdots, L^ku \in L^2(\Omega),
\]
for some positive integer $k$. Let $K$ be a compact subset of $\Omega$ and denote by $\varepsilon$ the same constant as in \eqref{subelliptic estimate}. If $k>\frac{n}{4 \varepsilon}$, then $u$ is a continuous function on the interior of $K$ and there exists a positive constant $C$ such that 
\[
\sup_{x \in K} | u(x) |^2 \le C  \sum_{j=0}^k \|L^j u \|^2_{ L^2(\Omega)}.
\]
More generally, if $k>\frac{m}{2\varepsilon}+\frac{n}{4\varepsilon}$ for some non negative integer $m$, then $u$ is $m$-times continuously differentiable in the interior of $K$ and  there exists a positive constant $C$  such that 
\[
\sup_{|\alpha| \le m}  \sup_{x \in K} |\partial^\alpha u(x) |^2 \le C  \sum_{j=0}^k \|L^j u \|^2_{ L^2(\Omega)} .
\]
\end{proposition}

As a consequence of the previous result, we see  in particular that
\[
\bigcap_{k \ge 0} \mathcal{D}(L^k) \subset C^\infty(\mathbb{R}^n).
\]

We can define subelliptic operators on a manifold by using charts:

\begin{definition}
Let $L$ be a diffusion operator on a manifold $\mathbb{M}$. We say that $L$ is subelliptic on $\mathbb{M}$ if it is in any local chart.
\end{definition}

Proposition \ref{regula} can then be extended to the manifold case:

\begin{proposition}\label{regula}
Let $\mathbb{M}$ be a manifold endowed with a smooth positive measure $\mu$, and let  $L$ be a subelliptic diffusion operator with smooth coefficients on an open set $\Omega \subset \mathbb{M}$. Let $u \in L^2(\Omega,\mu)$ such that, in the sense of distributions,
\[
Lu,L^2u,\cdots, L^ku \in L^2(\Omega,\mu),
\]
for some positive integer $k$. Let $K$ be a compact subset of $\Omega$. There exists a constant $\varepsilon >0$ such that  If $k>\frac{n}{4 \varepsilon}$, then $u$ is a continuous function on the interior of $K$ and there exists a positive constant $C$ such that 
\[
\sup_{x \in K} | u(x) |^2 \le C \sum_{j=0}^k \|L^j u \|^2_{ L^2(\Omega,\mu)} .
\]
More generally, if $k>\frac{m}{2\varepsilon}+\frac{n}{4\varepsilon}$ for some non negative integer $m$, then $u$ is $m$-times continuously differentiable in the interior of $K$ and  there exists a positive constant $C$  such that 
\[
\sup_{|\alpha| \le m}  \sup_{x \in K} |\partial^\alpha u(x) |^2 \le C  \sum_{j=0}^k \|L^j u \|^2_{ L^2(\Omega,\mu)} .
\]
\end{proposition}

\subsection{The distance associated to subelliptic diffusion operators}

Let $L$  be a subelliptic diffusion operator defined on a manifold $\M$. For every smooth functions $f,g: \mathbb{M} \rightarrow \mathbb{R}$, let us define the so-called \textit{ carr\'e du champ}, which is   the  symmetric first-order differential form defined by:
\[
\Gamma (f,g) =\frac{1}{2} \left( L(fg)-fLg-gLf \right).
\]
A straightforward computation shows that if, in a local chart,
\[
L=\sum_{i,j=1}^n \sigma_{ij} (x) \frac{\partial^2}{ \partial x_i \partial x_j} +\sum_{i=1}^n b_i (x)\frac{\partial}{\partial x_i},
\]
then, in the same chart
\[
\Gamma (f,g)=\sum_{i,j=1}^n  \sigma_{ij} (x) \frac{\partial f}{\partial x_i} \frac{\partial g}{\partial x_j}.
\]
As a consequence, for every smooth function $f$,
\[
\Gamma(f) \ge 0.
\]

\begin{definition}
An absolutely continuous curve $\gamma: [0,T] \rightarrow \bM$ is said to be subunit for the operator $L$ if for every smooth function $f : \M \to \mathbb{R}$ we have $ \left| \frac{d}{dt} f ( \gamma(t) ) \right| \le \sqrt{ (\Gamma f) (\gamma(t)) }$.  We then define the subunit length of $\gamma$ as $\ell_s(\gamma) = T$. 
\end{definition}

Given $x, y\in \M$, we indicate with 
\[
S(x,y) =\{\gamma:[0,T]\to \M\mid \gamma\ \text{is subunit for}\ \Gamma, \gamma(0) = x,\ \gamma(T) = y\}.
\]
In these lecture notes we always assume that 
\[
S(x,y) \not= \emptyset,\ \ \ \ \text{for every}\ x, y\in \M.
\]
If $L$ is an elliptic operator or if $L$ is a sum of squares operator that satisfies H\"ormander's condition, then this assumption is satisfied (this is the so-called Chow theorem).

Under such assumption  it is easy to verify that
\begin{equation}\label{ds}
d(x,y) = \inf\{\ell_s(\gamma)\mid \gamma\in S(x,y)\},
\end{equation}
defines a true distance on $\M$. This is the intrinsic distance associated to the subelliptic operator $L$. A beautiful result by Fefferman and Phong \cite{FP} relates  the subellipticity of $L$ to the size of the balls for this metric:

\begin{theorem}[Fefferman-Phong]\label{FPtheorem}
Let $\Omega$ be a relatively compact open subset of $\M$. For some $\varepsilon >0$, there are constants $r_0 >0$ and $C$ such that for all $x$ and $r$,
\[
B(x,r)\cap \Omega \subset B_d (x,Cr^\varepsilon) \cap \Omega
\]
whenever $0 \le r < r_0$. $B_d$ denotes here the ball for the metric $d$ and $B$ the ball for an arbitrary Riemannian metric on $\Omega$.
\end{theorem}

A corollary of this result is that the topology induced by $d$ coincides with the manifold topology of $\M$. The distance $d$ can also be computed using the following definition: 

\begin{proposition}
For every $x,y \in \M$,
\begin{equation}\label{di}
d(x,y)=\sup \left\{ |f(x) -f(y) | , f \in  C^\infty(\bM) , \| \Gamma(f) \|_\infty \le 1 \right\},\ \ \  \ x,y \in \bM.
\end{equation}
\end{proposition}
\begin{proof}
Let $x,y \in \mathbb{M}$. We denote 
\[
\delta (x,y)=\sup \{ | f(x)-f(y) |, f \in C_0^\infty(\mathbb{M}), \| \Gamma(f) \|_\infty \le 1  \}.
\]
Let $\gamma: [0,T] \to \mathbb{M}$ be a sub-unit curve such that
\[
\gamma(0)=x, \gamma(T)=x.
\]
We have $ \left| \frac{d}{dt} f ( \gamma(t) ) \right| \le \sqrt{ (\Gamma f) (\gamma(t)) }$, therefore, if $\Gamma(f) \le 1$,
\begin{align*}
\left| f(y)-f(x) \right| \le T.
\end{align*}
As a consequence
\[
\delta (x,y) \le d(x,y).
\]

We now prove the converse inequality which is trickier. The idea would be to use the function $f(y)=d(x,y)$ that satisfies $| f(x) -f(y)|=d(x,y)$ and $"\Gamma(f,f) =1"$. However, giving a precise meaning to $\Gamma(f,f) =1$, is not so easy, because it turns out that $f$ is not everywhere differentiable.  If the operator $L$ is elliptic, then the distance $d$ is Riemannian and one can use an approximate identity to regularize $f$, that is to find $C^\infty$ functions $\eta_\varepsilon$ such that $| f(y)-\eta_\varepsilon(y) | \le \varepsilon$ and $\| \Gamma(\eta_\varepsilon) \|_\infty \le 1 +C\varepsilon$ for some constant $C>0$ (see \cite{MR898496} and \cite{MR521983} for more details). We have then $| \eta_\varepsilon(y) - \eta_\varepsilon(x) | \le (1 +C\varepsilon)\delta (x,y)$, which implies $d(x,y) \le \delta(x,y)$ by letting $\varepsilon \to 0$. Thus, if $L$ is elliptic then $d(x,y)=\delta(x,y)$. If $L$ is only subelliptic, we consider an elliptic diffusion operator $\Delta$ on $\M$ and consider the sequence of operators $L_k=L+\frac{1}{k} \Delta$. We denote by $d_k$ the distance associated to $L_k$. It is easy to see that $d_k$ increases with $k$ and that $d_k(x,y) \le d(x,y)$. We can find a curve $\gamma_k:[0,1] \to \M$, such that $\gamma(0)=x,\gamma(1)=y$ and for every $f \in C^\infty(M)$,
\[
\left| \frac{d}{dt} f(\gamma_k(t)) \right|^2 \le \left(d^2_k(x,y)+\frac{1}{k} \right)\left( \Gamma(f) (\gamma_k(t)) + \frac{1}{k} \Gamma_\Delta(f)(\gamma_k(t)) \right),
\]
where $\Gamma_\Delta$ is the carr\'e du champ operator of $\Delta$. Since $d_k \le d$, we see  that the sequence $\gamma_k$ is uniformly equicontinuous. As a consequence of the Arzel\`a-Ascoli theorem, we deduce that there exists a subsequence which we continue to denote  $\gamma_k$ that converges uniformly to a curve  $\gamma:[0,1] \to \M$, such that $\gamma(0)=x,\gamma(1)=y$ and for every $f \in C^\infty(M)$,
\[
\left| \frac{d}{dt} f(\gamma (t)) \right|^2 \le \sup_k d^2_k(x,y) \Gamma(f) (\gamma_k(t)).
\]
By definition of $d$, we deduce $d(x,y) \le  \sup_k d_k(x,y)$. As a consequence, we proved that $d(x,y)=\lim_{k \to \infty} d_k(x,y)$. Since it is clear that 
\[
d_k(x,y)= \sup \left\{ |f(x) -f(y) | , f \in  C^\infty(\bM) , \left\| \Gamma(f)  + \frac{1}{k} \Gamma_\Delta(f) \right\|_\infty \le 1 \right\} \le \delta(x,y),
\]
we finally conclude that $d(x,y) \le \delta(x,y)$, hence $d(x,y)=\delta(x,y)$.
\end{proof}

A straightforward corollary of the previous proposition is the following useful result:

\begin{corollary}
If $f \in C^\infty(\M)$ satisfies $\Gamma(f)=0$, then $f$ is constant.
\end{corollary}

The following theorem  is  known as the Hopf-Rinow theorem, it provides a necessary and sufficient condition for the completeness of the metric space  $(\mathbb{M},d)$.

\begin{theorem}(Hopf-Rinow theorem)
The metric space $(\mathbb{M},d)$ is complete (i.e. Cauchy sequences are convergent) if and only the compact sets are the closed and bounded  sets.
\end{theorem}

\begin{proof}
It is clear that if closed and bounded sets are compact then the metric space $(\mathbb{M},d)$ is complete; It comes from the fact that Cauchy sequence are convergent if and only if they have at least one cluster value.
So, we need to prove that closed and bounded sets for the distance $d$ are compact provided that $(\mathbb{M},d)$ is complete. To check this, it is enough to prove that closed balls are compact.

Let $x \in \mathbb{M}$. Observe that if the closed ball $\bar{B}(x, r)$ is compact for some $r>0$, then $\bar{B}(x, \rho)$ is closed for any $\rho <r$. Define 
\[
R=\sup \{ r>0, \bar{B}(x, r) \text{ is compact } \}.
\]
Since $d$ induces the manifold topology of $\mathbb{M}$, $R>0$. Let us assume that $R<\infty$ and let us show that it leads to a contradiction.

We first show that $\bar{B}(x, R)$ is compact. Since $(\mathbb{M},d)$ is assumed to be complete, it suffices to prove that $\bar{B}(x, R)$ is totally bounded: That is, for every $\varepsilon >0$ there is a finite set $S_\varepsilon$ such that every point of $\bar{B}(x, R)$  lies in a $\varepsilon$-neighborhood of $S_\varepsilon$. 

So, let $\varepsilon >0$ small enough. By definition of $R$, the ball $\bar{B}(x, R-\varepsilon / 4)$ is compact; It is therefore totally bounded. We can find a finite set $S=\{ y_1,\cdots,y_N\}$ such that every point of $\bar{B}(x, R-\varepsilon / 4)$  lies in a $\varepsilon / 2$-neighborhood of $S$. Let now $y \in \bar{B}(x, R)$. We claim that there exists $y' \in \bar{B}(x, R-\varepsilon / 4)$ such that $d(y,y') \le \varepsilon /2 $. If $y \in \bar{B}(x, R-\varepsilon / 4)$, there is nothing to prove, we may therefore assume that $y \notin \bar{B}(x, R-\varepsilon / 4)$. Consider then a sub-unit curve $\gamma: [0,R+\varepsilon / 4]\to \M$ such that $\gamma(0)=x$,  $\gamma(R+\varepsilon/2)=y$. Let
\[
\tau =\inf \{t, \gamma(t) \notin \bar{B}(x, R-\varepsilon / 4) \}.
\]
We have $\tau \ge R-\varepsilon / 4$. On the other hand,
\[
d(\gamma(\tau), \gamma(R+\varepsilon/2)) \le R+\varepsilon/4 -\tau.
\]
As a consequence,
\[
d(\gamma(\tau),y) \le \varepsilon /2.
\]
In all cases, there exists therefore $y' \in \bar{B}(x, R-\varepsilon / 4)$ such that $d(y,y') \le \varepsilon /2 $. We may then pick $y_k$ in $S$ such that $d(y_k,y') \le \varepsilon / 2$. From the triangle inequality, we have $d(y,y_k) \le \varepsilon$. 

So, at the end, it turns out that every point of $\bar{B}(x, R)$  lies in a $\varepsilon$-neighborhood of $S$. This shows that $\bar{B}(x, R)$ is totally bounded and therefore compact because $(\mathbb{R}^n,d)$ is assumed to be complete. Actually, the previous argument shows more, it shows that if every point of $\bar{B}(x, R)$  lies in a $\varepsilon /2$-neighborhood of a finite $S$, then every point of $\bar{B}(x, R+\varepsilon/4)$  will lie  $\varepsilon$-neighborhood of $S$, so that the ball $\bar{B}(x, R+\varepsilon/4)$ is also compact. This contradicts  the definition of $R$. Therefore every closed ball is a compact set, due to the arbitrariness of $x$.  
\end{proof}

\subsection{Essentially self-adjoint subelliptic operators}

We consider on a smooth manifold $\mathbb{M}$ a subelliptic  diffusion operator $L$. We assume that $L$  is symmetric with respect to some measure $\mu$, that is: For every smooth and compactly supported functions $f,g : \mathbb{M} \rightarrow \mathbb{R}$,
\[
\int_{\mathbb{M}} g Lf d\mu= \int_{\mathbb{M}} fLg d\mu.
\]

We observe that if $\mu$ is a symmetric measure for $L$ as above, then in the sense of distributions
\[
L'\mu=0,
\] 
where $L'$ is the adjoint of $L$ in distribution sense. As a consequence,  $\mu$  needs to be a smooth measure.

\

 The \textit{carr\'e du champ} introduced above allows to integrate by parts: For every $f,g \in C_0^\infty(\mathbb{M})$,
\[
\int_\mathbb{M} fLg d\mu=-\int_\mathbb{M} \Gamma(f,g) d\mu =\int_{\mathbb{M}} g Lf d\mu.
\]

\

The operator $L$ on its domain $\mathcal{D}(L)= C_0^\infty (\mathbb{M})$ is a densely defined  non positive symmetric operator on the Hilbert space  $L^2(\M,\mu)$. However, it is not self-adjoint, indeed it is easily checked that
\[
\left\{ f \in  C^\infty (\mathbb{M}), \| f \|_{L^2(\M,\mu)} +\|Lf \|_{L^2(\M,\mu)} < \infty \right\} \subset \mathcal{D}(L^*).
\] 

\ 

A famous theorem of Von Neumann asserts that any non negative and symmetric operator may be extended into a self-adjoint operator. The following construction, due to Friedrich, provides a canonical non negative self-adjoint extension.

\begin{theorem}\label{Friedrichs_extension}(Friedrichs extension)
On the Hilbert space $L^2(\M,\mu)$, there exists a densely defined  non positive self-adjoint extension of $L$. 
\end{theorem}

\begin{proof}
 On $C_0^\infty (\mathbb{M})$, let us consider the following norm
\[
\| f\|^2_{\mathcal{E}}=\| f \|^2_{L^2(\M,\mu) }+\int_\mathbb{M} \Gamma(f,f)d\mu.
\]
By completing $C_0^\infty (\mathbb{M})$ with respect to this norm, we get a Hilbert space $(\mathcal{H},\langle \cdot , \cdot \rangle_{\mathcal{E}})$. Since for $f \in C_0^\infty (\mathbb{M})$,  $\| f \|_{L^2(\M,\mu) } \le \| f\|_{\mathcal{E}}$, the injection map 
\[
\iota : ( C_0^\infty (\mathbb{M}), \| \cdot \|_{\mathcal{E}}) \rightarrow (L^2(\M,\mu), \| \cdot \|_{L^2(\M,\mu) })
\]
is continuous and it may therefore be extended into a continuous map 
\[
\bar{\iota}: (\mathcal{H}, \| \cdot \|_{\mathcal{E}}) \rightarrow (L^2(\M,\mu), \| \cdot \|_{L^2(\M,\mu) }).
\]
 Let us show that $\bar{\iota}$ is injective so that $\mathcal{H}$ may be identified with a subspace of $L^2(\M,\mu) $. So, let $f \in \mathcal{H}$ such that $\bar{\iota} (f)=0$. We can find a sequence $f_n \in C_0^\infty (\mathbb{M})$, such that $\| f_n -f \|_{\mathcal{E}} \to 0$ and $\| f_n \|_{\mathbf{L}_{\mu}^2 (\mathbb{M}) } \to 0$. We have
\begin{align*}
\| f \|_{\mathcal{E}}  &=\lim_{m,n \to + \infty} \langle f_n, f_m \rangle_{\mathcal{E}} \\
 &=\lim_{m \to + \infty} \lim_{n \to + \infty}  \langle f_n,f_m \rangle_{\mathbf{L}_{\mu}^2 (\mathbb{R}^n,\mathbb{R}) }- \langle f_n,Lf_m \rangle_{\mathbf{L}_{\mu}^2 (\mathbb{R}^n,\mathbb{R}) } \\
 &=0,
\end{align*}
thus $f=0$ and $\bar{\iota}$ is injective.  Let us now consider  the map
\[
B=\bar{\iota} \cdot \bar{\iota}^* : L^2(\M,\mu)  \rightarrow L^2(\M,\mu) .
\]
It is well defined due to the fact that since $\bar{\iota}$ is bounded, it is easily checked that
\[
\mathcal{D}(\bar{\iota}^*)= L^2(\M,\mu).
\] 

 Moreover, $B$ is easily seen to be symmetric, and thus self-adjoint because its domain is equal to $L^2(\M,\mu)$. Also, it is readily checked that the injectivity of $\bar{\iota}$ implies the injectivity of $B$. Therefore, we deduce that the inverse
\[
A=B^{-1}: \mathcal{R} (\bar{\iota} \cdot \bar{\iota}^*) \subset L^2(\M,\mu) \rightarrow L^2(\M,\mu)
\]
is a densely defined self-adjoint operator on $L^2(\M,\mu)$.  Now, we observe that for $f,g \in C_0^\infty (\mathbb{M})$,
\begin{align*}
 & \langle f,g \rangle_{L^2(\M,\mu)}-\langle Lf,g \rangle_{L^2(\M,\mu)} \\
=&  \langle \bar{i}^{-1}(f),\bar{i}^{-1}(g) \rangle_\mathcal{E} \\
=& \langle (\bar{i}^{-1})^* \bar{i}^{-1} f,g \rangle_{L^2(\M,\mu)} \\
=& \langle  (\bar{i} \bar{i}^*)^{-1} f,g \rangle_{L^2(\M,\mu)}
\end{align*}
Thus $A$ and $\mathbf{Id}-L$ coincide on $C_0^\infty (\mathbb{M})$.
By  defining,
\[
-\bar{L}=A-\mathbf{Id},
\]
we get the required self-adjoint extension of $-L$.
\end{proof}

\begin{remark}
The operator $\bar{L}$, as constructed above, is called the Friedrichs extension of $L$.
\end{remark}

\begin{definition}
If $\bar{L}$ is the unique non positive self-adjoint extension of $L$, then the operator $L$ is said to be essentially self-adjoint on $C_0^\infty (\mathbb{M})$. In that case, there is no ambiguity and we shall denote $\bar{L}=L$.
\end{definition}

We have the following first criterion for essential self-adjointness.

\begin{lemma}
If for some $\lambda >0$, 
\[
\mathbf{Ker} (-L^* +\lambda \mathbf{Id} )= \{ 0 \},
\]
then the operator $L$ is essentially self-adjoint on $C_0^\infty (\mathbb{M})$.
\end{lemma}

\begin{proof}
We make the proof for $\lambda=1$ and let the reader adapt it for $\lambda \neq 0$.

Let $-\tilde{L}$ be a non negative self-adjoint extension of $-L$. We want to prove that actually, $-\tilde{L}=-\bar{L}$. The assumption 
\[
\mathbf{Ker} (-L^* + \mathbf{Id} )= \{ 0 \}
\]
implies that $C_0^\infty (\mathbb{M})$ is dense in $\mathcal{D}(-L^*)$ for the norm
\[
\| f \|^2_{\mathcal{E}}=\| f \|^2_{L^2(\M,\mu) } -\langle f , L^* f \rangle_{L^2(\M,\mu)}.
\]
Since, $-\tilde{L}$ is a non negative self-adjoint extension of $-L$, we have 
\[
\mathcal{D}(-\tilde{L}) \subset  \mathcal{D}(-L^*).
\]
The space $C_0^\infty (\mathbb{M})$ is therefore dense in $\mathcal{D}(-\tilde{L}) $ for the norm $\| \cdot \|_{\mathcal{E}}$.  At that point, we use some notations introduced in the proof of the Friedrichs extension (Theorem \ref{Friedrichs_extension}). Since  $C_0^\infty (\mathbb{M})$ is dense in $\mathcal{D}(-\tilde{L}) $ for the norm $\| \cdot \|_{\mathcal{E}}$, we deduce that the equality
\begin{align*}
  \langle f,g \rangle_{L^2(\M,\mu)}-\langle \tilde{L}f,g \rangle_{L^2(\M,\mu)}=  \langle \bar{i}^{-1}(f),\bar{i}^{-1}(g) \rangle_\mathcal{E} ,
\end{align*}
which is obviously satisfied for $f,g \in \mathcal{C}_0^\infty (\mathbb{M})$ actually also holds for $f,g \in \mathcal{D}(\tilde{L})$. From the definition of the Friedrichs extension, we deduce that $\bar{L}$ and $\tilde{L}$ coincide on $\mathcal{D}(\tilde{L})$. Finally, since these two operators are self adjoint we conclude $\bar{L}=\tilde{L}$.
\end{proof}

\begin{remark}
Given the fact that $-L$ is given here with the domain  $C_0^\infty (\mathbb{M})$, the condition
\[
\mathbf{Ker} (-L^* +\lambda \mathbf{Id} )= \{ 0 \},
\]
is equivalent to the fact that if  $f \in L^2(\M,\mu)$ is a function that satisfies in the sense of distributions
\[
-Lf+\lambda f=0,
\]
then $f=0$.
\end{remark}

As a corollary of the previous lemma,  the following proposition provides a useful sufficient condition for essential self-adjointness that is easy to check for several subelliptic diffusion operators (including Laplace-Beltrami operators on complete Riemannian manifolds).

\begin{proposition}\label{completeness self adjoint} 
If the diffusion operator $L$ is a subelliptic diffusion operator and if there exists an increasing sequence $h_n\in C_0^\infty (\mathbb{M}))$, $0 \le h_n \le 1$,  such that $h_n\nearrow 1$ on
$\mathbb{M}$, and $||\Gamma (h_n)||_{\infty} \to 0$, as $n\to \infty$,  then the operator $L$ is essentially self-adjoint on $C_0^\infty (\mathbb{M})$.
\end{proposition}

\begin{proof}
Let $\lambda >0$. According to the previous lemma, it is enough to check that  if $L^* f=\lambda f$ with $\lambda >0$, then
$f=0$. As it was observed above, $L^* f=\lambda f$ is equivalent to the fact that, in sense of distributions, $Lf =\lambda f$.
From the hypoellipticity of $L$, we deduce therefore that $f$ is  a smooth function. Now, for $h \in C_0^\infty (\mathbb{M})$,
\begin{align*}
\int_{\mathbb{M}} \Gamma( f, h^2f) d\mu & =-\langle f, L(h^2f)\rangle_{L^2(\M,\mu)}\\
 & =-\langle L^*f ,h^2f \rangle_{L^2(\M,\mu)}\\
  & =-\lambda  \langle f,h^2f\rangle_{L^2(\M,\mu)}\\
   &=-\lambda \langle f^2,h^2 \rangle_{L^2(\M,\mu)} \\
   & \le 0.
\end{align*}
Since
\[
 \Gamma( f, h^2f)=h^2 \Gamma (f)+2 fh \Gamma(f,h),
 \]
 we deduce that
 \[
 \langle h^2, \Gamma (f) \rangle_{L^2(\M,\mu)}+2 \langle fh, \Gamma(f,h)\rangle_{L^2(\M,\mu)} \le 0.
 \]
 Therefore, by Cauchy-Schwarz inequality
 \[
 \langle h^2, \Gamma (f) \rangle_{L^2(\M,\mu)} \le 4 \| f \|_{L^2(\M,\mu)}^2 \| \Gamma (h) \|_\infty.
 \]
 If we now use the sequence $h_n$  and let $n \to \infty$, we obtain $\Gamma(f,f)=0$ and therefore $f=0$, as desired.
\end{proof}

Interestingly, the existence of such a sequence $h_n$ is closely related to the completeness of the metric space $(\M,d)$ where $d$ is the intrinsic distance associated to the operator $L$.

\begin{proposition}\label{L:exhaustion}
There exists an increasing
sequence $h_n\in C^\infty_0(\bM)$, $0 \le h_n \le 1$,  such that $h_n\nearrow 1$ on
$\bM$, and $\| \Gamma(h_n) \|_{\infty} \to 0$, as $n\to \infty$ if and only if the metric space $(\M,d)$ is complete.
\end{proposition}

\begin{proof}
Let us assume that the metric space $(\M,d)$ is complete. Let $\Delta$ be an elliptic operator on $\M$ and denote by $d_R$ the Riemannian distance associated to the operator $L+\Delta$. We have $d_R \le d$. Together with the Fefferman-Phong result recalled in Proposition \ref{FPtheorem}, we deduce that the metric space  $(\M,d_R)$ is complete.

Iif we fix a base point $x_0\in \bM$, we can find an
exhaustion function $\rho\in C^\infty(\bM)$ such that
\[ |\rho - d_R(x_0,\cdot)| \le L,\ \ \ \ \ \   \|\nabla_R \rho \| \le L
\ \ \text{on}\ \bM. \] By the completeness of $(\M,d_R)$ and the
Hopf-Rinow theorem, the level sets $\Omega_s = \{x\in \bM\mid
\rho(x)<s\}$ are relatively compact and, furthermore, $\Omega_s
\nearrow \bM$ as $s\to \infty$. We now pick an increasing sequence
of functions $\phi_n\in C^\infty([0,\infty))$ such that
$\phi_n\equiv 1$ on $[0,n]$, $\phi_n \equiv 0$ outside $[0,2n]$, and
$|\phi_n'|\le \frac{2}{n}$. If we set $h_n(x) = \phi_n(\rho(x))$,
then we have $h_n\in C^\infty_0(\bM)$, $h_n\nearrow1$ on $\bM$ as
$n\to \infty$, and \[ \| \sqrt{\Gamma(h_n)} \|_\infty \le ||\nabla_R h_n||_{\infty} \le \frac{2L}{n}.
\]

\

Conversely, let us assume that here exists an increasing
sequence $h_n\in C^\infty_0(\bM)$, $0 \le h_n \le 1$,  such that $h_n\nearrow 1$ on
$\bM$, and $\| \Gamma(h_n) \|_{\infty} \to 0$, as $n\to \infty$. Let $(x_k)$ be a Cauchy sequence in the metric space $(\M,d)$. We can find $N$ such that $h_N(x_1) \ge 1/2$ and $d(x_1,x_k) \le \frac{1}{4\| \Gamma(h_N) \|_{\infty}} $ for every $k$. For every $k$, we have then $h_N(x_k) \ge \frac{1}{4}$, so that $x_k$ belongs to the support of $h_N$. By compactness, we deduce that $x_k$ is convergent, hence $(\M,d)$ is complete.
\end{proof}

\subsection{The heat semigroup associated to a subelliptic diffusion operator}

We consider in this section a subelliptic diffusion operator $L$ which is essentially self-adjoint on $C^\infty_0(\bM)$. As a consequence, $L$ admits a unique self-adjoint extension (its Friedrichs extension). We shall continue to denote such extension by $L$. The domain of this extension shall be denoted by $\mathcal{D}(L)$. 

If $L=-\int_0^{+\infty} \lambda dE_\lambda$ denotes the spectral
decomposition of $L$ in $L^2 (\bM,\mu)$, then by definition, the
heat semigroup $(P_t)_{t \ge 0}$ is given by $P_t= \int_0^{+\infty}
e^{-\lambda t} dE_\lambda$. It is a one-parameter family of bounded operators on
$L^2 (\bM,\mu)$. Since the quadratic form $-<f,Lf>$ is a Dirichlet
form, we deduce 
that $(P_t)_{t \ge 0}$ is a sub-Markov semigroup: it transforms non-negative functions into non-negative functions and satisfies
\begin{equation}\label{submarkov}
P_t 1 \le 1.
\end{equation}
The sub-Markov property and  Riesz-Thorin interpolation classically allows to construct the semigroup $(P_t)_{t \ge 0}$ in $L^p(\M,\mu)$ and for $f \in L^p(\M,\mu)$ one has
\begin{equation}\label{smp}
||P_tf||_{L^p(\bM , \mu)} \le ||f||_{L^p(\bM, \mu)},\ \  1\le p\le \infty.
\end{equation}

For more details about the functional analytic construction of heat semigroups, we refer fro instance  to \cite{Bau6}.

We now address the proof of the existence and regularity of the heat kernel associated to a a locally subelliptic operator. 

The key estimate to prove the existence of a heat kernel is the following:
\begin{lemma}
Let $K$ be a compact  subset of $\mathbb{M}$. There exists a positive constant $C$  and $k >0$ such that for $f \in L^2(\M,\mu)$:
\[
\sup_{x \in K} | P_t f(x)| \le C \left( 1 +\frac{1}{t^{k}} \right) \| f \|_{L^2(\M,\mu)}.
\]
\end{lemma}

\begin{proof}
Let us first  observe that from the spectral theorem that if $f \in L^2(\M,\mu)$ then for every $k \ge 0$, $L^k P_t f \in L^2(\M,\mu)$ and
\[
 \|L^k P_t f \|_{ L^2(\M,\mu)} \le \left(\sup_{\lambda \ge 0} \lambda^k e^{-\lambda t}\right) \|f \|_{L^2(\M,\mu)}.
 \]
Now, let $K$ be  a compact set of $\mathbb{M}$. From Proposition  \ref{regula}  there exists a positive constant $C$ such that 
\[
\sup_{x \in K} | P_t f(x) |^2 \le C  \sum_{k=0}^{k} \|L^k P_t f \|^2_{ L^2(\M,\mu)} .
\]
Since it is immediately checked that 
 \[
 \sup_{\lambda \ge 0} \lambda^k e^{-\lambda t}=\left( \frac{k}{t}\right)^k e^{-k},
 \]
the bound 
\[
\sup_{x \in K} | P_t f(x)| \le C \left( 1 +\frac{1}{t^{k}} \right) \| f \|_{L^2(\M,\mu)}.
\]
easily follows. 
\end{proof}

We are now in position to prove the smoothing property of the semigroup.

\begin{proposition}
For $f \in L^2(\M,\mu)$, the function $(t,x)\rightarrow P_tf (x)$ is smooth on $(0,+\infty)\times \mathbb{M}$.
\end{proposition}

\begin{proof}
Let $f \in L^2(\M,\mu)$. Let  $t >0$. As above,  from the spectral theorem, $ P_t f  \in \cap_{k \ge 1} \mathcal{D}(L^k) \subset C^\infty(\M)$. Hence $P_t f $ is a smooth function.

Next, we prove joint continuity in the variables $(t,x)\in (0,+\infty)\times \mathbb{R}^n$. It is enough to prove that if $t_0 >0$ and if $K$ is a compact set on $\mathbb{M}$,
\[
\sup_{x \in K} | P_{t} f(x) - P_{t_0} f(x) | \rightarrow_{t \to t_0} 0.
\]
From Proposition \ref{regula},  there exists a positive constant $C$ such that 
\[
\sup_{x \in K} | P_{t} f(x) - P_{t_0} f(x) |  \le C  \sum_{k=0}^{\kappa} \|L^k P_t f-L^k P_{t_0} f \|^2_{ L^2(\M,\mu)  }.
\]
Now, again from the spectral theorem, it is checked that
\[
\lim_{t \to t_0} \sum_{k=0}^{\kappa} \|L^k P_t f-L^k P_{t_0} f \|^2_{ L^2(\M,\mu)  }=0.
\]
This gives the expected joint continuity in $(t,x)$. The joint smoothness in $(t,x)$ is a consequence of the second part of Proposition \ref{regula} and the details are let to the reader.
\end{proof}

We now prove the following fundamental theorem about the existence of a heat kernel for the semigroup generated by subelliptic diffusion operators.

\begin{theorem}
There is a smooth function $p(x,y,t)$, $t \in (0,+\infty), x,y \in \mathbb{M}$, such that for every $f \in L^2(\M,\mu)$ and $x \in \mathbb{M}$ ,
\[
P_t f (x)=\int_{\mathbb{M}} p(x,y,t) f(y) d\mu (y).
\]
The function $p(x,y,t)$ is called the heat kernel associated to $(P_t)_{t \ge 0}$. It satisfies furthermore:
\begin{itemize}
\item (Symmetry) $p(x,y,t)=p(y,x,t)$;
\item (Chapman-Kolmogorov relation) $p(x,y,t+s)=\int_{\mathbb{M}} p(x,z,t)p(z,y,s)d\mu(z)$. 
\end{itemize}
\end{theorem}

\begin{proof}
Let $x\in \mathbb{M}$ and $t>0$. From the previous proposition, the linear form $f \rightarrow P_t f (x)$ is continuous on $ L^2(\M,\mu)$, therefore from the Riesz representation theorem, there is a function $p(x,\cdot,t)\in  L^2(\M,\mu)$, such that for $f \in  L^2(\M,\mu)$,
\[
P_t f (x)=\int_{\mathbb{M}} p(x,y,t) f(y) d\mu (y).
\]
From the fact that $P_t$ is self-adjoint on  $L^2(\M,\mu)$:
\[
\int_{\mathbb{M}} (P_t f) g d\mu=\int_{\mathbb{M}} f(P_t g)  d\mu,
\]
we easily deduce the symmetry property:
\[
p(x,y,t)=p(y,x,t).
\]
And the Chapman-Kolmogorov relation $p(t+s,x,y)=\int_{\mathbb{M}} p(x,z,t)p(z,y,s)d\mu(z)$ stems from the semigroup property $P_{t+s} =P_t P_s$. Finally, from the previous proposition the map $(t,x) \rightarrow p(x,\cdot,t) \in L^2(\M,\mu)$ is smooth on $ \mathbb{M} \times (0,+\infty)$ for the weak topology on $L^2(\M,\mu)$. This implies that it is also smooth on $(0,+\infty) \times \mathbb{R}^n$ for the norm topology. Since, from the Chapman-Kolmogorov relation
\[
p(x,y,t)=\langle p(x,\cdot,t/2), p(y, \cdot, t/2) \rangle_{L_{\mu}^2 (\mathbb{M} )},
\]
we conclude that $(x,y,t)\rightarrow p(x,y,t)$ is smooth on $\mathbb{M} \times \mathbb{M} \times (0,+\infty)$.
\end{proof}

The semigroup $(P_t)_{t \ge 0}$  actually solves a parabolic Cauchy problem.
 
\begin{lemma}
Let $f \in L^p(\M,\mu)$, $1 \le p \le \infty$, and let
\[
u (t,x)= P_t f (x), \quad t \ge 0, x\in \mathbb{M}.
\]
Then $u$ is smooth on $(0,+\infty)\times \mathbb{M}$ and is a strong solution of the Cauchy problem
\[
\frac{\partial u}{\partial t}= L u,\quad u (0,x)=f(x).
\]
\end{lemma}

\begin{proof}
For $\phi \in C^\infty_0 ((0,+\infty) \times \mathbb{M})$, we have

\begin{align*}
\int_{\M \times \mathbb{R}} \left( \left( -\frac{\partial}{\partial t} -L \right) \phi (t,x) \right) u(t,x) d\mu(x) dt & =\int_{\mathbb{R}} \int_{\mathbb{M}} \left( \left( -\frac{\partial}{\partial t} -L \right) \phi (t,x) \right)  P_t f (x) dx dt  \\
 &= \int_{\mathbb{R}} \int_{\mathbb{M}}   P_t \left( \left( -\frac{\partial}{\partial t} -L \right) \phi (t,x) \right)  f (x) dx dt \\
 &= \int_{\mathbb{R}} \int_{\mathbb{M}}   -\frac{\partial}{\partial t} \left(  P_t \phi (t,x) f(x) \right) dx dt \\
 &=0.
\end{align*}
Therefore $u$ is a weak solution of the equation $\frac{\partial u}{\partial t}= L u$. Since $u$ is smooth it is also a strong solution.
\end{proof}

We now address the uniqueness of solutions.

\begin{lemma}\label{uniqueness heat equation}
Let $v(x,t)$ be a non negative function such that
\[
\frac{\partial v}{\partial t} \le L v,\quad v(x,0)=0,
\]
and such that for every $t >0$,
\[
\| v ( \cdot,t) \|_{L^p(\M,\mu)} <+\infty,
\]
where $1 <p <+\infty$. Then $v(x,t)=0$.
\end{lemma}

\begin{proof}
Let $h \in C_0^\infty(\M)$. Since $u$ is a subsolution with the zero initial condition, for any $\tau\in (0,T)$,
\begin{align*}
 & \int_0^\tau \int_\M  h^2(x) v^{p-1}(x,t) L v(x,t) d\mu(x) dt  \\
  \geq & \int_0^\tau \int_\M h^2(x) v^{p-1} \frac{\partial v}{\partial t} d\mu(x) dt  \\
  = & \frac{1}{p} \int_0^\tau \frac{\partial }{\partial t} \left( \int_\M h^2(x) v^{p} d\mu(x)\right) dt \\
   = & \frac{1}{p}\int_\M h^2(x) v^{p}(x,\tau) d\mu(x).
\end{align*}
On the other hand, integrating by parts yields
\begin{align*}
  &  \int_0^\tau \int_\M  h^2(x) v^{p-1}(x,t) L v(x,t) d\mu(x) dt \\
 =&  - \int_0^\tau \int_\M 2h v^{p-1} \Gamma(h,v)  d\mu dt - \int_0^\tau \int_\M  h^2 (p-1) v^{p-2} \Gamma(v)  d\mu dt .
\end{align*}
Observing that
\begin{align*}
0\leq & \left(\sqrt{\frac{2}{p-1}\Gamma(h)}v - \sqrt{\frac{p-1}{2}\Gamma(v)}h \right)^2  \leq \frac{2}{p-1}\Gamma(h)v^2 + 2 \Gamma(h,v) h v +\frac{p-1}{2}\Gamma(v)h^2 ,
\end{align*}
we obtain the following estimate.
\begin{align*}
 &  \int_0^\tau \int_\M  h^2(x) v^{p-1}(x,t) L v(x,t) d\mu(x) dt \\
  \leq&  \int_0^\tau \int_\M  \frac{2}{p-1} \Gamma(h) v^p  d\mu dt - \int_0^\tau \int_\M  \frac{p-1}{2}h^2 v^{p-2} \Gamma(v)  d\mu dt \\
 =& \int_0^\tau \int_\M  \frac{2}{p-1} \Gamma(h) v^p  d\mu dt - \frac{2(p-1)}{p^2} \int_0^\tau \int_\M  h^2 \Gamma(v^{p/2})  d\mu dt .
\end{align*}
Combining with the previous conclusion we obtain ,
\[
 \int_\M h^2(x) v^{p}(x,\tau) d\mu(x) + \frac{2(p-1)}{p} \int_0^\tau \int_\M  h^2 \Gamma(v^{p/2})  d\mu dt
   \leq \frac{2 p}{(p-1) } \| \Gamma(h) \|_\infty^2 \int_0^\tau \int_\M   v^p  d\mu dt.
\]
By using (H.1) and the previous inequality  with an increasing sequence $h_n\in C_0^\infty(\M)$, $0 \le h_n \le 1$,  such that $h_n\nearrow 1$ on $\mathbb{M}$, and $||\Gamma (h_n,h_n)||_{\infty} \to 0$, as $n\to \infty$, and letting $n \to +\infty$, we obtain $\int_\M  v^{p}(x,\tau) d\mu(x)=0$ thus $v=0$.
\end{proof}

As a consequence of this result, any solution  in $L^p(\M,\mu)$, $1<p<+\infty$ of the heat equation $\frac{\partial u}{\partial t}= L u$ is uniquely determined by its initial condition, and is therefore of the form $u(t,x)=P_tf(x)$. We stress that without further conditions, this result may fail when $p=1$ or $p=+\infty$.

\section{The heat semigroup on a complete Riemannian manifold and its geometric applications}

In this Section we shall consider a smooth and complete Riemannian manifold $(\mathbb{M},g)$ with dimension $n$. The Riemannian measure will be denoted by $\mu$ and we will often use the notation $\langle \cdot, \cdot \rangle$ for $g(\cdot,\cdot)$. A canonical elliptic diffusion operator on $\M$ is the Laplace-Beltrami operator. In this Section, we will show some applications to geometry of the Laplace-Beltrami operator and the semigroup it generates. We assume some basic knowledge about Riemannian geometry (see for instance \cite{Chavel,Petersen}). Some of the covered topics are inspired by M. Ledoux \cite{Led}.

\subsection{The Laplace-Beltrami operator}

The Laplace-Beltrami of $\mathbb{M}$ will be denoted by $L$.  We recall that it is the generator of the pre-Dirichlet form
\[
\mathcal{E}(f_1,f_2)=\int_\M g(\nabla f_1, \nabla f_2) d\mu, \quad f_1,f_2 \in C_0^\infty(\M),
\]
where $\nabla$ is the Riemannian gradient on $\M$.
Since $\mathbb{M}$ is assumed to be complete, as we have seen in the previous section, the operator $L$ is essentially self-adjoint on the space $C_0^\infty(\M)$. More precisely, 
there exists an increasing
sequence $h_n\in C_0^\infty(\M)$  such that $h_n\nearrow 1$ on
$\mathbb{M}$, and $||\Gamma(h_n,h_n)||_{\infty} \to 0$, as $n\to \infty$. Observe that for the Laplace-Beltrami operator $L$, the operator $\Gamma$ is simply given by $\Gamma(f_1,f_2)=\langle \nabla f_1, \nabla f_2 \rangle$.

The Friedrichs extension of $L$, which is therefore the unique self-adjoint extension of $L$ in $L^2(\M,\mu)$ will still be denoted by $L$ and the domain of this extension is denoted by $\mathcal{D}(L)$.

Using the results of the previous section, we then have:

\begin{itemize}
\item By using the spectral theorem for $L$ in the Hilbert space $L^2(\M,\mu)$, we may construct a strongly continuous contraction semigroup $(P_t)_{t \ge 0}$ in $L^2(\M,\mu)$ whose infinitesimal generator is $L$;
\item By using the ellipticity of $L$, we may prove that $(P_t)_{t \ge 0}$ admits a heat kernel, that is: There is a smooth function $p(t,x,y)$, $t \in (0,+\infty), x,y \in \mathbb{M}$, such that for every $f \in L^2(\M,\mu)$ and $x \in \mathbb{M}$ ,
\[
P_t f (x)=\int_{\mathbb{M}} p(t,x,y) f(y) d\mu (y).
\]
 Moreover, the heat kernel satisfies the two following conditions:
\begin{itemize}
\item (Symmetry) $p(t,x,y)=p(t,y,x)$;
\item (Chapman-Kolmogorov relation) $p(t+s,x,y)=\int_{\mathbb{M}} p(t,x,z)p(s,z,y)d\mu(z)$. 
\end{itemize}
\item The semigroup  $(P_t)_{t \ge 0}$ is a sub-Markov semigroup: If $ 0\le f \le 1$ is a function in $L^2(\M,\mu)$, then $0 \le P_t f \le 1$.
\item By using the Riesz-Thorin interpolation theorem, $(P_t)_{t \ge 0}$ defines a contraction semigroup on $L^p(\M,\mu)$, $1 \le p \le \infty$.
\end{itemize} 

\subsection{The heat semigroup on a compact Riemannian manifold}

In this Section, we study some spectral properties of the Laplace-Beltrami operator and of the heat semigroup on a compact Riemannian manifold. So, let $(\mathbb{M},g)$ be a compact Riemannian manifold. As usual, we denote by $(P_t)_{t \ge 0}$ the heat semigroup and by $p(t,x,y)$ the corresponding heat kernel.
As a preliminary result, we have the following Liouville's type theorem\index{Liouville theorem}.

\begin{lemma}
Let $f \in \mathcal{D}(L)$ such that $Lf=0$, then $f$ is a constant function.
\end{lemma}

\begin{proof}
From the ellipticity of $L$, we first deduce that $f$ is smooth. Then, since $\mathbb{M}$ is compact, the following equality holds
\[
-\int_{\mathbb{M}} f Lf d\mu=\int_{\mathbb{M}} \Gamma(f,f) d\mu.
\]
Therefore $\Gamma(f,f)=0$, which implies that $f$ is a constant function.
\end{proof}

In the compact case, the heat semigroup satisfies the so-called stochastic completeness (or Markov) property.

\begin{proposition}
For $t \ge 0$,
\[
P_t 1=1.
\]
\end{proposition}

\begin{proof}
Since the constant function $1$ is in $L^2(\M,\mu)$, by compactness of $\mathbb{M}$, we may  apply Proposition \ref{uniqueness heat equation}.
\end{proof}

It turns out that the compactness of $\mathbb{M}$ implies the compactness of the semigroup.

\begin{proposition}
For $t>0$ the operator $P_t$ is a compact operator on the Hilbert space $L^2(\M,\mu)$. It is moreover trace class and 
\[
\mathbf{Tr} (P_t)=\int_\mathbb{M} p(t,x,x)\mu(dx).
\]
\end{proposition}

\begin{proof}
From the existence of the heat kernel we have
\[
P_t f (x)=\int_{\mathbb{M}} p(t,x,y) f(y) d\mu (y).
\]
But from the compactness of $\mathbb{M}$, we have
\[
\int_{\mathbb{M}}\int_{\mathbb{M}} p(t,x,y)^2 d \mu(x) d\mu(y) <+\infty.
\]
Therefore,  the operator 
\[
P_t: L^2(\mathbb{M},\mu) \rightarrow L^2(\mathbb{M},\mu)
\]
is a Hilbert-Schmidt operator. It is thus in particular a compact operator.

Since $P_t=P_{t/2} P_{t/2}$, $P_t$ is a product of two Hilbert-Schmidt operators. It is therefore a class trace operator and,
\[
\mathbf{Tr} ( P_t)=\int_{\mathbb{M}}\int_{\mathbb{M}} p(t/2,x,y)p(t/2,y,x) d \mu(x) d\mu(y).
\]
 We conclude then by applying the Chapman-Kolmogorov relation.
\end{proof}

In this compact framework, we have the following theorem

\begin{theorem}\label{spectral compact manifold}
There exists a complete orthonormal basis $(\phi_n)_{n \in \mathbb{N}}$ of $L^2(\mathbb{M},\mu)$, consisting of eigenfunctions of $-L$, with $\phi_n$ having an eigenvalue $\lambda_n$ with finite multiplicity satisfying
\[
0=\lambda_0 < \lambda_1\le \lambda_2  \le \cdots \nearrow +\infty.
\]
Moreover, for $t>0$, $x,y \in \mathbb{M}$,
\[
p(t,x,y)=\sum_{n=0}^{+\infty} e^{-\lambda_n t} \phi_n (x) \phi_n (y),
\]
with convergence absolute and uniform for each $t>0$.
\end{theorem}

\begin{proof}
Let $t>0$. From the Hilbert-Schmidt theorem for the  non negative self adjoint compact operator $P_t$, there exists a complete orthonormal basis $(\phi_n(t) )_{n \in \mathbb{N}}$ of $L^2(\mathbb{M},\mu)$ and a non increasing sequence $\alpha_n(t) \ge 0$, $\alpha_n(t) \searrow 0$ such that
\[
P_t \phi_n(t)=\alpha_n(t) \phi_n (t).
\]
The semigroup property $P_{t+s}=P_t P_s$ implies first that for $k \in \mathbb{N}$, $k \ge 1$,
\[
\phi_n(k)=\phi_n(1), \alpha_n(k)=\alpha_n (1)^k.
\]
The same result is then seen to hold for $k \in \mathbb{Q}$, $k >0$ and finally for $k \in \mathbb{R}$, due to the strong continuity of the semigroup. Since the map $t \to \| P_t \|_2$ is decreasing, we deduce that $\alpha_n (1) \le 1$. Thus, there is a $\lambda_n \ge 0$ such that
\[
\alpha_n(1)=e^{-\lambda_n}.
\]
As a conclusion, there exists a complete orthonormal basis $(\phi_n)_{n \in \mathbb{N}}$ of $L^2(\M,\mu)$,  and a sequence $\lambda_n$  satisfying
\[
0 \le \lambda_0 \le \lambda_1\le \lambda_2  \le \cdots \nearrow +\infty,
\]
such that 
\[
P_t \phi_n =e^{-\lambda_n t}  \phi_n.
\]
Since $P_t 1=1$, we actually have $\lambda_0=0$. Also,  if $f \in L^2(\M,\mu)$ is such that $P_t f=f$, it is straightforward that $f \in \mathcal{D}(L)$ and that $Lf=0$, so that thanks to Liouville theorem, $f$ is a constant function. Therefore $\lambda_1>0$.

Since $P_t \phi_n =e^{-\lambda_n t}  \phi_n$, by differentiating as $t \to 0$ in $L^2(\M,\mu)$, we obtain furthermore that $\phi_n \in \mathcal{D}(L)$ and that $L\phi_n=-\lambda_n \phi_n$.

The family $(x,y)\to \phi_n(x) \phi_m(y)$ forms an orthonormal basis of $L^2 (\mathbb{M}\times \mathbb{M},\mu \otimes \mu)$. We therefore have a decomposition in $L^2 (\mathbb{M}\times \mathbb{M},\mu \otimes \mu)$,
\[
p(t,x,y)=\sum_{m,n \in \mathbb{M}} c_{mn} \phi_m(x) \phi_n(y).
\]
Since $p(t,\cdot,\cdot)$ is the kernel of $P_t$, it is then straightforward that for $m \neq n$, $c_{mn}=0$ and that $c_{nn}=e^{-\lambda_n t}$. Therefore in $L^2(\M,\mu)$,
\[
p(t,x,y)=\sum_{n=0}^{+\infty} e^{-\lambda_n t} \phi_n (x) \phi_n (y).
\]
The continuity of $p$, together with the positivity of $P_t$ imply, via Mercer's theorem that actually, the above series is absolutely and uniformly convergent for $t>0$.
\end{proof}

As we stressed it in the statement of the theorem, in the decomposition
\[
p(t,x,y)=\sum_{n=0}^{+\infty} e^{-\lambda_n t} \phi_n (x) \phi_n (y),
\]
the eigenvalue $\lambda_n$ is repeated according to its multiplicity. It is often useful to rewrite this decomposition under the form
\[
p(t,x,y)=\sum_{n=0}^{+\infty} e^{-\alpha_n t} \sum_{k=1}^{d_n} \phi^n_k (x) \phi^n_k (y),
\]
where the eigenvalue $\alpha_n$ is not repeated, that is 
\[
0=\alpha_0 < \alpha_1 <\alpha_2 < \cdots
\]
In this decomposition, $d_n$ is  the dimension of the eigenspace $\mathcal{V}_n$ corresponding to the eigenvalue $\alpha_n$ and $(\phi^n_k)_{1 \le k \le d_n } $ is an orthonormal basis of $\mathcal{V}_n$.
If we denote,
\[
\mathcal{K}_n(x,y)= \sum_{k=1}^{d_n} \phi^n_k (x) \phi^n_k (y),
\]
then $\mathcal{K}_n$ is called the reproducing kernel\index{Reproducing kernel} of the eigenspace $\mathcal{V}_n$. It satisfies the following properties whose proofs are let to the reader:

\begin{proposition}

\

\begin{itemize}
\item $\mathcal{K}_n$ does not depend on the choice of the basis  $(\phi^n_k)_{1 \le k \le d_n } $;
\item If $f \in \mathcal{V}_n$, then $\int_{\mathbb{M}} \mathcal{K}_n (x,y) f(y) d\mu(y) =f(x)$.
\end{itemize}
\end{proposition}

From the very definition of the reproducing kernels, we have
\begin{align}\label{heat_semigroup_compact}
p(t,x,y)=\sum_{n=0}^{+\infty} e^{-\alpha_n t} \mathcal{K}_n(x,y).
\end{align}

The compactness of $\mathbb{M}$ also implies the convergence to equilibrium for the semigroup.

\begin{proposition}
Let $f \in L^2(\M,\mu)$, then uniformly on $\mathbb{M}$, when $t \to +\infty$,
\[
P_t f \rightarrow \frac{1}{\mu(\mathbb{M})} \int_\mathbb{M} f d\mu.
\]
\end{proposition}

\begin{proof}
It is obvious from the previous proposition and from spectral theory that in $L^2(\M,\mu)$, $P_t f $ converges to a constant function that we denote $P_\infty f $. The convergence is also uniform, because for $s,t,T>0$,
\begin{align*}
\| P_{t+T} f -P_{s+T} f \|_\infty &= \sup_{x \in \mathbb{M}} \left| P_{T} ( P_{t} f -P_{s} f) (x) \right| \\
 & = \sup_{x \in \mathbb{M}} \left|\int_{\mathbb{M}} p(T,x,y) ( P_{t} f -P_{s} f) (y) d\mu(y) \right| \\
 & \le  \left( \sup_{x \in \mathbb{M}} \sqrt{ \int_{\mathbb{M}} p(T,x,y)^2 d\mu(y)}\right) \| P_{t} f -P_{s} f \|_2. 
\end{align*}
Moreover, for every $t \ge 0$, $\int_{\mathbb{M}} P_t f d\mu =\int_{\mathbb{M}} f d\mu$. Therefore
\[
\int_{\mathbb{M}} P_\infty f d\mu =\int_{\mathbb{M}} f d\mu.
\]
Since $ P_\infty f$ is constant, we finally deduce the expected result:
\[
 P_\infty f=\frac{1}{\mu(\mathbb{M})} \int_\mathbb{M} f d\mu.
\]
\end{proof}

\subsection{Bochner's identity}

The Bochner's identity is a fundamental identity that connects the Ricci curvature of a Riemannian manifold $(\M,g)$ to the Laplace-Beltrami-operator (see \cite{Chavel,Petersen}).

\begin{theorem}
If $f \in C^\infty(\M)$, then
\[
\frac{1}{2} L (\| \nabla f \|^2) - \langle \nabla f ,\nabla Lf\rangle=\| \mathbf{Hess} f \|^2_{HS}+ \mathbf{Ric} (\nabla f, \nabla f),
\]
where $\mathbf{Ric}$ denotes the Ricci curvature tensor of $\M$ and $\| \mathbf{Hess} f \|_{HS}$ the Hilbert-Schmidt norm of the Hessian of $f$.
\end{theorem}

\subsection{The curvature dimension inequality}
We now introduce the Bakry's $\Gamma_2$ operator. For $f,g \in C^\infty(\M)$, it is defined as
\[
\Gamma_2(f,f)=\frac{1}{2} L (\| \nabla f \|^2) - \langle \nabla f ,\nabla Lf\rangle.
\]
In the previous Section, we have seen that on  our Riemannian manifold $\mathbb{M}$,
\[
\Gamma_2 (f,f)=\| \mathbf{Hess} f \|^2_{HS}+ \mathbf{Ric} (\nabla f, \nabla f),
\] 
Therefore it should come as no surprise that  a lower bound on $\mathbf{Ric}$ translates into a lower bound on $\Gamma_2$.

\begin{theorem}
Let $\mathbb{M}$ be a Riemannian manifold. We have, in the sense of bilinear forms,  $\mathbf{Ric} \ge \rho$ if and only if for every $f \in C^\infty(\mathbb{M})$,
\[
\Gamma_2(f,f) \ge \frac{1}{n} (Lf)^2 + \rho \Gamma(f,f).
\]
\end{theorem}

\begin{proof}
Let us assume that  $\mathbf{Ric} \ge \rho$. In that case, from Bochner's formula we deduce that
\[
\Gamma_2 (f,f) \ge \| \mathbf{Hess} f \|^2_{HS}+  \rho \Gamma(f,f).
\]
From Cauchy-Schwartz inequality, we have the bound
\[
 \| \mathbf{Hess} f \|^2_{HS} \ge \frac{1}{n} \mathbf{Tr} \left(\mathbf{Hess} f \right)^2.
\]
Since $\mathbf{Tr} \left(\mathbf{Hess} f \right)=Lf$, we conclude that
\[
\Gamma_2(f,f) \ge \frac{1}{n} (Lf)^2 + \rho \Gamma(f,f).
\]
Conversely, let us now assume that for every $f \in C^\infty(\mathbb{M})$,
\[
\Gamma_2(f,f) \ge \frac{1}{n} (Lf)^2 + \rho \Gamma(f,f).
\]
Let $x \in \mathbb{M}$ and $v \in \mathbf{T}_x \mathbb{M}$. It is possible to find a function $f \in C^\infty(\mathbb{M})$ such that, at $x$,  $\mathbf{Hess} f =0$ and $\nabla f=v$. We have then, by using Bochner's identity at $x$,
\[
\mathbf{Ric} (v,v) \ge \rho \| v \|^2.
\]
\end{proof}

\begin{remark}
The inequality
\[
\Gamma_2(f,f) \ge \frac{1}{n} (Lf)^2 + \rho \Gamma(f,f).
\]
is called the curvature dimension inequality. It is satisfied for more general operators than Laplace-Beltrami operators on manifolds with Ricci curvature lower bounds. Many of the results covered in those notes extend to those operators (see \cite{BGL}).
\end{remark}

We finally mention another consequence of Bochner's identity which shall be later used.

\begin{lemma}
Let $\mathbb{M}$ be a Riemannian manifold  such that  $\mathbf{Ric} \ge \rho$.  For every $f \in C^\infty(\mathbb{M})$, 
\[
\Gamma(\Gamma(f)) \le 4 \Gamma (f) \left( \Gamma_2(f)-\rho\Gamma(f)\right).
\]
\end{lemma}

\begin{proof}
It follows from the fact that 
\[
\Gamma_2 (f,f) \ge \| \mathbf{Hess} f \|^2_{HS}+  \rho \Gamma(f,f).
\]
and Cauchy-Schwarz inequality implies $\Gamma(\Gamma(f)) \le 4  \| \mathbf{Hess} f \|^2_{HS} \Gamma(f)$. Details are let to the reader.
\end{proof}

\subsection{Stochastic completeness}

In this section, we will prove a first interesting consequence of the Bochner's identity: We will prove that if, on a complete Riemannian manifold $\mathbb{M}$, the Ricci curvature  is bounded from below, then the heat semigroup is stochastically complete, that is $P_t 1=1$.  This result is due to S.T. Yau, and we will see this property is also  equivalent  to the uniqueness in $L^\infty$ for solutions of the heat equation. The proof we give is due to D. Bakry.

Let $\mathbb{M}$ be a complete Riemannian manifold and denote by $L$ its Laplace-Beltrami operator. As usual, we denote by $P_t$ the heat semigroup generated by $L$. Throughout the section, we will assume that the Ricci curvature of $\mathbb{M}$ is bounded from below by $\rho \in \mathbb{R}$. As seen in the previous Section, this is equivalent to the fact that for every $f \in C^\infty(\mathbb{M})$,
\[
\Gamma_2(f,f) \ge \frac{1}{n} (Lf)^2 + \rho \Gamma(f,f).
\]
We start with a technical lemma:
\begin{lemma}
If $f \in L^2(\M,\mu)$, then for every $t>0$,   the functions $\Gamma (P_t f), L\Gamma(P_t f), \Gamma(P_tf, LP_t f)$ and $\Gamma_2 (P_t f)$ are in  $ L^1(\mathbb{M},\mu)$.
\end{lemma}

\begin{proof}
It is straightforward to see from the spectral theorem that $\Gamma (P_t f) \in L^1(\M,\mu)$. Similarly, $| \Gamma(P_tf, LP_t f) | \le \sqrt{\Gamma(P_tf) \Gamma(LP_tf) }   \in L^1(\M,\mu)$. 
Since, 
\[
\Gamma_2 (P_t f) =\frac{1}{2} \left( L\Gamma(P_t f) -2  \Gamma(P_tf, LP_t f)\right) 
\]
we are let with the problem of proving that $\Gamma_2 (P_t f) \in  L^1(\M,\mu)$. If $g \in C_0^\infty(\mathbb{M})$, then an integration by parts easily yields
\[
\int_\mathbb{M} \Gamma_2 (g) d\mu=\int_\mathbb{M} (Lg)^2 d\mu.
\]
As a consequence, 
\[
\int_\mathbb{M} \Gamma_2 (g)-\rho \Gamma(g)  d\mu=\int_\mathbb{M} (Lg)^2 +\rho g Lg  d\mu,
\]
and we obtain
\[
\int_\mathbb{M} | \Gamma_2 (g)-\rho \Gamma(g) | d\mu \le\left(1 +\frac{1}{2}| \rho | \right) \int_\mathbb{M} (Lg)^2 d\mu +\frac{1}{2}| \rho | \int_\mathbb{M} g^2    d\mu.
\]
Using a density argument, it is then easily proved that for $g \in \mathcal{D}(L)\cap C^\infty(\mathbb{M})$ we have
\[
\int_\mathbb{M} | \Gamma_2 (g)-\rho \Gamma(g) | d\mu \le\left(1 +\frac{1}{2}| \rho | \right) \int_\mathbb{M} (Lg)^2 d\mu +\frac{1}{2}| \rho | \int_\mathbb{M} g^2    d\mu.
\]
In particular, we deduce that if $g \in \mathcal{D}(L)\cap C^\infty(\mathbb{M})$, then $\Gamma_2(g) \in L^1(\M,\mu)$.
\end{proof}

We will also need the following fundamental parabolic comparison theorem that shall be extensively used throughout these lecture notes.

\begin{proposition}
Let $T>0$. Let $u,v: \mathbb{M}\times [0,T] \to \mathbb{R}$ be  smooth functions such that:
\begin{itemize}
\item[(i)]  For every $t \in [0,T]$, $u(\cdot,t) \in L^2(\mathbb{M})$ and $\int_0^T \| u(\cdot,t)\|_2 dt <\infty$;
\item[(ii)]  $\int_0^T \| \sqrt{\Gamma(u) (\cdot,t)} \|_p dt <\infty$ for some $1 \le p \le \infty$;
\item[(iii)]  For every $t \in [0,T]$, $v(\cdot,t) \in L^q(\mathbb{M})$ and $\int_ 0^T \| v(\cdot,t ) \|_q dt <\infty$ for some $1 \le q \le \infty$. 
\end{itemize}
 If the  inequality 
\[
Lu+\frac{\partial u}{\partial t} \ge v,
\]
holds on $\mathbb{M}\times [0,T]$, then we have
\[
P_T u(\cdot,T)(x) \ge u(x,0) +\int_0^T P_s v(\cdot,s)(x) ds.
\]
\end{proposition}

\begin{proof}
Let $f,g \in C_0^\infty (\mathbb{M})$, $f,g \ge 0$. We claim that we must have
\begin{align}\label{cp1}
& \int_\mathbb{M} g P_T(fu(\cdot,T)) d\mu - \int_\mathbb{M} g f u(x,0) d\mu
\ge   -  \|\sqrt{\Gamma(f)}\|_\infty \int_0^T  \int_\mathbb{M}   (P_t g) \sqrt{\Gamma(u)}d\mu dt
\\
& -  \| \sqrt{\Gamma(f)} \|_\infty \int_0^T \| \sqrt{\Gamma(P_t g) }\|_2 \| u(\cdot,t) \|_2   dt  +  \int_\mathbb{M} g \int_0^T  P_t( f v(\cdot,t)) d\mu dt,
\notag
\end{align}
where for every $1\le p \le \infty$ and a measurable $F$, we have let $||F||_p = ||F||_{L^p(\mathbb{M})}$.
To establish \eqref{cp1} we consider the function
\[
\phi(t)=\int_\mathbb{M} g P_t (fu(\cdot,t)) d\mu.
\]
Differentiating $\phi$ we find
\begin{align*}
\phi'(t)& =\int_\mathbb{M} g P_t \left(L( fu) + f\frac{\partial u}{\partial t} \right) d\mu \\
 &= \int_\mathbb{M} g P_t \left((L f) u+2 \Gamma (f,u) +f Lu + f\frac{\partial u}{\partial t} \right) d\mu \\
 &\ge \int_\mathbb{M} g P_t \left((L f) u+2 \Gamma (f,u)  \right) d\mu+\int_\mathbb{M} g P_t( f v) d\mu.
\end{align*}
Since
\begin{align*}
\int_\mathbb{M} g P_t \left((L f) u\right) d\mu &= \int_\mathbb{M}  (P_t g) (L f) u d\mu \\
 &= -\int_\mathbb{M}  \Gamma( f, u(P_t g)) d\mu  \\
 &=-\left( \int_\mathbb{M}  P_t g \Gamma( f, u)+ u \Gamma(f,P_t g) d\mu\right),
\end{align*}
we obtain
\[
\phi'(t) \ge \int_\mathbb{M} P_t g \Gamma(f,u) d\mu - \int_\mathbb{M} u \Gamma(f,P_t g) d\mu + \int_\mathbb{M} g P_t(fv) d\mu.
\]
Now, we can bound
\[
\left| \int_\mathbb{M}  (P_t g) \Gamma( f, u) d\mu\right| 
 \le    \| \sqrt{\Gamma(f) } \|_\infty \int_\mathbb{M}  (P_t g)  \sqrt{\Gamma( u)} d\mu,
\]
and for a.e. $t\in [0,T]$ the integral in the right-hand side is finite in view of the assumption (ii) above.
We have thus obtained
\begin{align*}
\phi'(t)  \ge -  \| \sqrt{\Gamma(f) } \|_\infty  \int_\mathbb{M}  (P_t g)  \sqrt{\Gamma(u)} d\mu- \int_\mathbb{M} u \Gamma(f , P_t g) d\mu+ \int_\mathbb{M} g P_t( f v(\cdot,t)) d\mu.
\end{align*}
As a consequence, we find
\begin{align*}
 & \int_\mathbb{M} g P_T (fu(\cdot,T)) d\mu-\int_\mathbb{M} gfu(x,0)d\mu  \\
 \ge&  -  \| \sqrt{\Gamma(f) } \|_\infty   \int_0^T  \int_\mathbb{M}   (P_t g) \sqrt{\Gamma(u)} d\mu dt -  \int_0^T \int_\mathbb{M}  u \Gamma\left(f,P_t g\right) d\mu dt + \int_0^T  \int_\mathbb{M} g  P_t(f v(\cdot,t)) d\mu dt
 \\
 \ge &  -  \| \sqrt{\Gamma(f) } \|_\infty   \int_0^T \int_\mathbb{M} (P_t g) \sqrt{\Gamma(u)} d\mu dt - \int_0^T \| u(\cdot,t) \|_2 \| \Gamma(f, P_t g) \|_2 dt  + \int_\mathbb{M} g \int_0^T  P_t( f v(\cdot,t)) dt d\mu \\
 \ge &    -  \| \sqrt{\Gamma(f) } \|_\infty   \int_0^T  \int_\mathbb{M} (P_t g) \sqrt{\Gamma(u)} d\mu dt -  \| \sqrt{\Gamma(f)} \|_\infty \int_0^T \| u(\cdot,t) \|_2  \| \sqrt{\Gamma(P_t g) }\|_2 dt \\ +  & \int_\mathbb{M} g \int_0^T  P_t(f v(\cdot,t)) dt d\mu,
\end{align*}
which proves \eqref{cp1}.
Let now $h_k\in C^\infty_0(\mathbb{M})$ be a sequence such that $0 \le h_k \le 1$, $\| \Gamma(h_k) \|_\infty \to 0$ and $h_k$ increases to 1.                
Using $h_k$ in place of $f$ in \eqref{cp1}, and letting $k \to \infty$, gives
\begin{align*}
  \int_\mathbb{M} g P_T (u(\cdot,T)) d\mu - \int_\mathbb{M} g u(x,0)d\mu  
 \ge   \int_\mathbb{M} g \int_0^T  P_t(v(\cdot,t)) dt d\mu.
\end{align*}
We observe that the assumption on $v$ and Minkowski's integral inequality guarantee that the function $x\to \int_0^T  P_t(v(\cdot,t))(x) dt$ belongs to $L^q(\mathbb{M})$. We have in fact
\begin{align*}
\left(\int_\mathbb{M} \left|\int_0^T  P_t(v(\cdot,t)) dt\right|^q d\mu\right)^{\frac 1q} & \le \int_0^T \left| \int_\mathbb{M} \left|P_t(v(\cdot,t))\right|^q d\mu\right|^{\frac 1q} dt \le \int_0^T \left| \int_\mathbb{M} \left|v(\cdot,t)\right|^q d\mu\right|^{\frac 1q} dt
\\
& \le T^{\frac{1}{q'}}  \left(\int_0^T \int_\mathbb{M} \left|v(\cdot,t)\right|^q d\mu dt \right)^{\frac 1q} <\infty.
\end{align*}
Since this must hold for every non negative $g \in C_0^\infty (\mathbb{M})$, we conclude that
\[
P_T(u(\cdot,T))(x) \ge u(x,0) +\int_0^T P_s (v(\cdot,s))(x) ds,
\]
which completes the proof.
\end{proof}

We are in position to prove the first gradient bound for the semigroup $P_t$.

\begin{proposition}
If $f$ is a smooth function in $\mathcal{D}(L)$, then for every $t \ge 0$ and $x \in \mathbb{M}$,
\[
\sqrt{\Gamma(P_t f)}(x) \le e^{-\rho t} P_t \sqrt{\Gamma(f)} (x).
\]
\end{proposition}

\begin{proof}
 We fix $T>0$ and consider the functional
\[
\Phi(x,t)=e^{-\rho t} \sqrt{ \Gamma(P_{T-t} f)}(x).
\]
We first assume that $(x,t)\to \Gamma(P_t f)(x)>0$ on $\mathbb{M} \times [0,T]$. From the previous lemma,  we have $\Phi(t) \in L^2(\mathbb{M})$. Moreover $\Gamma(\Phi)(t)=e^{-2\rho t}\frac{\Gamma(\Gamma(P_{T-t} f))}{4 \Gamma(P_{T-t} f)}$. So,  we have $\Gamma( \Phi)(t) \le e^{-2\rho t}( \Gamma_2(P_{T-t} f)-\rho\Gamma(P_tf))$.  Therefore, again from the previous proposition , we deduce that  $\Gamma( \Phi)(t) \in L^1(\mathbb{M})$. Next, we easily compute that
\[
\frac{\partial \Phi}{\partial t}+ L\Phi =e^{-\rho t} \left( \frac{\Gamma_2(P_{T-t} f)}{\sqrt{\Gamma(P_{T-t} f)}}-\frac{\Gamma(\Gamma(P_{T-t} f))}{4 \Gamma(P_{T-t} f)^{3/2} } -\rho \sqrt{\Gamma(P_{T-t} f)} \right).
\]
Thus,
\[
\frac{\partial \Phi}{\partial t}+ L\Phi \ge 0.
\]
We can then use the parabolic comparison theorem  to infer that 
\[
\sqrt{ \Gamma(P_{T} f)} \le e^{-\rho T} P_T \left(\sqrt{\Gamma (f)} \right).
\]
If $(x,t) \to \Gamma(P_t f)(x)$ vanishes on $\mathbb{M} \times [0,T]$, we consider the functional 
\[
\Phi(t)=e^{-\rho t} g_\varepsilon (\Gamma(P_{T-t} f) ),
\]
where, for $0 < \varepsilon <1$,
\begin{align*}
g_\varepsilon (y)=\sqrt{ y+\varepsilon^2}-\varepsilon.
\end{align*}
Since $\Phi(t) \in L^2(\mathbb{M})$, an argument similar to that above (details are let to the reader) shows that
\[
g_\varepsilon (\Gamma(P_{T} f) )\le e^{-\rho T} P_T \left( g_\varepsilon( \Gamma (f)) \right).
\]
Letting $\varepsilon \to 0$, we conclude that 
\[
\sqrt{ \Gamma(P_{T} f)} \le  e^{-\rho T} P_T \left(\sqrt{\Gamma (f)} \right).
\]
\end{proof}

We now prove the promised stochastic completeness result:

\begin{theorem}\label{T:sc}
For $t \ge 0$, one has $ P_t 1 =1$.
\end{theorem}
\begin{proof}
Let $f,g \in  C^\infty_0(\mathbb M)$, we have
\begin{align*}
\int_{\bM} (P_t f -f) g d\mu = \int_0^t \int_{\bM}\left(
\frac{\partial}{\partial s} P_s f \right) g d\mu ds= \int_0^t
\int_{\bM}\left(L P_s f \right) g d\mu ds=- \int_0^t \int_{\bM}
\Gamma ( P_s f , g) d\mu ds.
\end{align*}
By means of the previous Proposition and Cauchy-Schwarz inequality, we
find
\begin{equation}\label{P1}
\left| \int_{\bM} (P_t f -f) g d\mu \right| \le \left(\int_0^t
e^{-\rho s} ds\right) \sqrt{ \| \Gamma (f) \|_\infty  } \int_{\bM}\Gamma (g)^{\frac{1}{2}}d\mu.
\end{equation}

We now apply the previous inequality with $f = h_n$, and then let $n\to \infty$.
Since by Beppo Levi's monotone convergence theorem we have $P_t
h_n(x)\nearrow P_t 1(x)$ for every $x\in \bM$, we see that the
left-hand side converges to $\int_{\bM} (P_t 1 -1) g d\mu$.  We thus reach the conclusion
\[
\int_{\bM} (P_t 1 -1) g d\mu=0,\ \ \ g\in C^\infty_0(\bM).
\]
It follows that $P_t 1 =1$.

\end{proof}

\subsection{Convergence to equilibrium, Poincar\'e and log-Sobolev inequalities}

Let $\mathbb{M}$ be a complete $n$-dimensional Riemannian manifold and denote by $L$ its Laplace-Beltrami operator. As usual, we denote by $P_t$ the heat semigroup generated by $L$. Throughout the Section, we will assume that the Ricci curvature of $\mathbb{M}$ is bounded from below by $\rho >0$. We recall that this is equivalent to the fact that for every $f \in C^\infty(\mathbb{M})$,
\[
\Gamma_2(f,f) \ge \frac{1}{n} (Lf)^2 + \rho \Gamma(f,f).
\]
Readers knowing Riemannian geometry know that from Bonnet-Myers theorem, the manifold needs to be compact and we therefore expect the semigroup to converge to equilibrium. However, our goal will be to not use the Bonnet-Myers theorem, because eventually we shall provide a proof of this fact using semigroup theory. Thus the results in this Section will not use the compactness of $\mathbb{M}$.

\begin{lemma}
The Riemannian measure $\mu$ is finite, i.e. $\mu(\bM) <+\infty$ and for every $f \in L_\mu^2(\mathbb{M})$, the following convergence holds pointwise and in $L^2(\M,\mu)$,
\[
P_t f  \to_{t \to +\infty} \frac{1}{\mu(\mathbb{M})} \int_\mathbb{M} f d\mu.
\]
\end{lemma}
\begin{proof}
 Let $f,g \in C^\infty_0(\mathbb{M})$, we have
\begin{align*}
\int_{\mathbb{M}} (P_t f -f) g d\mu = \int_0^t \int_{\mathbb{M}}\left(
\frac{\partial}{\partial s} P_s f \right) g d\mu ds= \int_0^t
\int_{\mathbb{M}}\left(L P_s f \right) g d\mu ds=- \int_0^t \int_{\mathbb{M}}
\Gamma ( P_s f , g) d\mu ds.
\end{align*}
By means of  Cauchy-Schwarz inequality, we
find
\begin{equation}\label{P1}
\left| \int_{\mathbb{M}} (P_t f -f) g d\mu \right| \le \left(\int_0^t
e^{-2\rho s} ds\right) \sqrt{ \| \Gamma (f) \|_\infty } \int_{\mathbb{M}}\Gamma (g)^{\frac{1}{2}}d\mu.
\end{equation}
Now it is seen from spectral theorem that in $L^2(\mathbb{M})$ we have  a convergence $P_t f \to P_\infty f$, where $P_\infty f$ belongs to the domain of $L$. Moreover $LP_\infty f=0$. By ellipticity of $L$ we deduce that $P_\infty f$ is a smooth function. Since $LP_\infty f=0$, we have $\Gamma(P_\infty f)=0$ and therefore $P_\infty f $ is constant.

 Let us now assume that $\mu(\mathbb{M})=+\infty$. This implies in particular that $P_\infty f =0$  because no constant besides $0$ is in $L^2(\mathbb{M})$. Using then (\ref{P1}) and letting $t \to +\infty$, we infer
 \begin{equation*}
\left| \int_{\mathbb{M}} f g d\mu \right| \le \left(\int_0^{+\infty}
e^{-2\rho s} ds\right)  \sqrt{ \| \Gamma (f) \|_\infty } \int_{\mathbb{M}}\Gamma (g)^{\frac{1}{2}}d\mu.
\end{equation*}
Let us assume $g \ge 0$, $g \ne 0$ and take for $f$ the usual localizing  sequence $h_n$. Letting $n \to \infty$, we deduce
\[
\int_\mathbb{M} g d\mu \le 0,
\]
which is clearly absurd. As a consequence $\mu (\mathbb{M}) <+\infty$. 

The invariance of $\mu$ implies then 
\[
\int_\mathbb{M} P_\infty f d\mu =\int_\mathbb{M} f d\mu,
\]
and thus
\[
P_\infty f =\frac{1}{\mu(\mathbb{M})} \int_\mathbb{M} f d\mu.
\]
Finally, using the Cauchy-Schwarz inequality, we find that for $x \in \bM$, $f \in L^2(\mathbb{M})$, $s,t,\tau \ge 0$,
\begin{align*}
| P_{t+\tau} f (x)-P_{s+\tau} f (x) | & = | P_\tau (P_t f -P_s f) (x) | \\
 &=\left| \int_\mathbb{M} p(\tau, x, y) (P_t f -P_s f) (y) \mu(dy) \right| \\
 &\le \int_\mathbb{M} p(\tau, x, y)^2 \mu(dy) \| P_t f -P_s f\|^2_2 \\
 &\le p(2\tau,x,x) \| P_t f -P_s f\|^2_2.
\end{align*}
Thus, we also have 
\[
P_t f (x) \to_{t \to +\infty} \frac{1}{\mu(\mathbb{M})} \int_\bM f d\mu.
\]
\end{proof}

\begin{proposition}
The following Poincar\'e inequality is satisfied: For $f \in \mathcal{D}(L)$,
\[
 \frac{1}{\mu(\mathbb{M})}\int_{\mathbb{M}} f^2 d\mu \le \left( \frac{1}{\mu(\mathbb{M})} \int_{\mathbb{M}} f d\mu \right)^2 +\frac{n-1}{n\rho} \frac{1}{\mu(\mathbb{M})} \int_{\mathbb{M}} \Gamma(f,f) d\mu.
\]
\end{proposition}
\begin{proof}
Let $f \in C_0^\infty(\mathbb{M})$. We have by assumption
\[
\Gamma_2(f,f) \ge  \frac{1}{n} (Lf)^2 + \rho \Gamma (f,f).
\]
Therefore, by integrating the latter inequality we obtain
\[
\int_{\mathbb{M}} \Gamma_2(f,f)  d\mu \ge  \frac{1}{n} (Lf)^2+ \rho \int_{\mathbb{M}} \Gamma (f,f) d\mu.
\]
But we have
\[
\int_{\mathbb{M}} \Gamma_2(f,f)  d\mu=-\int_{\mathbb{M}} \Gamma(f,Lf)  d\mu=\int_{\mathbb{M}} (Lf)^2  d\mu.
\]
Therefore we obtain
\[
\int_{\mathbb{M}} (Lf)^2  d\mu \ge \rho \int_{\mathbb{M}} \Gamma (f,f) d\mu=-\frac{n\rho}{n-1}  \int_{\mathbb{M}} fLf d\mu.
\]
By density, this last inequality is seen to hold for every function $f\in \mathcal{D}(L)$. It means that the $L^2$ spectrum of $-L$ lies in $\{0\} \cup \left[\frac{n\rho}{n-1} ,+\infty\right)$. Since from the previous proof the projection of $f$ onto the $0$-eigenspace is given by $\frac{1}{\mu(\mathbb{M})} \int_\bM f d\mu$, we deduce that
\[
 \int_\mathbb{M} \left( f-\frac{1}{\mu(\mathbb{M})} \int_\bM f d\mu\right)^2d\mu \le \frac{n-1}{n\rho} \int_{\mathbb{M}} \Gamma(f,f) d\mu,
 \]
 which is exactly Poincar\'e inequality
\end{proof}

As observed in the proof, the Poincar\'e inequality
\[
\int_{\mathbb{M}} f^2 d\mu \le \left( \int_{\mathbb{M}} f d\mu \right)^2 +\frac{n-1}{n\rho} \int_{\mathbb{M}} \Gamma(f,f) d\mu.
\]
is equivalent to the fact that  the $L^2$ spectrum of $-L$ lies in $\{0\} \cup \left[\frac{n\rho}{n-1} ,+\infty\right)$, or in other words that $-L$ has a spectral gap of size at least $\frac{n\rho}{n-1}$. This is Lichnerowicz estimate. It is sharp, because on the $n$-dimensional sphere it is known that $\rho=n-1$ and that the first non zero eigenvalue is exactly equal to $n$.

As a basic consequence of the spectral theorem and of the above spectral gap estimate, we also get  the rate  convergence to equilibrium in $L^2(\M,\mu)$ for $P_t$.

\begin{proposition}
Let $f \in L^2(\M,\mu)$, then for $t \ge 0$,
\[
\left\| P_tf -\int_\mathbb{M} f d\mu  \right\|^2_2 \le e^{-\frac{2n\rho}{n-1} t} \left\| f -\int_\mathbb{M} f d\mu  \right\|^2_2 .
\]
\end{proposition}

As we have just seen, the convergence in $L^2$ of $P_t$  is connected and actually equivalent to the Poincar\'e inequality. 

We now turn to the so-called log-Sobolev inequality which is connected to the convergence in entropy for $P_t$. This inequality is much stronger (and more useful) than the Poincar\'e inequality. To simplify a little the expressions, we assume in the sequel that $\mu(\mathbb{M})=1$ (Otherwise, just replace $\mu$ by $\frac{\mu}{\mu(\mathbb{M})}$ in the following results).

\begin{theorem}
For $f \in \mathcal{D}(L)$, $f \ge 0$,
\[
 \int_{\mathbb{M}} f^2 \ln f^2 d\mu \le  \int_{\mathbb{M}} f^2 d\mu \ln \left( \int_{\mathbb{M}} f^2 d\mu \right) +\frac{2}{ \rho}  \int_{\mathbb{M}} \Gamma(f,f) d\mu.
 \]
\end{theorem}
\begin{proof}
By considering $\sqrt{f}$ instead of $f$, it is enough to show that if $f$ is positive,
\[
 \int_{\mathbb{M}} f \ln f  d\mu \le  \int_{\mathbb{M}} f  d\mu \ln \left( \int_{\mathbb{M}} f d\mu \right) +\frac{1}{2 \rho}  \int_{\mathbb{M}} \frac{\Gamma(f,f)}{f} d\mu.
 \]
 We now have
 \begin{align*}
 \int_{\mathbb{M}} f \ln f  d\mu -  \int_{\mathbb{M}} f  d\mu \ln \left( \int_{\mathbb{M}} f d\mu \right)&= - \int_0^{+\infty} \frac{d}{dt} \int_{\mathbb{M}} P_t f \ln P_t f d\mu dt \\
  &=- \int_0^{+\infty}  \int_{\mathbb{M}} L P_t f \ln P_t f d\mu dt \\
  &= \int_0^{+\infty}  \int_{\mathbb{M}} \Gamma (P_t f,  \ln P_t f) d\mu dt \\
  &=\int_0^{+\infty}  \int_{\mathbb{M}} \frac{\Gamma (P_t f,   P_t f) }{P_t f} d\mu dt 
 \end{align*}
 Now, we know that
 \[
 \Gamma (P_t f,   P_t f) \le e^{-2\rho t} \left( P_t \sqrt{\Gamma(f,f)}\right)^2.
 \]
 And, from Cauchy-Schwarz inequality,
 \[
 (P_t \sqrt{\Gamma(f,f)})^2  \le P_t  \frac{\Gamma(f,f)}{f} P_t f.
 \]
 Therefore,
 \[
 \int_{\mathbb{M}} f \ln f  d\mu -  \int_{\mathbb{M}} f  d\mu \ln \left( \int_{\mathbb{M}} f d\mu \right) \le \int_0^{+\infty} e^{-2\rho t} dt \int_{\mathbb{M}} \frac{\Gamma(f,f)}{f} d\mu,
 \]
 which is the required inequality.
\end{proof}

We finally prove the entropic convergence of $P_t$.

\begin{theorem}
Let $f \in L^2(\M,\mu)$, $f \ge 0$. For $t \ge 0$,
\[
  \int_{\mathbb{M}} P_t f \ln P_t  f  d\mu - \int_{\mathbb{M}} P_t f  d\mu \ln \left( \int_{\mathbb{M}} P_t f d\mu\right) \le e^{-2\rho t}  \left( \int_{\mathbb{M}} f \ln f  d\mu-  \int_{\mathbb{M}} f  d\mu \ln \left( \int_{\mathbb{M}} f d\mu\right)\right).
\]
\end{theorem}

\begin{proof}
Let us assume $\int_{\mathbb{M}} f d\mu=1$, otherwise we use the following argument with $\frac{f}{\int_{\mathbb{R}^n} f d\mu}$ and consider the functional 
\[
\Phi(t)= \int_{\mathbb{M}} P_t f \ln P_t  f  d\mu,
\]
which by differentiation gives
\[
\Phi'(t)= \int_{\mathbb{M}} LP_t f \ln P_t  f  d\mu=-\int_{\mathbb{M}} \frac{ \Gamma(P_t f)}{P_t f} d\mu.
\]
Using now the log-Sobolev inequality, we obtain
\[
\Phi'(t)\le -2\rho \Phi(t).
\]
The Gronwall's differential inequality implies then:
\[
\Phi(t) \le e^{-2\rho t} \Phi(0),
\]
that is
\[
 \int_{\mathbb{M}} P_t f \ln P_t  f  d\mu \le e^{-2\rho t} \int_{\mathbb{M}} f \ln f d\mu.
\]
\end{proof}

\subsection{The Li-Yau inequality}

Let $\mathbb{M}$ be a complete $n$-dimensional Riemannian manifold and, as usual, denote by $L$ its Laplace-Beltrami operator. Throughout the Section, we will assume again  that the Ricci curvature of $\mathbb{M}$ is bounded from below by $\rho \in \mathbb{R}$.  The Section is devoted to the proof of a beautiful inequality due to P. Li and S.T. Yau. 

\

Henceforth, we will indicate $C_b^\infty(\mathbb M) = C^\infty(\mathbb M)\cap L^\infty(\mathbb M)$.

\begin{lemma}\label{L:derivatives}
Let $f \in C^\infty_b(\mathbb{M})$, $f > 0$ and $T>0$, and consider the function
\[
\phi (x,t)=(P_{T-t} f) (x)\Gamma (\ln P_{T-t}f)(x),
\]
which is defined on $\mathbb{M} \times [0,T]$.  We have \[
L\phi+\frac{\partial \phi}{\partial t} =2 (P_{T-t} f) \Gamma_2 (\ln P_{T-t}f). 
\]
\end{lemma}

\begin{proof} Let for simplicity $g(x,t) = P_{T-t} f(x)$. A simple computation gives
\[
\frac{\p \phi}{\p t} = g_t \Gamma(\ln g) + 2 g \Gamma \left(\ln g,\frac{g_t}{g}\right).
\]
On the other hand,
\[
L\phi = Lg \Gamma(\ln g) + g L \Gamma(\ln g) + 2 \Gamma(g,\Gamma(\ln g)).
\]
Combining these equations we obtain
\[
L\phi + \frac{\p \phi}{\p t} = g L\Gamma(\ln g) +  2\Gamma(g,\Gamma(\ln g)) + 2 g \Gamma \left(\ln g,\frac{g_t}{g}\right).
\]
From \eqref{gamma2} we see that
\begin{align*}
2 g \Gamma_2(\ln g) & = g (L \Gamma(\ln g) - 2 \Gamma(\ln g,L(\ln g)))
\\
& = g L\Gamma(\ln g) - 2 g \Gamma(\ln g,L(\ln g)).
\end{align*}
Observing that
\[
L(\ln g) = - \frac{\Gamma(g)}{g^2} - \frac{g_t}{g},
\]
we conclude   that
\[
L\phi+\frac{\partial \phi}{\partial t} =2 (P_{T-t} f) \Gamma_2 (\ln P_{T-t}f). 
\]
\end{proof}

We now turn to an important variational inequality that shall extensively be used throughout these sections.  Given a function $f\in C^\infty_b(\mathbb{M})$ and $\ve>0$, we let $f_\varepsilon=f+\varepsilon$. 

Suppose that $T>0$, and $x\in \mathbb{M}$ be given. For a function $f\in  C^\infty_b(\mathbb M)$ with $f \ge 0$ we define for $t\in [0,T]$,
\[
\Phi (t)=P_t \left( (P_{T-t} f_\varepsilon) \Gamma (\ln P_{T-t}f_\varepsilon) \right).
\]

\begin{theorem}\label{T:source}
 Let $a \in C^1([0,T],[0,\infty))$ and  $\gamma \in C((0,T),\R)$.  Given $f \in C_0^\infty(\mathbb M)$, with $f\ge 0$, we have
\begin{align*}
  & a(T) P_T \left(  f_\varepsilon \Gamma (\ln f_\varepsilon) \right) -a(0)(P_{T} f_\varepsilon) \Gamma (\ln P_{T}f_\varepsilon)
 \\
 \ge &  \int_0^T \left(a'+2\rho a -\frac{4a\gamma}{n} \right)\Phi (s)  ds +\left(\frac{4}{n}\int_0^T a\gamma ds\right)LP_{T} f_\varepsilon -\left(\frac{2 }{n}\int_0^T a\gamma^2ds\right)P_T f_\varepsilon.
\end{align*}
\end{theorem}

\begin{proof}
Let $f \in C^\infty(\mathbb{M})$, $ f \ge 0$.
Consider the function
\[
\phi (x,t)=a(t)(P_{T-t} f) (x)\Gamma (\ln P_{T-t}f)(x)+b(t)(P_{T-t} f) (x) \Gamma^Z (\ln P_{T-t}f)(x).
\]
Applying the previous lemma and the curvature-dimension inequality, we obtain
\begin{align*}
  L\phi+\frac{\partial \phi}{\partial t} &
=a' (P_{T-t} f) \Gamma (\ln P_{T-t}f)+2a (P_{T-t} f) \Gamma_2 (\ln P_{T-t}f) \\
&\ge  \left(a'+2\rho a \right)(P_{T-t} f) \Gamma (\ln P_{T-t}f)+\frac{2a}{n}  (P_{T-t} f) (L(\ln P_{T-t} f))^2. 
\end{align*}
But, we have
\[
(L(\ln P_{T-t} f))^2 \ge 2\gamma L(\ln P_{T-t}f) -\gamma^2,
\]
and
\[
 L(\ln P_{T-t}f)=\frac{LP_{T-t}f}{P_{T-t}f} -\Gamma(\ln P_{T-t} f ).
 \]
Therefore we obtain,
 \begin{align*}
L\phi+\frac{\partial \phi}{\partial t}  & \ge \left(a'+2\rho a -\frac{4a\gamma}{n} \right) (P_{T-t} f) \Gamma (\ln P_{T-t}f) +\frac{4a\gamma}{n} LP_{T-t} f - \frac{2a\gamma^2}{n} P_{T-t} f.
\end{align*}
We then easily reach the conclusion by using the parabolic comparison theorem in $L^\infty$.
\end{proof}

As a first application the previous result, we derive a family of Li-Yau type inequalities. We choose the function $\gamma$ in a such a way that
\[
a' -\frac{4a\gamma}{n} +2\rho a=0.
\]
That is
\[
\gamma=\frac{n}{4} \left( \frac{a'}{a}+2\rho \right).
\]
Integrating  the inequality from $0$ to $T$, and denoting $V=\sqrt{a}$, we obtain the following result.

\begin{proposition}
Let $V:[0,T]\rightarrow \mathbb{R}^+$ be a smooth function such that
\[
V(0)=1, V(T)=0.
\]
We have
\begin{align}\label{family Li Yau}
\Gamma (\ln P_T f) & \le \left( 1-2\rho\int_0^T V^2(s) ds\right) \frac{L P_Tf}{P_T f} +\frac{n}{2} \left(  \int_0^T V'(s)^2 ds +\rho^2  \int_0^T V(s)^2 ds -\rho \right).
\notag\end{align}
\end{proposition}

A first family of interesting inequalities may be obtained with the choice
\[
V(t)=\left( 1-\frac{t}{T}\right)^\alpha, \alpha>\frac{1}{2}.
\]
In this case we have
\[
 \int_0^T V(s)^2 ds=\frac{T}{2\alpha+1}
 \]
 and
\[
 \int_0^T V'(s)^2 ds=\frac{\alpha^2}{(2\alpha-1)T},
 \]
In particular, we therefore proved the celebrated Li-Yau inequality:
 
 \begin{theorem}[Li-Yau inequality]
 If $f \in C_0^\infty(\mathbb{M})$, $f \ge 0$. For $\alpha >\frac{1}{2}$ and $T > 0$, we have
 \begin{equation*}
 \Gamma (\ln P_T f) \le \left( 1-\frac{2\rho T}{2\alpha+1}\right) \frac{L P_Tf}{P_T f} +\frac{n}{2} \left(   \frac{\alpha^2}{(2\alpha-1)T}+\frac{\rho^2 T}{2\alpha+1} -\rho \right).
 \end{equation*}
 \end{theorem}
 
In the case, $\rho=0$ and $\alpha=1$,  it reduces to the beautiful sharp inequality:
\begin{align}\label{LiYau}
\Gamma(\ln P_t f) \le \frac{  L P_t f }{P_t f }   + \frac{n}{2t},\ \ \ \ \ \ t>0 .
\end{align}

Although in the sequel,  we shall  firstl focus on the case $\rho=0$, let us presently briefly discuss the case $\rho >0$. 

Using the Li-Yau inequality  with $\alpha=3/2$ leads to the Bakry-Qian inequality:  
\[
 \frac{L P_tf}{P_t f}\le \frac{n \rho}{4},\ \ \ \ \ t \ge \frac{2}{\rho}.
 \]
 Also, by using 
 \[
 V(t)=\frac{e^{-\frac{\rho t}{3}} (e^{-\frac{2\rho t}{3}}-e^{-\frac{2\rho T}{3}})}{1-e^{-\frac{2\rho T}{3}}},
 \]
 we obtain the following inequality: 
 \[
\Gamma(\ln P_t f) \le e^{-\frac{2\rho t}{3}}  \frac{  L P_t f }{P_t f } +\frac{n\rho}{3} \frac{e^{-\frac{4\rho t}{3}}}{ 1-e^{-\frac{2\rho t}{3}}},\ \ \ \ \ t \ge 0.
 \]
\subsection{The parabolic Harnack inequality }

Let $\mathbb{M}$ be a complete $n$-dimensional Riemannian manifold and, as usual, denote by $L$ its Laplace-Beltrami operator. Throughout the section, we will assume that the Ricci curvature of $\mathbb{M}$ is bounded from below by $-K$ with $K \ge 0$.  Our purpose is to prove a first  important consequence of the Li-Yau  inequality: The parabolic Harnack inequality.

\begin{theorem}\label{T:harnack}
Let $f \in L^\infty(\mathbb{M})$, $f \ge 0$. For every $s \le t$ and $x,y \in \mathbb{M}$,
\begin{equation}\label{beauty}
P_s f(x) \le P_t f(y) \left(\frac{t}{s}\right)^{\frac{n}{2}} \exp\left(\frac{d(x,y)^2}{4(t-s)} +\frac{K d(x,y)^2}{6} +\frac{nK}{4}(t-s)\right).
\end{equation}
\end{theorem}

\begin{proof}
We first assume that $f \in C_0^\infty(\mathbb{M})$. Let $x,y \in \mathbb{M}$ and let $\gamma:[s,t] \to \mathbb{M}$, $s<t$ be an absolutely continuous path such that $\gamma(s)=x, \gamma(t)=y$.
We write the Li-Yau inequality in the form
\begin{align}\label{LIYAU2}
\Gamma( \ln P_u f (x) ) \le a(u) \frac{  L P_u f (x)}{P_u f (x)} +b(u),
 \end{align}
 where 
 \[
 a(u)=1+\frac{2K}{3} u, 
 \]
 and 
 \[
 b(u)=\frac{n}{2} \left(\frac{1}{u}+\frac{K^2 u}{3}+K \right).
 \]
 Let us now consider
 \[
 \phi(u)=\ln P_u f(\gamma(u)).
 \]
 We  compute
 \[
  \phi'(u)= ( \partial_u \ln P_u  f) (\gamma(u))+\langle \nabla \ln P_u f (\gamma(u)),\gamma'(u) \rangle.
 \]
 Now, for every $\lambda >0$, we have
 \[
 \langle \nabla \ln P_u f (\gamma(u)),\gamma'(u) \rangle \ge -\frac{1}{2\lambda^2} \| \nabla \ln P_u f (x) \|^2 -\frac{\lambda^2}{2} \| \gamma'(u) \|^2.
 \]
 Choosing  $\lambda=\sqrt{\frac{a(u)}{2} }$ and using then (\ref{LIYAU2}) yields
 \[
 \phi'(u) \ge -\frac{b(u)}{a(u)} -\frac{1}{4} a(u) \| \gamma'(u) \|^2.
 \]
 By integrating this inequality from $s$ to $t$ we get as a result.
 \[
 \ln P_tf(y)-\ln P_s f(x)\ge -\int_s^t \frac{b(u)}{a(u)} du  -\frac{1}{4} \int_s^t a(u) \| \gamma'(u) \|^2 du.
 \]
We now minimize the quantity  $\int_s^t a(u) \| \gamma'(u) \|^2 du$ over the set of absolutely continuous paths such that $\gamma(s)=x, \gamma(t)=y$. By using reparametrization of paths, it is seen that
 \[
 \int_s^t a(u) \| \gamma'(u) \|^2 du \ge \frac{d^2(x,y)}{\int_s^t \frac{dv}{a(v)}},
 \]
 with equality achieved for $\gamma(u)=\sigma\left( \frac{\int_s^u \frac{dv}{a(v)}}{\int_s^t \frac{dv}{a(v)}} \right)$ where $\sigma:[0,1] \to \mathbb{M}$ is a unit geodesic joining $x$ and $y$. As a conclusion,
 \[
 P_sf(x) \le \exp\left( \int_s^t \frac{b(u)}{a(u)} du + \frac{d^2(x,y)}{4\int_s^t \frac{dv}{a(v)}}  \right) P_tf(y).
 \]
 Now, from Cauchy-Schwarz inequality we have
 \[
 \int_s^t \frac{dv}{a(v)} \ge \frac{(t-s)^2}{\int_s^t a(v)dv}=\frac{(t-s)^2}{(t-s)+\frac{2K}{3}(t-s)^2 },
 \]
 and also
 \[
  \int_s^t \frac{b(u)}{a(u)} du=\frac{n}{2} \int_s^t \frac{ 1/u +K^2 u/3+K}{1+2Ku/3} du\le \frac{n}{2} \int_s^t \left( \frac{ 1}{u}+\frac{K}{2} \right) du
 \]
This proves
\eqref{beauty} when $f \in C_0^\infty(\mathbb{M})$. We can then extend the result to  $f \in L^\infty(\mathbb{M})$ by considering the approximations $h_n P_\tau f \in C_0^\infty(\mathbb{M})$ , where $h_n \in C_0^\infty(\mathbb{M})$, $h_n \ge 0$, $h_n \to_{n \to \infty} 1$ and let $n \to \infty$ and $\tau \to 0$.

\end{proof}

The following result represents an important consequence of Theorem
\ref{T:harnack}.

\begin{corollary}\label{C:harnackheat}
 Let $p(x,y,t)$ be the heat kernel on $\mathbb{M}$. For every $x,y, z\in
\mathbb{M}$ and every $0<s<t<\infty$ one has
\[
p(x,y,s) \le p(x,z,t) \left(\frac{t}{s}\right)^{\frac{n}{2}}  \exp\left(\frac{d(y,z)^2}{4(t-s)} +\frac{K d(y,z)^2}{6} +\frac{nK}{4}(t-s)\right).
\]
\end{corollary}

\begin{proof}
Let $\tau >0$ and $x\in \bM$ be fixed. By the hypoellipticity of $L - \p_t$, we know that $p(x,\cdot,\cdot +
\tau)\in C^\infty(\mathbb{M} \times (-\tau,\infty))$. From the semigroup property we
have
\[
p (x,y,s+\tau)=P_s (p(x,\cdot,\tau))(y)
\]
and
\[
p (x,z,t+\tau)=P_t (p(x,\cdot,\tau))(z)
\]
Since we cannot apply Theorem \ref{T:harnack} directly to $u(y,t) =
P_t(p(x,\cdot,\tau))(y)$, we consider 
$u_n(y,t) = P_t(h_n p(x,\cdot,\tau))(y)$, where $h_n\in
C^\infty_0(\bM)$, $0\le h_n\le 1$, and $h_n\nearrow  1$. From
\eqref{beauty} we find
\[
P_s (h_np(x,\cdot,\tau))(y) \le P_t (h_np(x,\cdot,\tau))(z)
 \left(\frac{t}{s}\right)^{\frac{n}{2}}  \exp\left(\frac{d(y,z)^2}{4(t-s)} +\frac{K d(y,z)^2}{6} +\frac{nK}{4}(t-s)\right)
\]
Letting $n \to \infty$, by Beppo Levi's monotone convergence theorem
we obtain
\[
p (x,y,s+\tau) \le p (x,z,t+\tau)
 \left(\frac{t}{s}\right)^{\frac{n}{2}}  \exp\left(\frac{d(y,z)^2}{4(t-s)} +\frac{K d(y,z)^2}{6} +\frac{nK}{4}(t-s)\right)
 \]
The desired conclusion follows by letting $\tau \to 0$.
\end{proof}

A nice  consequence of the parabolic Harnack inequality for the heat kernel is the following lower bound for the heat kernel:

\begin{proposition}[Cheeger-Yau lower bound]
For $x,z \in \mathbb{M}$ and $t>0$,
\[
 p(x,z,t) \ge  \frac{1}{(4\pi t)^{n/2}}  \exp\left(-\frac{d(x,z)^2}{4t} -\frac{K d(x,z)^2}{6} -\frac{nK}{4}t \right).
 \]
\end{proposition}

\begin{proof}
We just need to use the above  Harnack inequality with $y=x$ and let $s \to 0$ using the asymptotics $\lim_{s\to 0} s^{n/2} p_s(x,x)=\frac{1}{(4\pi  )^{n/2}}.$
\end{proof}

Observe that when $K=0$, the inequality is sharp, since it is actually an equality on the Euclidean space !

\subsection{The Gaussian upper bound}

Let $\mathbb{M}$ be a complete $n$-dimensional Riemannian manifold and, as usual, denote by $L$ its Laplace-Beltrami operator. As in the previous section, we will assume that the Ricci curvature of $\mathbb{M}$ is bounded from below by $-K$ with $K \ge 0$.  Our purpose in this section is to prove a Gaussian upper bound for the heat kernel. Our main tools are the parabolic Harnack inequality proved in the previous section and the following integrated maximum principle:

\begin{proposition}
Let $g :\mathbb{M}\times \mathbb{R}_{\ge 0}  \to \mathbb{R}$ be a non positive continuous function such that, in the sense of distributions,
\[
\frac{\partial g}{\partial t} +\frac{1}{2} \Gamma(g) \le 0,
\]
then, for every $ f \in L^2(\M,\mu)$, we have
\[
\int_{\mathbb{M}} e^{g(y,t)}  (P_t f)^2 (y) d\mu(y) \le \int_{\mathbb{M}} e^{g(y,0)}  f^2 (y) d\mu(y).
\]
\end{proposition}

\begin{proof}
Since
\[ \left(L-\frac{\partial }{\partial t}\right)(P_tf)^2 = 2 P_tf\left(L-\frac{\partial }{\partial  t}\right)(P_t f) + 2 \Gamma(P_t f) =
2 \Gamma(P_tf), \] multiplying this identity by $h_n^2(y) e^{g(y,t)}$,
where $h_n$ is the usual localizing sequence , and
integrating by parts, we obtain
\begin{align*}
0 & = 2 \int_0^\tau \int_{\mathbb{M}} h_n^2 e^g \Gamma(P_tf) d\mu(y) dt -
\int_0^\tau \int_{\mathbb{M}}h_n^2 e^g \left(L-\frac{\partial }{\partial t}\right)(P_tf)^2 d\mu(y) dt
\\
& = 2 \int_0^\tau \int_{\mathbb{M}} h_n^2 e^g \Gamma(P_tf) d\mu(y) dt + 4
\int_0^\tau \int_{\mathbb{M}} h_n e^g P_tf \Gamma(h_n,P_tf) d\mu(y) dt \\
& + 2 \int_0^\tau\int_{\mathbb{M}}h_n^2 e^g P_tf \Gamma(P_tf,g)d\mu(y) dt -
\int_0^\tau \int_{\mathbb{M}} h_n e^g (P_tf)^2 \frac{\partial g}{\partial t}  d\mu(y) dt
\\
& -   \int_{\mathbb{M}} h_n e^g (P_tf)^2 d\mu(y)\bigg|_{t=0} + \int_{\mathbb{M}} h_n e^g (P_tf)^2 d\mu(y)\bigg|_{t=\tau}
\\
& \ge 2 \int_0^\tau \int_{\mathbb{M}} h_n^2 e^g \left(\Gamma(P_tf)+P_tf\Gamma(P_tf,g)+\frac{P_tf^2}{4} \Gamma(g)\right) d\mu(y) dt + 4 \int_0^\tau \int_{\mathbb{M}} h_n
e^g P_tf \Gamma(h_n,P_tf) d\mu(y) dt
\\
& + \int_{\mathbb{M}} h_n e^g (P_tf)^2 d\mu(y)\bigg|_{t=\tau} -   \int_{\mathbb{M}} h_n e^g (P_tf)^2 d\mu(y)\bigg|_{t=0}.
\end{align*}
From this we conclude
\[
\int_{\bM} h_n e^g (P_t f)^2 d\mu(y)\bigg|_{t=\tau} \le \int_{\mathbb{M}} h_n e^g (P_tf)^2 d\mu(y)\bigg|_{t=0} - 4 \int_0^\tau \int_{\bM} h_n e^g P_tf \Gamma(h_n,P_tf) d\mu(y) dt.
\]
We now claim that
\[
\underset{n\to \infty}{\lim} \int_0^\tau \int_{\mathbb{M}} h_n e^g P_tf \Gamma(h_n,P_tf) d\mu(y) dt = 0.
\]
To see this we apply Cauchy-Schwarz inequality which gives
\begin{align*}
& \left|\int_0^\tau \int_{\mathbb{M}} h_n e^g P_tf \Gamma(h_n,P_tf) d\mu(y) dt\right|\le \left(\int_0^\tau \int_{\mathbb{M}} h_n^2 e^g (P_tf)^2 \Gamma(h_n) d\mu(y) dt\right)^{\frac{1}{2}} \left(\int_0^\tau \int_{\mathbb{M}} e^g \Gamma(P_tf) d\mu(y) dt\right)^{\frac{1}{2}}
\\
& \le \left(\int_0^\tau \int_{\mathbb{M}} e^g (P_tf)^2 \Gamma(h_n) d\mu(y)
dt\right)^{\frac{1}{2}} \left(\int_0^\tau \int_{\mathbb{M}} e^g \Gamma(P_tf)
d\mu(y) dt\right)^{\frac{1}{2}} \to 0,
\end{align*}
as $n\to \infty$. With the
claim in hands we now let $n\to \infty$ in the above inequality
obtaining
\[
\int_{\mathbb{M}} e^{g(y,t)} ) (P_t f)^2 (y) d\mu(y) \le \int_{\mathbb{M}} e^{g(y,0)} ) f^2 (y) d\mu(y)
\]
\end{proof}

We are now ready for the main bound of this section:

\begin{theorem}\label{T:ub}
For any $0<\epsilon <1$ there exist positive  constants $C_1=C_1(\epsilon)$ and   $C_2=C_2 (n,\epsilon)>0$, such that for every $x,y\in \mathbb{M}$ and $t>0$ one has
\[
p(x,y,t)\le \frac{C_1}{\mu(B(x,\sqrt t))^{\frac{1}{2}}\mu(B(y,\sqrt t))^{\frac{1}{2}}} \exp \left(C_2Kt-\frac{d(x,y)^2}{(4+\epsilon)t}\right).
\]
\end{theorem}

\begin{proof}
Given $T>0$, and $\alpha>0$ we fix $0<\tau \le (1+\alpha)T$. For a function $\psi\in C^\infty_0(\mathbb{M})$, with $\psi \ge 0$, in $\mathbb{M} \times (0,\tau)$ we consider the function
\[
f(y,t) = \int_{\mathbb{M}} p(y,z,t) p(x,z,T) \psi(z) d\mu(z),\ \ \ x\in
\bM.
\]
Since $f = P_t(p(x,\cdot,T)\psi)$, it satisfies the Cauchy problem
\[
\begin{cases}
Lf - f_t = 0 \ \ \ \ \text{in}\ \mathbb{M} \times (0,\tau),
\\
f(z,0) = p(x,z,T)\psi(z),\ \  \ z\in \mathbb{M}.
\end{cases}
\]
Let $g :\mathbb{M}\times [0,\tau]  \to \mathbb{R}$ be a non positive continuous function such that, in the sense of distributions,
\[
\frac{\partial g}{\partial t} +\frac{1}{2} \Gamma(g) \le 0.
\]
From the previous lemma, we know that:
\begin{equation}\label{CYpsi}
\int_{\bM}  e^{g(y,\tau)} f^2(y,\tau) d\mu(y) \le \int_{\mathbb{M}}
e^{g(y,0)} f^2(y,0) d\mu(y).
\end{equation}
At this point we fix $x\in \mathbb{M}$ and for $0<t\le\tau$ consider the
indicator function $\mathbf 1_{B(x,\sqrt t)}$ of the ball $B(x,\sqrt
t)$. Let $\psi_k\in C^\infty_0(\mathbb{M})$, $\psi_k \ge 0$, be a sequence
such that $\psi_k \to \mathbf 1_{B(x,\sqrt t)}$ in $L^2(\mathbb{M})$, with
supp$\ \psi_k\subset B(x,100\sqrt t)$. Slightly abusing the notation
we now set \[ f(y,s) = P_s(p(x,\cdot,T)\mathbf{1}_{B(x,\sqrt t)})(y)
= \int_{B(x,\sqrt t)} p(y,z,s) p(x,z,T) d\mu(z). \] Thanks to the
symmetry of $p(x,y,s) = p(y,x,s)$, we have
\begin{equation}\label{psquare}
 f(x,T) = \int_{B(x,\sqrt t)} p(x,z,T)^2 d\mu(z).
\end{equation}

Applying \eqref{CYpsi} to $f_k(y,s) = P_s(p(x,\cdot,T)\psi_k)(y)$,
we find
\begin{equation}\label{CYpsik}
\int_{\mathbb{M}}  e^{g(y,\tau)} f^2_k(y,\tau) d\mu(y) \le \int_{\mathbb{M}}
e^{g(y,0)} f^2_k(y,0) d\mu(y).
\end{equation}
At this point we observe that as $k\to \infty$
\begin{align*}
& \left|\int_{\mathbb{M}}  e^{g(y,\tau)} f^2_k(y,\tau) d\mu(y) - \int_{\mathbb{M}}
e^{g(y,\tau)} f^2(y,\tau) d\mu(y)\right|
\\
& \le 2 ||e^{g(\cdot,\tau)}||_{L^\infty(\mathbb{M})}
||p(x,\cdot,T)||_{L^2(\mathbb{M})} ||p(x,\cdot,\tau)||_{L^\infty(B(x,110
\sqrt t))} ||\psi_k - \mathbf 1_{B(x,\sqrt t)}||_{L^2(\mathbb{M})} \to 0.
\end{align*}
By similar considerations we find
\begin{align*}
& \left|\int_{\bM}  e^{g(y,0)} f^2_k(y,0) d\mu(y) - \int_{\mathbb{M}}
e^{g(y,0)} f^2(y,0) d\mu(y)\right|
\\
& \le 2 ||e^{g(\cdot,0)}||_{L^\infty(\mathbb{M})}
||p(x,\cdot,T)||_{L^\infty(B(x,110 \sqrt t))} ||\psi_k - \mathbf
1_{B(x,\sqrt t)}||_{L^2(\mathbb{M})} \to 0.
\end{align*}
Letting $k\to \infty$ in \eqref{CYpsik} we thus conclude that the
same inequality holds with $f_k$ replaced by $f(y,s) =
P_s(p(x,\cdot,T)1_{B(x,\sqrt t)})(y)$. This implies in particular
the basic estimate
\begin{align}\label{be}
& \underset{z\in B(x,\sqrt t)}{\inf}\ e^{g(z,\tau)} \int_{B(x,\sqrt
t)} f^2(z,\tau) d\mu(z)
\\
& \le \int_{B(x,\sqrt t)} e^{g(z,\tau)} f^2(z,\tau) d\mu(z) \le
\int_{\bM} e^{g(z,\tau)} f^2(z,\tau) d\mu(z) \notag\\
& \le \int_{\bM} e^{g(z,0)} f^2(z,0) d\mu(z) = \int_{B(y,\sqrt t)}
e^{g(z,0)} p(x,z,T)^2 d\mu(z) \notag\\
& \le \underset{z\in B(y,\sqrt t)}{\sup}\ e^{g(z,0)} \int_{B(y,\sqrt
t)} p(x,z,T)^2 d\mu(z). \notag
\end{align}

At this point we choose in \eqref{be}
\[ g(y,t) = g_x(y,t) = -
\frac{d(x,y)^2}{2((1+2\alpha) T - t)}.
\]
Using the fact that $\Gamma(d)\le 1$, one can easily check that $g$ satisfies
\[
\frac{\partial g}{\partial t} +\frac{1}{2} \Gamma(g) \le 0.
\]
Taking into account that
\[
\underset{z\in B(x,\sqrt t)}{\inf}\ e^{g_x(z,\tau)} = \underset{z\in
B(x,\sqrt t)}{\inf}\ e^{-\frac{d(x,z)^2}{2((1+2\alpha)T- \tau)}} \ge
e^{\frac{-t}{2((1+2\alpha)T- \tau)}},
\]
if we now choose $\tau = (1+\alpha)T$, then from the previous
inequality and from \eqref{psquare} we conclude that
\begin{equation}\label{lemmaub}
\int_{B(x,\sqrt t)} f^2(z,(1+\alpha)T) d\mu(z) \le
\left(\underset{z\in B(y,\sqrt t)}{\sup}\
e^{-\frac{d(x,z)^2}{2(1+2\alpha)T} + \frac{t}{2\alpha T}}\right)
\int_{B(y,\sqrt t)} p(x,z,T)^2 d\mu(z).
\end{equation}
We now apply Theorem \ref{T:harnack} which gives for every $z\in
B(x,\sqrt t)$
\[
f(x,T)^2 \le f(z,(1+\alpha)T)^2(1+\alpha)^{n} e^{\frac{t}{2\alpha T}+\frac{Kt}{3} +\frac{nK\alpha T}{2 }}.
\]
Integrating this inequality on $B(x,\sqrt t)$ we find
\[
\left(\int_{B(y,\sqrt t)} p(x,z,T)^2 d\mu(z)\right)^2 = f(x,T)^2 \le
\frac{(1+\alpha)^{n} e^{\frac{t}{2\alpha T}+\frac{Kt}{3} +\frac{nK\alpha T}{2 }}}{\mu(B(x,\sqrt
t)) } \int_{B(x,\sqrt t)} f^2(z,(1+\alpha)T) d\mu(z).
\]
If we now use \eqref{lemmaub} in the last inequality we obtain
\begin{align*}
& \int_{B(y,\sqrt t)} p(x,z,T)^2 d\mu(z) \le
\frac{(1+\alpha)^{n} e^{\frac{t}{2\alpha T}+\frac{Kt}{3} +\frac{nK\alpha T}{2 }}}{\mu(B(x,\sqrt
t)) } \left(\underset{z\in B(y,\sqrt t)}{\sup}\
e^{-\frac{d(x,z)^2}{2(1+2\alpha)T} + \frac{t}{2\alpha T}}\right).
\end{align*}
Choosing $T = (1+\alpha)t$ in this inequality we find
\begin{align}\label{ub2}
& \int_{B(y,\sqrt t)} p(x,z,(1+\alpha)t)^2 d\mu(z) \le
\frac{(1+\alpha)^{n}
e^{\frac{Kt}{3} +\frac{nK}{2} \alpha (1+\alpha)t+
\frac{1}{2\alpha (1+\alpha)}}}{\mu(B(x,\sqrt t))}
\left(\underset{z\in B(y,\sqrt t)}{\sup}\
e^{-\frac{d(x,z)^2}{2(1+2\alpha)(1+\alpha)t} + \frac{1}{2\alpha
(1+\alpha)}}\right).
\end{align}
We now apply Corollary \ref{C:harnackheat} obtaining for every $z\in
B(y,\sqrt t)$
\[
p(x,y,t)^2 \le p(x,z,(1+\alpha)t)^2 (1+\alpha)^{n} \exp\left(\frac{1}{2\alpha }+\frac{Kt}{3}+\frac{nK\alpha t}{4} \right).
\]
Integrating this inequality in $z\in B(y,\sqrt t)$, we have
\[
\mu(B(y,\sqrt t)) p(x,y,t)^2 \le (1+\alpha)^{n} \exp\left(\frac{1}{2\alpha }+\frac{Kt}{3}+\frac{nK\alpha t}{4} \right) \int_{B(y,\sqrt t)}
p(x,z,(1+\alpha)t)^2 d\mu(z).
\]
Combining this inequality with \eqref{ub2} we conclude
\[
p(x,y,t) \le \frac{(1+\alpha)^{n} e^{\frac{3+\alpha}{4\alpha(1+\alpha)} +\frac{Kt}{3} +\frac{nKt}{4} \left( \alpha^2 +\frac{3}{2} \alpha \right) }}{\mu(B(x,\sqrt
t))^{\frac{1}{2}}\mu(B(y,\sqrt
t))^{\frac{1}{2}}}\left(\underset{z\in B(y,\sqrt t)}{\sup}\
e^{-\frac{d(x,z)^2}{4(1+2\alpha)(1+\alpha)t}}\right).
\]
If now $x\in B(y,\sqrt t)$, then
\[
d(x,z)^2 \ge (d(x,y) - \sqrt t)^2 > d(x,y)^2 - t,
\]
and therefore
\[
\underset{z\in B(y,\sqrt t)}{\sup}\
e^{-\frac{d(x,z)^2}{4(1+2\alpha)(1+\alpha)t}} \le
e^{\frac{1}{4(1+2\alpha)(1+\alpha)}}
e^{-\frac{d(x,y)^2}{4(1+2\alpha)(1+\alpha)t}}.
\]
If instead $x\not\in B(y,\sqrt t)$, then for every $\delta >0$ we
have
\[
d(x,z)^2 \ge (1-\delta) d(x,y)^2  - (1+ \delta^{-1}) t
\]
Choosing $\delta = \alpha/(\alpha+1)$ we find
\[
d(x,z)^2 \ge \frac{d(x,y)^2}{1+\alpha}  - (2 + \alpha^{-1}) t,
\]
and therefore
\[
\underset{z\in B(y,\sqrt t)}{\sup}\
e^{-\frac{d(x,z)^2}{4(1+2\alpha)(1+\alpha)t}} \le
e^{-\frac{d(x,y)^2}{4(1+2\alpha)(1+\alpha)^2 t} + \frac{2 +
\alpha^{-1}}{4(1+2\alpha)(1+\alpha)}}
\]
For any $\epsilon >0$ we now choose $\alpha>0$ such that
$4(1+2\alpha)(1+\alpha)^2 = 4+\epsilon$ to reach the desired
conclusion.

\end{proof}

\subsection{Volume doubling property}

In this Section we consider a complete and $n$-dimensional  Riemannian manifold $(\mathbb{M},g)$ with non negative Ricci curvature. Our goal is to prove  the following fundamental result, which is known as the volume doubling property. We follow an approach developed by N. Garofalo and the author.

\begin{theorem}
There exists a constant $C=C(n)>0$ such that for every $x\in \mathbb{M}$ and every $r>0$ one has
\[
\mu(B(x,2r))\le C \mu (B(x,r)).
\]
\end{theorem}

Actually by suitably adapting the arguments given in the sequel, the previous result can be extended to the case of negative Ricci curvature as follows:
\begin{theorem}
Assume $\mathbf{Ric} \ge -K$ with $K \ge 0$. There exist positive  constants $C_1=C_1(n,K), C_2=C_2(n,K)$ such that for every $x\in \mathbb{M}$ and every $r>0$ one has
\[
\mu(B(x,2r))\le C_1e^{K r^2} \mu (B(x,r)).
\]
\end{theorem}
For simplicity, we show the arguments in the case $K=0$ and let the reader work out the arguments in the case $K \neq 0$.

This result can be obtained from geometric methods as a consequence of the Bishop-Gromov comparison theorem. The proof we give instead only relies on the previous methods and has the advantage to generalize to a much larger class of operators than Laplace-Beltrami on Riemannian manifolds.

The key heat kernel estimate that leads to the doubling property is the following uniform and scale invariant lower bound on the heat kernel measure of balls.

\begin{theorem}
There exist an absolute constant $K>0$, and $A>0$, depending only on $n$, such that
\begin{equation*}
 P_{Ar^2}(\mathbf 1_{B(x,r)})(x) \ge K, \ \ \ \ \ x\in \mathbb{M}, r>0.
\end{equation*}
\end{theorem}

\begin{proof}
We first recall the following result that was proved in a previous section: Let $a \in C^1([0,T],[0,\infty))$ and  $\gamma \in C((0,T),\R)$.  Given $f \ge 0$, which is bounded and such that $\sqrt{f}$ is Lipschitz, we have
\begin{align*}
  & a(T) P_T \left(  f \Gamma (\ln f) \right) -a(0)(P_{T} f) \Gamma (\ln P_{T}f)
 \\
 \ge &  \int_0^T \left(a'+2\rho a -\frac{4a\gamma}{n} \right)\Phi (s)  ds +\left(\frac{4}{n}\int_0^T a\gamma ds\right)LP_{T} f -\left(\frac{2 }{n}\int_0^T a\gamma^2ds\right)P_T f.
\end{align*}

We choose
\[
a(t)=\tau+T-t,
\]
\[
\gamma(t)=-\frac{n}{4(\tau+T-t)} 
\]
where  $\tau >0$ will later be optimized. Noting that we presently have
\[
a' = 1,\ \ \ a\gamma = - \frac{n}{4},\ \ \ \ a\gamma^2 = \frac{n^2}{16(\tau + T - t)^2},
\]
we obtain the inequality
\begin{equation*}
\tau P_T(f \Gamma(\ln f)) -(T+\tau) P_T f \Gamma(\ln P_T f)
\ge -T L P_T f -\frac{n}{8} \ln \left( 1+\frac{T}{\tau}\right) P_T f
\end{equation*}
In what follows we consider a bounded function $f$ on $\mathbb{M}$ such that 
$\Gamma(f) \le 1$ almost everywhere on $\mathbb{M}$.  For any $\lambda \in \mathbb R$
we consider the function $\psi$ defined by 
\[
\psi(\lambda,t) = \frac{1}{\lambda} \log P_t(e^{\lambda f}), \ \ \ \text{or alternatively}\ \ \ P_t(e^{\lambda f}) = e^{\lambda \psi}.
\]
Notice that Jensen's inequality gives $\lambda \psi \ge  \lambda P_t f,$ and so we have $P_t f \le \psi.$

We now apply the previous inequality to the function $e^{\lambda f}$, obtaining
\begin{align*}
 \lambda^2 \tau P_T\left(e^{\lambda f} \Gamma(f)\right) - \lambda^2 (T+\tau) e^{\lambda \psi} \Gamma(\psi) \ge - T L P_T(e^{\lambda f}) - \frac{n}{8} e^{\lambda \psi} \ln\left(1+\frac{T}{\tau}\right).
\end{align*}
Keeping in mind that 
$\Gamma(f) \le 1$, we see that  $P_T(e^{\lambda f} \Gamma(f)) \le  e^{\lambda \psi}.$
Using this observation in combination with the fact that
\[
L \left(P_t (e^{\lambda f})\right) = \frac{\partial}{\partial t} \left(P_t (e^{\lambda f})\right) = \frac{\partial e^{\lambda \psi}}{\partial t} = \lambda e^{\lambda \psi} \frac{\partial \psi}{\partial t} ,
\]
and switching notation from $T$ to $t$, we infer
\begin{align*}
 \lambda^2 \tau  \ge \lambda^2 (t+\tau) e^{\lambda \psi} \Gamma(\psi)- \lambda t \frac{\partial \psi}{\partial t} - \frac{n}{8}  \ln\left(1+\frac{t}{\tau}\right).
\end{align*}
The latter inequality finally gives
\begin{equation}\label{ine3}
\frac{\partial \psi}{\partial t} \ge - \frac{\lambda}{t}\left(\tau + \frac{n}{8\lambda^2} \ln\left(1+\frac{t}{\tau}\right)\right)\ge - \frac{\lambda}{t}\left(\tau + \frac{nt}{8\lambda^2\tau} \right).
\end{equation}
We now optimize the right-hand side of the inequality with respect to $\tau$. We notice explicitly that the maximum value of the right-hand side is attained at
\[
\tau_0 = \sqrt{\frac{nt}{8 \lambda^2} }.
\]
We find therefore
\[
\frac{\partial \psi}{\partial t} \ge -\sqrt{\frac{n}{2t} }
\]
We now integrate the inequality  between $s$ and $t$, obtaining
\[
\psi(\lambda,s) \le \psi(\lambda,t)   +\sqrt{ \frac{n}{2}} \int_{s}^t   \frac{d\tau}{\sqrt{\tau}}.
\]
 We infer then
\begin{equation*}
P_s(\lambda f) \le  \lambda \psi(\lambda,t)  + \lambda \sqrt{ 2nt}.
\end{equation*}
Letting $s\to 0^+$ we conclude
\begin{equation}\label{ine5}
\lambda f \le  \lambda \psi(\lambda,t)  + \lambda \sqrt{ 2nt}.
\end{equation}

At this point we   let $B = B(x,r) = \{x\in \mathbb{M}\mid d(y,x)<r\}$, and consider the function $f(y) = - d(y,x)$. Since we clearly have
 \[
 e^{\lambda f} \le e^{-\lambda r} \mathbf 1_{B^c} + \mathbf 1_B,
 \]
 it follows that for every $t>0$ one has
 \[
 e^{\lambda \psi(\lambda,t)(x)}  = P_t(e^{\lambda f})(x) \le e^{-\lambda r} + P_t(\mathbf 1_B)(x).
 \]
 This gives the lower bound
 \[
  P_t(\mathbf 1_B)(x) \ge  e^{\lambda \psi(\lambda,t)(x)} - e^{-\lambda r} .
  \]
To estimate the first term in the right-hand side of the latter inequality, we use the previous estimate  which gives
\[
 P_{t}(\mathbf 1_B)(x) \ge  e^{-\lambda \sqrt{2nt}} - e^{-\lambda r}.
 \]
To make use of this estimate, we now choose $\lambda = \frac{1}{r}$, $t = Ar^2$, obtaining
\[
 P_{Ar^2}(\mathbf 1_B)(x) \ge  e^{-A\sqrt{2n}} - e^{-1}.
\]
The conclusion follows then easily.
\end{proof}

We now turn to the proof of the volume doubling property. We first recall the following basic result which is a straightforward consequence of the Li Yau inequality.

\begin{corollary}\label{C:harnackheat}
Let $p(x,y,t)$ be the heat kernel on $\bM$. For every $x,y, z\in
\bM$ and every $0<s<t<\infty$ one has
\[
p(x,y,s) \le p(x,z,t) \left(\frac{t}{s}\right)^{\frac{n}{2}}
\exp\left( \frac{d(y,z)^2}{4(t-s)} \right).
\]
\end{corollary}

We are now in position to prove the volume doubling property.

From the semigroup property and the symmetry of the heat kernel we
have for any $y\in \bM$ and $t>0$
\[ p(y,y,2t) = \int_\bM  p(y,z,t)^2 d\mu(z).
\]

Consider now a function $h\in C^\infty_0(\bM)$ such that $0\le h\le
1$, $h\equiv 1$ on $B(x,\sqrt{t}/2)$ and $h\equiv 0$ outside
$B(x,\sqrt t)$. We thus have
\begin{align*}
P_t h(y) & = \int_\bM p(y,z,t) h(z) d\mu(z) \le \left(\int_{B(x,\sqrt t)}
p(y,z,t)^2 d\mu(z)\right)^{\frac{1}{2}} \left(\int_\bM h(z)^2
d\mu(z)\right)^{\frac{1}{2}}
\\
& \le p(y,y,2t)^{\frac{1}{2}} \mu(B(x,\sqrt t))^{\frac{1}{2}}.
\end{align*}
If we take $y=x$, and $t =r^2$, we obtain
\begin{equation}\label{ine7}
P_{r^2} \left(\mathbf 1_{B(x,r)}\right)(x)^2 \le P_{r^2} h(x)^2 \leq
p(x,x,2r^2)\ \mu(B(x,r)).
\end{equation}
At this point we use the crucial previous theorem, which gives for some $0<A = A(n)<1$
\[
 P_{Ar^2}(\mathbf 1_{B(x,r)})(x) \ge K, \ \ \ \ \ x\in \mathbb{M}, r>0.
\]
Combining the latter inequality with the Harnack inequality and with \eqref{ine7}, we obtain the following 
on-diagonal lower bound
\begin{equation}\label{odlb}
p(x,x,2r^2) \ge \frac{K^*}{\mu(B(x,r))},\ \ \ \ \ x\in \bM,\ r>0.
\end{equation}
Applying Corollary \ref{C:harnackheat} to $(y,t)\to p(x,y,t)$
for every $y\in B(x,\sqrt t)$ we find
\[
p(x,x,t) \le  C(n) p(x,y,2t).
\]
Integration over $B(x,\sqrt t)$ gives
\[
p(x,x,t)\mu(B(x,\sqrt t)) \le C(n) \int_{B(x,\sqrt
t)}p(x,y,2t)d\mu(y) \le C(n),
\]
where we have used $P_t1\le 1$. Letting $t = r^2$, we obtain from this the on-diagonal upper
bound \begin{equation}\label{odub} p(x,x,r^2) \le
\frac{C(n)}{\mu(B(x,r))}.
\end{equation}
Combining \eqref{odlb} with \eqref{odub} we finally obtain
\[
\mu(B(x,2r)) \le \frac{C}{p(x,x,4r^2)} \le \frac{C^*}{p(x,x,2r^2)}
\le C^{**} \mu(B(x,r)),
\]
where we have used once more Corollary \ref{C:harnackheat}.
which gives
\[
\frac{p(x,x,2r^2)}{p(x,x,4r^2)}\le C.
\]

\subsection{Upper and lower Gaussian bounds for the heat kernel}

In this short  section, as in the previous one,  we consider a complete and $n$-dimensional  Riemannian manifold $(\mathbb{M},g)$ with non negative Ricci curvature. 
The volume doubling property that was proved is closely related to  sharp lower and upper Gaussian bounds that are due to P. Li and S.T. Yau.
We first record a basic consequence of the volume doubling property whose proof is let to the reader.

\begin{theorem}
Let $C>0$ be the constant such that for every $x \in \mathbb{M}$, $R>0$,
\[
\mu(B(x,2R)) \le C \mu (B(x,R)).
\]
Let $Q = \log_2 C$. For any $x\in \mathbb{M}$ and $r>0$ one has
\[
\mu(B(x,tr)) \ge C^{-1} t^{Q} \mu(B(x,r)),\ \ \ 0\le t\le 1.
\]
\end{theorem}

We are now in position to prove the main result of the section.
\begin{theorem}\label{T:gb}
For any $0<\ve <1$
there exists a constant $C = C(n,\ve)>0$, which tends to $\infty$ as $\ve \to 0^+$, such that for every $x,y\in \mathbb{M}$ and $t>0$ one has
\[
\frac{C^{-1}}{\mu(B(x,\sqrt
t))} \exp
\left(-\frac{ d(x,y)^2}{(4-\ve)t}\right)\le p(x,y,t)\le \frac{C}{\mu(B(x,\sqrt
t))} \exp
\left(-\frac{d(x,y)^2}{(4+\ve)t}\right).
\]
\end{theorem}

\begin{proof}
We begin by establishing the lower bound. First, from the Harnack inequality we obtain for all $y \in \mathbb{M}$, $t>0$, and every $0<\ve <1$,
\begin{align*}
p(x,y,t)\ge p(x,x,\ve t)  \ve^\frac{n}{2} \exp\left( -\frac{d(x,y)^2}{(4-\ve)t}\right).
\end{align*}
We thus need to estimate $p(x,x,\ve t)$ from below. But this has already been done in the proof of the volume doubling property where we established:
\[
p(x,x,\ve t) \ge \frac{C^*}{\mu(B(x,\sqrt{\ve/2} \sqrt t))},\ \ \ \ \ x\in \mathbb{M},\ t>0.
\]
On the other hand, since $\sqrt{\ve/2}<1$, by the trivial inequality $\mu(B(x,\sqrt{\ve/2} \sqrt t)) \le \mu(B(x,\sqrt t))$, we conclude
\[
p(x,y,t) \geq \frac{C^*}{ \mu(B(x,\sqrt t))}  \ve^\frac{n}{2} \exp\left( -\frac{d(x,y)^2}{(4-\ve)t}\right).
\]
This proves the Gaussian lower bound.

For the Gaussian upper bound, we first observe that the following upper bound was proved in a previous section:
$$
p(x,y,t)\le \frac{C}{\mu(B(x,\sqrt
t))^{\frac{1}{2}} \mu(B(y,\sqrt
t))^{\frac{1}{2}}} \exp
\left(-\frac{d(x,y)^2}{(4+\ve')t}\right).
$$
At this point, by the triangle inequality and the volume doubling property we find.
\beas
\mu(B(x,\sqrt{ t})) &\leq& \mu(B(y,d(x,y)+\sqrt{ t}))\\
                 &\leq& C_1 \mu(B(y,\sqrt{ t})) \left(\frac {d(x,y)+\sqrt{ t}}{\sqrt t} \right)^Q.
\eeas
with $Q=\log_2 C$, where $C$ is the doubling constant.
This gives
$$
\frac{1}{\mu(B(y,\sqrt{ t}))}\leq \frac{C_1}{\mu(B(x,\sqrt{ t}))} \left(\frac {d(x,y)}{\sqrt{ t}}+1 \right)^Q.
$$
Combining this with the above estimate we obtain
\[
p(x,y,t)\le \frac{C_1^{1/2}C}{\mu(B(x,\sqrt
t))}  \left(\frac {d(x,y)}{\sqrt{ t}}+1 \right)^{\frac{Q}{2}} \exp
\left(-\frac{d(x,y)^2}{(4+\ve')t}\right).
\]
If now $0<\ve<1$, it is clear that we can choose $0<\ve'<\ve$ such that 
\[
\frac{C_1^{1/2}C}{\mu(B(x,\sqrt
t))}  \left(\frac {d(x,y)}{\sqrt{ t}}+1 \right)^{\frac{Q}{2}} \exp
\left(-\frac{d(x,y)^2}{(4+\ve')t}\right) \le  \frac{C^*}{\mu(B(x,\sqrt
t))} \exp
\left(-\frac{d(x,y)^2}{(4+\ve)t}\right),
\]
where $C^*$ is a constant which tends to $\infty$ as $\ve \to 0^+$. The desired conclusion follows by suitably adjusting the values of both $\ve'$ and of the constant in the right-hand side of the estimate.

\end{proof}

To conclude, we finally mention without proof, what the previous arguments give in the case where $\mathbf{Ric} \ge -K$ with $K \ge 0$. We encourage the reader to do the proof by herself/himself as an exercise.

\begin{theorem}
Let us assume $\mathbf{Ric} \ge -K$ with $K \ge 0$. For any $0<\ve <1$
there exist  constants $C_1,C_2= C(n,K,\ve)>0$, such that for every $x,y\in \mathbb{M}$ and $t>0$ one has
\[
 p(x,y,t)\le \frac{C_1}{\mu(B(x,\sqrt
t))} \exp
\left(-\frac{d(x,y)^2}{(4+\ve)t} +KC_2 (t +d(x,y)^2)\right).
\]
\[
 p(x,y,t)\ge \frac{C^{-1}_1}{\mu(B(x,\sqrt
t))} \exp
\left(-\frac{d(x,y)^2}{(4-\ve)t} -KC_2 (t +d(x,y)^2)\right).
\]
\end{theorem}

\subsection{The Poincar\'e inequality on domains}

Let $(\mathbb{M},g)$ be a complete Riemannian manifold and $\Omega \subset \mathbb{M}$ be a non empty bounded set. Let $\mathcal{D}^\infty $ be the set of smooth functions $f \in C^\infty(\bar{\Omega})$ such that for every $g \in C^\infty(\bar{\Omega})$,
\[
\int_\Omega g Lf d\mu=-\int_\Omega \Gamma(f,g) d\mu.
\]
It is easy to see that $L$ is essentially self-adjoint on $\mathcal{D}^\infty $. Its Friedrichs extension, still denoted $L$,  is called the Neumann Laplacian on $\Omega$ and the semigroup it generates, the Neumann semigroup. If the boundary $\partial \Omega$ is smooth, then it is known from the Green's formula that
\[
\int_\Omega g Lf d\mu=-\int_\Omega \Gamma(f,g) d\mu+\int_{\partial \Omega} g Nf d\mu,
\]
where $N$ is the normal unit vector. As a consequence, $f \in \mathcal{D}^\infty$ if and only if $Nf=0$. However, we stress that no regularity assumption on the boundary $\partial \Omega$ is needed to define the Neumann Laplacian and the Neumann semigroup.

Since $\bar{\Omega}$ is compact, the Neumann semigroup is a compact operator and $-L$  has a discrete spectrum $0 =\lambda_0 <\lambda_1 \le \cdots$. We get then, the so-called Poincar\'e inequality on $\Omega$: For every $f \in C^\infty(\bar{\Omega})$, 
\[
\int_\Omega (f-f_\Omega)^2 d\mu \le \frac{1}{\lambda_1} \int_\Omega \Gamma(f) d\mu.
\]
Our goal is  to understand how the constant $\lambda_1$ depends on the size of the set $\Omega$. A first step in that direction was made by Poincar\'e himself in the Euclidean case.

\begin{theorem}
If $\Omega \subset \mathbb{R}^n$ is a bounded open convex set then for a smooth $f : \bar{ \Omega} \to \mathbb{R}$ with $\int_\Omega f(x) dx=0$, 
\[
\frac{C_n}{\mathbf{diam}(\Omega)^2} \int_\Omega f^2 (x)dx \le  \int_\Omega \| \nabla f (x) \|^2 dx.
\]
where $C_n$ is a constant depending on $n$ only.
\end{theorem}
\begin{proof}
The argument of Poincar\'e is beautifully simple. 
\begin{align*}
\frac{1}{\mu( \Omega )}  \int_\Omega f^2 (x)dx  &=\frac{1}{2} \frac{1}{\mu( \Omega )}  \int_\Omega f^2 (x)dx+\frac{1}{2} \frac{1}{\mu( \Omega )}  \int_\Omega f^2 (y)dy  \\
  & =\frac{1}{2} \frac{1}{\mu( \Omega )^2} \int_\Omega \int_\Omega (f(x)-f(y))^2 dx dy.
\end{align*}
We now have
\[
f(x)-f(y)=\int_0^1 (x-y)\cdot \nabla f(tx +(1-t)y) dt,
\]
which implies
\[
 (f(x)-f(y))^2 \le \mathbf{diam}(\Omega)^2\int_0^1 \| \nabla f \|^2 (tx +(1-t)y) dt.
\]
By a simple change of variables, we see that
\begin{align*}
 \int_\Omega \int_\Omega  \| \nabla f \|^2 (tx +(1-t)y) dxdy &=\frac{1}{t^n}  \int_\Omega \int_{t \Omega+(1-t)y}   \| \nabla f \|^2 (u) dudy \\
  &=\frac{1}{t^n}  \int_\Omega \int_{ \Omega}  \mathbf{1}_{t \Omega+(1-t)y} (u) \| \nabla f \|^2 (u) dudy.
\end{align*}
Now, we compute
\[
\int_{ \Omega}  \mathbf{1}_{t \Omega+(1-t)y} (u) dy=\mu \left( \Omega \cap \frac{1}{1-t} (u-t\Omega) \right)\le \min \left( 1, \frac{t^n}{(1-t)^n} \right) \mu(\Omega).
\]
As a consequence we obtain
\[
\frac{1}{\mu( \Omega )}  \int_\Omega f^2 (x)dx \le \frac{ \mathbf{diam}(\Omega)^2}{2\mu( \Omega )}\int_0^1  \min \left( 1, \frac{t^n}{(1-t)^n} \right)\frac{dt}{t^n}  \int_\Omega \| \nabla f (x) \|^2 dx
\]
\end{proof}

It is known (Payne-Weinberger) that the optimal $C_n$ is $\pi^2$. 

In this Section, we extend the above inequality to the case of Riemannian manifolds with non negative Ricci curvature. The key point is a lower bound on the Neumann heat kernel of $\Omega$. From now on we assume that $\mathbf{Ric} \ge 0$ and consider an open set in $\mathbb{M}$ that has a smooth and  convex boundary in the sense the second fundamental form of $\partial \Omega$ is non negative. Due to the convexity of the boundary, all the results we obtained so far may be extended to the Neumann semigroup (see \cite{BV}). In particular, we have the following lower bound on the Neumann heat kernel:

\begin{theorem}
Let $p^N (x,y,t)$ be the Neumann heat kernel of $\Omega$. There exists a constant $C$ depending only on the dimension of $\mathbb{M}$ such that for every $t >0$, $x,y \in \mathbb{M}$,
\[
p^N (x,y,t) \ge \frac{C}{\mu( B(x,\sqrt{t}))} \exp \left(-\frac{d(x,y)^2}{3t}\right).
\]
\end{theorem}

As we shall see, this implies the following Poincar\'e inequality:

\begin{theorem}
 For a smooth $f : \bar{ \Omega} \to \mathbb{R}$ with $\int_\Omega f  d\mu=0$, 
\[
\frac{C_n}{\mathbf{diam}(\Omega)^2} \int_\Omega f^2 d\mu \le  \int_\Omega \Gamma(f) d\mu.
\]
where $C$ is a constant depending on the dimension of $\mathbb{M}$ only.
\end{theorem}

\begin{proof}
We denote by $R$ the diameter of $\Omega$. From the previous lower bound on the Neumann kernel of $\Omega$, we have
\[
p^N (x,y,R^2) \ge \frac{C}{\mu( \Omega )},
\]
where $C$ only depends on $n$. Denote now by $P_t^N$ the Neumann semigroup. We have for $f \in \mathcal{D}^\infty$
\begin{align*}
P^N_{R^2} (f^2)-(P^N_{R^2} f)^2 =\int_0^{R^2} \frac{d}{dt}  P^N_{t} ((P^N_{R^2 -t} f)^2) dt.
\end{align*}
By integrating over $\Omega$, we find then,
\begin{align*}
\int_\Omega P^N_{R^2} (f^2)-(P^N_{R^2} f)^2 d\mu &  =-\int_0^{R^2} \int_\Omega  \frac{d}{dt}  (P^N_{t} f)^2 d\mu dt \\
 &=2 \int_0^{R^2} \int_\Omega  \Gamma( P^N_{t} f, P^N_{t} f ) d\mu dt  \\
 &\le 2 R^2 \int_\Omega \Gamma(f) d\mu.
\end{align*}
But on the other hand, we have
\begin{align*}
P^N_{R^2} (f^2)(x)-(P^N_{R^2} f)^2(x) &=P^N_{R^2}\left[ \left( f -(P^N_{R^2} f)(x) \right)^2 \right](x) \\
 & \ge \frac{C}{\mu(\Omega)} \int_\Omega (f(y)- (P^N_{R^2} f)(x) )^2 d\mu(y)
\end{align*}
which gives 
\[
\int_\Omega P^N_{R^2} (f^2)-(P^N_{R^2} f)^2 d\mu \ge C \int_\Omega \left( f(x) -\frac{1}{\mu(\Omega)}  \int_\Omega f d\mu \right)^2 d\mu(x)
\]
The  proof is complete.
\end{proof}

In applications, it is often interesting to have a scale invariant Poincar\'e inequality on balls. If the manifold $\mathbb{M}$ has conjugate points, the geodesic spheres may not be convex and thus the previous argument does not work. However the following result still holds true:

\begin{theorem}
There exists a constant $C_n>0$ depending only on the dimension of $\mathbb{M}$ such that  for every  $r>0$ and every smooth $f :B(x,r) \to \mathbb{R}$ with $\int_{B(x,r)} f  d\mu=0$, 
\[
\frac{C_n}{r^2} \int_{B(x,r)}  f^2 d\mu \le  \int_{B(x,r)} \Gamma(f) d\mu.
\]
\end{theorem}

We only sketch the argument. By using the global  lower bound 
\[
p (x,y,t) \ge \frac{C}{\mu( B(x,\sqrt{t}))} \exp \left(-\frac{d(x,y)^2}{3t}\right),
\]
for the heat kernel, it is possible to prove a lower bound for the Neumman heat kernel on the ball $B(x_0,r)$: For $x,y \in B(x_0.r/2)$,
\[
p^N (x,y,r^2) \ge \frac{C}{\mu( B(x_0,r))},
\]
Arguing as before, we get
\[
\frac{C_n}{r^2} \int_{B(x_0,{r/2})}  f^2 d\mu \le  \int_{B(x,r)} \Gamma(f) d\mu.
\]
and show then that the integral on the left hand side can be taken on $B(x_0,{r})$ by using a Whitney's type covering argument.

\subsection{Sobolev inequality and volume growth}

In this Section, we show how Sobolev inequalities on a Riemannian manifold are related to the volume growth of metric balls. The link between  the Hardy-Littlewood-Sobolev theory and  heat kernel upper bounds  is due to Varopoulos,

Let $(\mathbb{M},g)$ be a complete Riemannian manifold and let $L$ be the Laplace-Beltrami operator of $\mathbb{M}$. As usual, we denote by $P_t$ the semigroup generated by $P_t$ and we assume $P_t 1=1$.

We have the following so-called maximal ergodic lemma, which was first proved by Stein. We give here the probabilistic proof since it comes with a nice constant.

\begin{lemma}(Stein's maximal ergodic theorem) Let $p>1$. For $ f \in L^p(\M,\mu)$, denote $f^*(x)=\sup_{t \ge 0} |P_t f(x)|$. We have
\[
\| f^* \|_{L^p_\mu(\mathbb{M})} \le \frac{p}{p-1} \| f \|_{L^p_\mu(\mathbb{M})}.
\]
\end{lemma}

\begin{proof}
For $x \in \mathbb{M}$, we denote by $(X_t^x)_{t \ge 0}$ the Markov process with generator $L$ and started at $x$. We fix $T>0$. By construction, for $t \le T$, we have,
\[
P_{T-t}f (X_T^x) =\mathbb{E} \left( f (X_{2T-t}^x) | X_T^x \right),
\]
and thus 
\[
P_{2(T-t)}f (X_T^x) =\mathbb{E} \left( (P_{T-t} f) (X_{2T-t}^x) | X_T^x \right).
\]
As a consequence, we obtain
\[
\sup_{0 \le t \le T} | P_{2(T-t)}f (X_T^x) | \le \mathbb{E} \left(\sup_{0 \le t \le T} | (P_{T-t} f) (X_{2T-t}^x) | \mid X_T^x\right) .
\]
Jensen's inequality yieelds then
\[
\sup_{0 \le t \le T} | P_{2(T-t)}f (X_T^x) |^p  \le \mathbb{E} \left(\sup_{0 \le t \le T} | (P_{T-t} f) (X_{2T-t}^x) |^p \mid X_T^x\right). 
\]
We deduce
\[
\mathbb{E} \left( \sup_{0 \le t \le T} | P_{2(T-t)}f (X_T^x) |^p \right)  \le \mathbb{E} \left(\sup_{0 \le t \le T} | (P_{T-t} f) (X_{2T-t}^x) |^p \right). 
\]
Integrating the inequality with respect to the Riemannian measure $\mu$, we obtain
\[
\left\|  \sup_{0 \le t \le T} | P_{2(T-t)}f  | \right\|_p   \le \left( \int_\mathbb{M} \mathbb{E} \left(\sup_{0 \le t \le T} | (P_{T-t} f) (X_{2T-t}^x) |^p \right)d\mu(x)\right)^{1/p}. 
\]
By reversibility, we get then
\[
\left\|  \sup_{0 \le t \le T} | P_{2(T-t)}f  | \right\|_p   \le \left( \int_\mathbb{M} \mathbb{E} \left(\sup_{0 \le t \le T} | (P_{T-t} f) (X_t^x) |^p \right)d\mu(x)\right)^{1/p}. 
\]
We now observe that the process $(P_{T-t} f) (X_t^x)$ is martingale and thus Doob's maximal inequality gives
\[
 \mathbb{E} \left(\sup_{0 \le t \le T} | (P_{T-t} f) (X_t^x) |^p \right)^{1/p} \le \frac{p}{p-1}  \mathbb{E} \left( | f(X_T^x)|^p \right)^{1/p}.
\]
The proof is complete.
\end{proof}

We now turn to the theorem by Varopoulos.

\begin{theorem} Let $n>0$, $0<\alpha <n$, and $1 <p<\frac{n}{\alpha}$. If there exists $C>0$ such that for every $t>0$, $x,y \in \mathbb{M}$,
\[
p(x,y,t) \le \frac{C}{t^{n/2}},
\]
then for every $f \in L^p_\mu(\mathbb{M})$, 
\[
\| (-L)^{-\alpha/2} f \|_{\frac{np}{n-p\alpha}} \le \left( \frac{p}{p-1} \right)^{1-\alpha/n} \frac{ 2n C^{\alpha / n}}{ \alpha (n-p\alpha) \Gamma(\alpha /2)} \|f \|_p
\]
\end{theorem}

\begin{proof}
We first observe that the bound 
\[
p(x,y,t) \le \frac{C}{t^{n/2}},
\]
implies that $|P_t f(x)| \le  \frac{C^{1/p}}{t^{n/2p}} \| f \|_p$.
Denote $I_\alpha f (x)=(-L)^{-\alpha/2} f (x)$. We have
\[
I_\alpha f (x)=\frac{1}{\Gamma(\alpha /2) } \int_0^{+\infty} t^{\alpha /2 -1 }P_t f (x) dt
\]
Pick $\delta >0$, to be later chosen, and split the integral in two parts:
\[
I_\alpha f (x)=J_\alpha f(x) +K_\alpha f (x),
\]
where $J_\alpha f (x)=\frac{1}{\Gamma(\alpha /2) } \int_0^{\delta} t^{\alpha /2 -1 }P_t f (x) dt$ and $K_\alpha f (x)=\frac{1}{\Gamma(\alpha /2) } \int_\delta^{+\infty} t^{\alpha /2 -1 }P_t f (x) dt$. We have 
\[
| J_\alpha f (x) | \le \frac{1}{\Gamma(\alpha /2) } \int_0^{+\infty} t^{\alpha /2 -1 }dt | f^* (x) | =\frac{2}{\alpha \Gamma(\alpha /2) } \delta^{\alpha /2}  | f^* (x) |.
\]
On the other hand,
\begin{align*}
| K_\alpha f(x)| & \le \frac{1}{\Gamma(\alpha /2) } \int_\delta^{+\infty} t^{\alpha /2 -1 } | P_t f  (x)| dt \\
 & \le  \frac{C^{1/p}}{\Gamma(\alpha /2) } \int_\delta^{+\infty} t^{\frac{\alpha} {2}-\frac{n}{2p} -1 } dt \| f \|_p \\
 &\le   \frac{C^{1/p}}{\Gamma(\alpha /2) } \frac{1}{-\frac{\alpha} {2}+\frac{n}{2p} }\delta^{\frac{\alpha} {2}-\frac{n}{2p} }  \| f \|_p .
\end{align*}
We deduce
\[
| I_\alpha f (x) | \le \frac{2}{\alpha \Gamma(\alpha /2) } \delta^{\alpha /2}  | f^* (x) |+ \frac{C^{1/p}}{\Gamma(\alpha /2) } \frac{1}{-\frac{\alpha} {2}+\frac{n}{2p} }\delta^{\frac{\alpha} {2}-\frac{n}{2p} }  \| f \|_p.
\]
Optimizing the right hand side of the latter inequality with respect to $\delta$ yields
\[
| I_\alpha f (x) |\le \frac{ 2n C^{\alpha / n}}{ \alpha (n-p\alpha) \Gamma(\alpha /2)} \|f \|^{\alpha p /n}_p |f^*(x)|^{1-p\alpha/n}.
\]
The proof is then completed by using Stein's maximal ergodic theorem.
\end{proof}

A special case, of particular interest, is when $\alpha =1$ and $p=2$. We get in that case the following Sobolev inequality:
\begin{theorem}
Let $n>2$.  If there exists $C>0$ such that for every $t>0$, $x,y \in \mathbb{M}$,
\[
p(x,y,t) \le \frac{C}{t^{n/2}},
\]
then for every $f \in C^\infty_0(\mathbb{M})$,
\[
\|  f \|_{\frac{2n}{n-2}} \le 2^{1-1/n} \frac{ 2n C^{1 / n}}{  (n-2) \sqrt{\pi}} \| \sqrt{\Gamma(f)} \|_2
\]
\end{theorem}

We mention that the constant in the above Sobolev inequality is not sharp. 

Combining the above with the Li-Yau upper bound for the heat kernel, we deduce the following theorem:

\begin{theorem}
Assume that $\mathbf{Ric} \ge 0$ and that there exists a constant $C >0$ such that for every $x \in \mathbb{M}$ and $r \ge 0$, $\mu (B(x,r)) \ge C r^n$, then there exists a constant $C'=C'(n)>0$ such that for every  $f \in C^\infty_0(\mathbb{M})$,
\[
\|  f \|_{\frac{2n}{n-2}} \le C' \| \sqrt{\Gamma(f)} \|_2.
\]
\end{theorem}

In many situations, heat kernel upper bounds with a polynomial decay are only available in small times the following result is thus useful:

\begin{theorem} Let $n>0$, $0<\alpha <n$, and $1 <p<\frac{n}{\alpha}$. If there exists $C>0$ such that for every $0<t \le 1$, $x,y \in \mathbb{M}$,
\[
p(x,y,t) \le \frac{C}{t^{n/2}},
\]
then, there is constant $C'$ such that for every $f \in L^p_\mu(\mathbb{M})$, 
\[
\| (-L+1)^{-\alpha/2} f \|_{\frac{np}{n-p\alpha}} \le C' \|f \|_p.
\]
\end{theorem}

\begin{proof}
We apply the Varopoulos theorem to the semigroup $Q_t=e^{-t} P_t$. Details are let to the reader.
\end{proof}

The following corollary shall be later used:

\begin{corollary}
Let $n>2$.  If there exists $C>0$ such that for every $0<t \le 1$, $x,y \in \mathbb{M}$,
\[
p(x,y,t) \le \frac{C}{t^{n/2}},
\]
then there is constant $C'$ such that for every $f \in C^\infty_0(\mathbb{M})$,
\[
\|  f \|_{\frac{2n}{n-2}} \le C' \left(  \| \sqrt{\Gamma(f)} \|_2 + \| f \|_2 \right).
\]
\end{corollary}

\subsection{Isoperimetric inequality and volume growth}

In this section, we study in further details the connection between volume growth of metric balls, heat kernel upper bounds and the $L^1$ Sobolev inequality. As we shall see, on a manifold with non negative Ricci curvature, all these properties are equivalent one to each other and equivalent to the isoperimetric inequality as well. We start with some preliminaries about geometric measure theory  on Riemannian manifolds.

Let $(\mathbb{M},g)$ be a complete and non compact Riemannian manifold.

In what follows, given an open set $\Omega \subset \mathbb{M}$ we will indicate with $\mathcal F(\Omega)$ the set of $C^1$ vector fields $V$'s, on $\Omega$ such that $\| V \|_\infty \le 1$.

Given a function $f\in L^1_{loc}(\Omega)$ we define the total variation of $f$ in $\Omega$ as
\[
\text{Var} (f;\Omega) = \underset{\phi\in \mathcal{F}(\Omega)}{\sup} \int_\Om f \mathbf{div} \phi d\mu.
\]
The space \[ BV (\Omega) = \{f\in L^1(\Omega)\mid \text{Var}(f;\Omega)<\infty\},
\]
endowed with the norm
\[
||f||_{BV(\Omega)} = ||f||_{L^1(\bM)} + \text{Var} (f;\Omega),
\]
is a Banach space. It is well-known that $W^{1,1}(\Omega) = \{f\in L^1(\Omega)\mid \| \nabla f \| \in L^1(\Omega )\}$ is a strict subspace of
$BV(\Omega)$. It is important to note that when $f\in W^{1,1}(\Omega)$, then $f\in BV(\Omega)$, and one has in fact
\[
\text{Var}(f;\Omega) = ||\sqrt{\Gamma(f)}||_{L^1(\Omega)}.
\]
Given a measurable set $E\subset \mathbb{M}$ we say that it has finite perimeter in $\Omega $ if $\mathbf 1_E\in BV(\Omega)$. In such case the horizontal perimeter of $E$ relative to $\Omega$ is by
definition
\[
P(E;\Omega) = \text{Var}(\mathbf 1_E;\Omega).
\]
We say that a measurable set $E\subset \mathbb{M}$ is a Caccioppoli set if $P(E;\Omega)<\infty$ for any $\Omega \subset \mathbb{M}$. For instance, $O$ is an open relatively compact set in $\mathbb{M}$ whose boundary $E$ is $n-1$ dimensional sub manifold of $\mathbb{M}$, then it is a Caccioppoli set and $P(E;\mathbb{M})=\mu_{n-1} (E)$ where $\mu_{n-1}$ is the Riemannian measure on $E$.
We will need the following approximation result.

\begin{proposition} Let $f\in BV(\Omega)$, then there exists a sequence $\{f_n\}_{n\in \mathbb N}$ of functions in $C^\infty(\Omega)$ such that:
\begin{itemize}
\item[(i)] $||f_n - f||_{L^1(\Omega)} \to 0$;
\item[(ii)] $\int_\Omega \sqrt{\Gamma(f_n)} d\mu \to
\text{Var}(f;\Omega)$.
\end{itemize}
If $\Omega = \mathbb{M}$, then the sequence $\{f_n\}_{n\in \mathbb N}$ can be taken in $C^\infty_0(\mathbb{M})$.
\end{proposition}

 Our main result of the section  is the following:

\begin{theorem} Let $n>1$. Let us assume that $\mathbf{Ric} \ge 0$. then the following assertions  are equivalent:
\begin{itemize}
\item[(1)] There exists a constant $C_1 >0$ such that for every $x \in \mathbb{M}$, $r \ge 0$,
\[
\mu (B(x,r)) \ge C_1 r^n.
\]
\item[(2)] There exists a constant $C_2>0$ such that for $x \in \mathbb{M}$, $t>0$,
\[
p(x,x,t) \le \frac{C_2}{t^{\frac{n}{2}}}.
\]
\item[(3)] There exists a constant $C_3 >0$ such that for every Caccioppoli set $E\subset \mathbb{M}$ one has
\[
\mu(E)^{\frac{n-1}{n}} \le C_3 P(E;\mathbb{M}).
\]
\item[(4)] With the same constant $C_3>0$ as in (3), for every $f \in BV(\mathbb{M})$ one has
\[
\left( \int_\mathbb{M} |f|^{\frac{n}{n-1}} d\mu\right)^{\frac{n-1}{n}}  \le C_3 \emph{Var}(f;\mathbb{M}).
\]
\end{itemize}
\end{theorem}

\begin{proof}
In the proof, we denote by $d$ the dimension of $\M$.
That (1) $\rightarrow$ (2) follows immediately from the Li-Yau upper Gaussian bound.

The proof that (2) $\rightarrow$ (3) is not straightforward, it relies on the Li-Yau inequality. Let $f \in C_0(\mathbb{M})$ with $f\ge 0$. By Li-Yau inequality, we obtain
\begin{equation}\label{LY}
\Gamma (P_t f) - P_tf  \frac{\partial P_tf}{\partial t} \le \frac{d}{2t}(P_tf)^2.
\end{equation}
This gives in particular, ,
\begin{equation}\label{LYpt}
\left(\frac{\partial P_tf}{\partial t}\right)^- \le \frac{d}{2t} P_tf,
\end{equation}
where we have denoted $a^+ = \sup\{a,0\}$, $a^- = \sup\{-a,0\}$. Since $\int_\mathbb{M} \frac{\partial P_tf}{\partial t} d\mu=0$, we deduce
\[
||\frac{\partial P_tf}{\partial t}||_{L^1(\mathbb{M})} \le \frac{d}{t} ||f||_{L^1(\bM)},\ \ \ \ t>0.
\]
By duality, we deduce that for every $f \in C^\infty_0(\mathbb{M})$, $f\ge 0$,
\[
 \|\frac{\partial P_tf}{\partial t} \|_{L^\infty(\mathbb{M})} \le \frac{d}{t} \| f\|_{L^\infty(\bM)}.
 \]
Once we have this crucial information we can return to \eqref{LY} and infer
\begin{align*}
\Gamma (P_t f)  \le \frac{1}{t} \frac{3d}{2} \| f\|^2_{L^\infty(\mathbb{M})},\ \ \ \ t>0.
\end{align*}
Thus,
\[
\| \sqrt{\Gamma (P_t f)} \|_{L^\infty(\mathbb{M})} \le \sqrt{\frac{3d}{2t} }\| f\|_{L^\infty(\mathbb{M})}.
\]
Applying this inequality to $g \in C_0^\infty(\mathbb{M})$, with $g\ge 0$ and $||g||_{L^\infty(\mathbb{M})}\le 1$, if $f \in C_0^1(\mathbb{M})$ we have
\begin{align*}
\int_\mathbb{M}g(f-P_tf) d\mu & = \int_0^t \int_\mathbb{M} g \frac{\partial P_sf}{\partial s}
d\mu ds = \int_0^t \int_\mathbb{M} g L P_sf d\mu ds =  \int_0^t \int_\mathbb{M} L
g P_sf d\mu ds
\\
& = \int_0^t \int_\mathbb{M} P_sLg f d\mu ds = \int_0^t \int_\mathbb{M} L P_sg f
d\mu ds  = - \int_0^t \int_\mathbb{M} \Gamma(P_sg,f) d\mu ds
\\
& \le \int_0^t  \| \sqrt{\Gamma(P_sg)} \|_{L^\infty(\mathbb{M})}\int_\mathbb{M}
\sqrt{\Gamma(f)} d\mu ds  \le \sqrt{6d} \ \sqrt{t}
\int_\mathbb{M} \sqrt{\Gamma(f)} d\mu.
\end{align*}
We thus obtain the following basic inequality: for $f \in
C_0^1(\bM)$,
\begin{align}\label{PoincareP_t}
\|P_tf - f\|_{L^1(\mathbb{M})} \le \sqrt{6d}\ \ \sqrt{t}\ \| \sqrt{\Gamma(f)} \|_{L^1(\mathbb{M})},\ \ \ t>0.
\end{align}
Suppose now that $E\subset \mathbb{M}$ is a bounded Caccioppoli set. But then, $\mathbf 1_E\in BV(\Omega)$, for any bounded open set
$\Omega \supset E$. It is easy to see that $\text{Var}(\mathbf 1_E;\Omega) = \text{Var}_\Ho(\mathbf 1_E;\mathbb{M})$, and therefore $\mathbf 1_E\in
BV(\mathbb{M})$. There exists a sequence $\{f_n\}_{n\in \mathbb N}$ in $C^\infty_0(\mathbb{M})$ satisfying (i) and
(ii) above. Applying \eqref{PoincareP_t} to $f_n$ we obtain \[ \|P_tf_n -
f_n\|_{L^1(\bM)} \le \sqrt{6d} \ \sqrt{t}\ \| \sqrt{\Gamma(f_n)}
\|_{L^1(\mathbb{M})} = \sqrt{6d} \ \sqrt{t}\ Var_\Ho(f_n,\mathbb{M}),\ \ \ n\in
\mathbb N.
\]
Letting $n\to \infty$ in this inequality, we conclude
\[ \|P_t \mathbf 1_E -
\mathbf 1_E\|_{L^1(\bM)} \le \sqrt{6d} \ \sqrt{t}\ Var_\Ho(\mathbf 1_E,\mathbb{M}) = \sqrt{6d} \ \sqrt{t}\ P(E;\mathbb{M}),\ \
\ \ t>0.
\]
Observe now that, using $P_t 1 = 1$, we have
\[
||P_t \mathbf 1_E - \mathbf 1_E||_{L^1(\mathbb{M})}  = 2\left(\mu(E) - \int_E P_t \mathbf 1_E d\mu\right).
\]
On the other hand,
\[
\int_E  P_t \mathbf 1_E d\mu  = \int_\mathbb{M} \left(P_{t/2}\mathbf
1_\mathbb{M}\right)^2 d\mu.
\]
We thus obtain
\[
||P_t \mathbf 1_E - \mathbf 1_E||_{L^1(\mathbb{M})} = 2 \left(\mu(E) -
\int_\bM \left(P_{t/2}\mathbf 1_E\right)^2 d\mu\right).
\]
We now observe that the assumption (1) implies
\[
p(x,x,t)  \le \frac{C_4}{t^{n/2}},\ \ \ x\in \mathbb{M}, t>0.
\]
This gives
\begin{align*}
\int_\mathbb{M} (P_{t/2} \mathbf 1_E)^2 d\mu & \le \left(\int_E \left(\int_\mathbb{M} p(x,y,t/2)^2 d\mu(y)\right)^{\frac{1}{2}}d\mu(x)\right)^2 \\
& = \left(\int_E p(x,x,t)^{\frac{1}{2}}d\mu(x)\right)^2 \le \frac{C_4}{t^{n/2}} \mu(E)^2.
\end{align*}
Combining these equations we reach the conclusion
\[
\mu(E)  \le \frac{\sqrt{6d}}{2} \ \sqrt{t}\ P(E;\mathbb{M}) + \frac{C_4}{t^{n/2}} \mu(E)^2,\ \ \ \ t>0.
\]
Now the absolute minimum of the function $g(t) = A t^\alpha + B
t^{-\beta}$, $t>0$, where $A, B, \alpha, \beta>0$, is given by
\[
g_{\min} =
\left[\left(\frac{\alpha}{\beta}\right)^{\frac{\beta}{\alpha +
\beta}} + \left(\frac{\beta}{\alpha}\right)^{\frac{\alpha}{\alpha +
\beta}}\right] A^{\frac{\beta}{\alpha + \beta}}
B^{\frac{\alpha}{\alpha + \beta}}
\]
Applying this observation with $\alpha = \frac{1}{2}, \beta =
\frac{n}{2}$, we conclude
\[
\mu(E)^{\frac{n-1}{n}} \le C_3 P(E,\mathbb{M}).
\]
The fact that 3) implies 4) is classical geometric measure theory. It relies on the Federer co-area formula that we recall: For every $f,g \in C_0^\infty(\mathbb{M})$,
\[
\int_\mathbb{M} g \| \nabla f \| d\mu=\int_{-\infty}^{+\infty} \left( \int_{f(x)=t} g(x) d\mu_{n-1}(x) \right) dt.
\]
Let now $f \in C_0^\infty(\mathbb{M})$. We have
\[
f(x)=\int_0^{+\infty} \mathbf{1}_{f(x) > t} (t) dt.
\]
By using Minkowski inequality, we get then
\begin{align*}
\| f \|_{\frac{n}{n-1}} & \le \int_0^\infty \| \mathbf{1}_{f(\cdot) > t} \|_{\frac{n}{n-1}}dt \\
 & \le \int_0^\infty \mu ( f > t )^{\frac{n}{n-1}}dt \\
 &\le C_3 \int_0^\infty \mu_{n-1} ( f=t) dt =C_3 \int_\mathbb{M} \sqrt{\Gamma(f)} d\mu
\end{align*}

Finally, we show that $(4) \rightarrow (1)$. In what follows we let
$\nu = n/(n-1)$. Let $p,q\in (0,\infty)$ and $0<\theta\le 1$ be such
that
\[
\frac{1}{p} = \frac{\theta}{\nu} + \frac{1-\theta}{q}.
\]
H\"older inequality, combined with assumption (4), gives for any
$f\in Lip_d(\mathbb{M})$ with compact support
\[
||f||_{L^p(\mathbb{M})} \le ||f||^\theta_{L^{\nu}(\mathbb{M})}
||f||^{1-\theta}_{L^q(\mathbb{M})}\le \left(C_3
||\sqrt{\Gamma(f)}||_{L^1(\mathbb{M})}\right)^\theta
||f||^{1-\theta}_{L^q(\mathbb{M})}.
\]
For any $x\in \mathbb{M}$ and $r>0$ we now let $f(y) = (r-d(y,x))^+$.
Clearly such $f\in Lip_d(\bM)$ and supp$\ f = \overline B(x,r)$.
Since with this choice $||\sqrt{\Gamma(f)}||_{L^1(\mathbb{M})}^\theta \le \mu(B(x,r))^{\theta}$, the above inequality implies
\[
\frac{r}{2} \mu(B(x,\frac{r}{2})^{\frac{1}{p}} \le r^{1-\theta}
\left(C_3 \mu(B(x,r)\right)^\theta \mu(B(x,r))^{\frac{1-\theta}{q}},
\]
which, noting that $\frac{1-\theta}{q} + \theta = \frac{n+\theta
p}{pn}$, we can rewrite as follows
\[
\mu(B(x,r)) \ge \left(\frac{1}{2C_3^\theta}\right)^{pa}
\mu(B(x,\frac{r}{2}))^a r^{\theta p a},
\]
where we have let $a = \frac{n}{n+\theta p}$. Notice that $0<a<1$.
Iterating the latter inequality we find
\[
\mu(B(x,r)) \ge \left(\frac{1}{2C_3^\theta}\right)^{p\sum_{j=1}^k
a^j} r^{\theta p \sum_{j=1}^k a^j} 2^{-\theta p \sum_{j=1}^k
(j-1)a^j}\mu(B(x,\frac{r}{2^k}))^{a^k},\ \ \ k\in \mathbb N.
\]
From the doubling property for any $x\in \mathbb{M}$ there exist
constants $C(x), R(x)>0$ such that with $Q(x) = \log_2 C(x)$ one has
\[
\mu(B(x,tr)) \ge C(x)^{-1} t^{Q(x)} \mu(B(x,r)),\ \ \ 0\le t\le 1,
0<r\le R(x).
\]
This estimate implies that
\[
\underset{k\to \infty}{\liminf}\ \mu(B(x,\frac{r}{2^k}))^{a^k}\ge
1,\ \ \ x\in \bM, r>0.
\]
Since on the other hand $\sum_{j=1}^\infty a^j = \frac{n}{\theta
p}$, and $\sum_{j=1}^\infty (j-1) a^j = \frac{n^2}{\theta^2p^2}$, we
conclude that
\[
\mu(B(x,r)) \ge \left(2^{-\frac{1}{\theta}(1+\frac{n}{p})}
C_3^{-1}\right)^n r^n,\ \ \ x\in \mathbb{M}, r>0.
\]
This establishes (1), thus completing the proof.
\end{proof}

\subsection{Sharp Sobolev inequalities}

In this section we are interested in sharp Sobolev inequalities in positive curvature. Let $(\mathbb{M},g)$ be a complete and $n$-dimensional Riemannian manifold such that $\mathbf{Ricci} \ge \rho$ where $\rho >0$. We assume $n >2$.  As we already know, we have $\mu (\mathbb{M}) <+\infty$, but as we already stressed we do not want to use Bonnet-Myers theorem, since one of our goals will be to recover it by using heat kernel techniques. Without loss of generality, and to simplify the constants, we assume that $\mu(\mathbb{M}) =1$. Our goal is to prove the following sharp result:

\begin{theorem}
For every $1 \le p \le \frac{2n}{n-2}$ and $ f \in C^\infty_0(\mathbb{M})$,
\[
\frac{n \rho}{(n-1)(p-2)} \left( \left(  \int_{\mathbb{M}} | f |^p d\mu\right)^{2/p}-\int_\mathbb{M} f^2 d\mu \right) \le \int_\mathbb{M} \Gamma(f) d\mu.
\]
\end{theorem}
We observe that for $p=1$, the inequality becomes
\[
\frac{n \rho}{(n-1)}  \left( \int_\mathbb{M} f^2 d\mu -  \left( \int_{\mathbb{M}} | f | d\mu\right)^{2} \right) \le \int_\mathbb{M} \Gamma(f) d\mu.
\]
which is the Poincar\'e inequality with optimal Lichnerowicz constant. For $p=2$, we get the log-Sobolev inequality
\[
\frac{n \rho}{2(n-1)} \left( \int_\mathbb{M} f^2 \ln f^2 d\mu -\int_\mathbb{M} f^2 d\mu \ln\int_\mathbb{M} f^2 d\mu \right) \le \int_\mathbb{M} \Gamma(f) d\mu.
\]
We prove our Sobolev inequality in several steps.

\begin{lemma}
For every $ 1 \le  p \le \frac{2n}{n-2}$, there exists a constant $C_p >0$ such that for every  $ f \in C^\infty_0(\mathbb{M})$,
\[
C_p \left( \left(  \int_{\mathbb{M}} | f |^p d\mu\right)^{2/p}-\int_\mathbb{M} f^2 d\mu \right) \le \int_\mathbb{M} \Gamma(f) d\mu.
\]
\end{lemma} 

\begin{proof}
Using Jensen's inequality, it is enough to prove the result for $p= \frac{2n}{n-2}$. We already proved the following Li-Yau inequality:  For $f \in C^\infty_0(\mathbb{M})$, $f \neq 0$, $t  > 0$, and $x \in \bM$,
\[
\| \nabla \ln P_t f (x) \|^2 \le e^{-\frac{2\rho t}{3}}  \frac{  L P_t f (x)}{P_t f (x)} +\frac{n\rho}{3} \frac{e^{-\frac{4\rho t}{3}}}{ 1-e^{-\frac{2\rho t}{3}}}.
 \]
As a consequence we have
\[
 \frac{  L P_t f (x)}{P_t f (x)} \ge-  \frac{n\rho}{3} \frac{e^{-\frac{2\rho t}{3}}}{ 1-e^{-\frac{2\rho t}{3}}},
\]
which yields,
\[
\int_t^{+\infty} \partial_t \ln P_sf (x) ds  \ge-  \frac{n\rho}{3} \int_t^{+\infty} \frac{e^{-\frac{2\rho s}{3}}}{ 1-e^{-\frac{2\rho s}{3}}} ds.
\]
We obtain then
\[
 P_tf(x)  \le  \left( \frac{1}{1-e^{-\frac{2\rho t}{3}} }\right)^{n/2} \int_\mathbb{M} f d\mu.
\]
This of course implies the following upper bound on the heat kernel,
\[
p(x,y,t) \le \left( \frac{1}{1-e^{-\frac{2\rho t}{3}} }\right)^{n/2}.
\]
Using Varopoulos theorem, we deduce therefore that there is constant $C_p'$ such that for every $f \in C^\infty_0(\mathbb{M})$,
\[
\|  f \|^2_{p} \le C'_p \left(  \| \sqrt{\Gamma(f)} \|^2_2 + \| f \|^2_2 \right).
\]
We now use the following inequality which is easy to see:
\[
\left( \int_{\mathbb{M}} | f |^p d\mu \right)^{2/p} \le \left( \int_{\mathbb{M}}  f  d\mu \right)^{2}+\left( \int_{\mathbb{M}} \left| f -\int_{\mathbb{M}}  f  d\mu \right|^p d\mu \right)^{2/p} 
\]
This yields
\[
\left( \int_{\mathbb{M}} | f |^p d\mu \right)^{2/p} \le \left( \int_{\mathbb{M}}  f  d\mu \right)^{2} +(p-1)C'_p  \left(  \| \sqrt{\Gamma(f)} \|^2_2 + \left\| f-\int_{\mathbb{M}}  f  d\mu \right\|^2_2 \right)
\]
We can now bound
\[
\left\| f-\int_{\mathbb{M}}  f  d\mu \right\|^2_2 \le \frac{1}{\lambda_1} \int_\mathbb{M} \Gamma(f) d\mu,
\]
using the Poincar\'e inequality.
\end{proof}
Our proof follows now an argument due to Bakry and largely follows  the presentation by Ledoux \cite{Led} to which we refer for the details. 
We now want to prove that the optimal $C_p$ in the previous inequality is given by $C_p=\frac{n \rho}{(n-1)(p-2)} $. We assume $p>2$ and consider  the functional
\[
\frac{\left(  \int_{\mathbb{M}} | f |^p d\mu\right)^{2/p}-\int_\mathbb{M} f^2 d\mu}{ \int_\mathbb{M} \Gamma(f) d\mu}.
\]
Classical non linear variational principles on the functional provide then a positive non trivial solution of the equation
\[
C_p(f^{p-1}-f)=-Lf.
\]
We point out that existence and smoothness results for this non linear pde are non a  trivial issue, we refer to the comments by Ledoux in \cite{Led} and the references therein. Set $f=u^r$ where $r$ is a constant to be later chosen. By the chain rule for diffusion operators, we get
\[
C_p ( u^{r(p-1)}-u^r)=-ru^{r-1} -r(r-1)u^{r-2} \Gamma(u).
\]
Multiplying by $u^{-r} \Gamma(u)$ and integrating yields
\[
C_p \left( \int_{\mathbb{M}} u^{r(p-2)} \Gamma(u) d\mu -\int_{\mathbb{M}} \Gamma (u) d\mu \right)=-r \int_{\mathbb{M}} \frac{Lu}{u} \Gamma(u) d\mu -r (r-1) \int_{\mathbb{M}} \frac{\Gamma(u)^2}{u^2} d\mu.
\]
Now, integrating by parts,
\[
\int_{\mathbb{M}} u^{r(p-2)} \Gamma(u) d\mu=-\frac{1}{r(p-2)+1}  \int_{\mathbb{M}} u^{r(p-2)+1} Lu d\mu.
\]
On the other hand, multiplying
\[
C_p ( u^{r(p-1)}-u^r)=-ru^{r-1} -r(r-1)u^{r-2} \Gamma(u),
\]
by $u^{1-r} Lu$ and integrating with respect to $\mu$ yields
\[
C_p \left( \int_{\mathbb{M}} u^{r(p-2)+1} Lu d\mu - \int_{\mathbb{M}} uLu d\mu \right)=-r \int_{\mathbb{M}} (Lu)^2 d\mu -r(r-1)  \int_{\mathbb{M}} \frac{Lu}{u} \Gamma(u) d\mu. 
\]
Combining the previous computations gives
\[
C_p \left( r(p-2) +1\right) \int_{\mathbb{M}} u^{r(p-2)} \Gamma(u) d\mu = r  \int_{\mathbb{M}} (Lu)^2 d\mu +r(r-1)  \int_{\mathbb{M}} \frac{Lu}{u} \Gamma(u) d\mu +C_p  \int_{\mathbb{M}} \Gamma(u) d\mu
\]
Hence, we have
\[
C_p (p-2)  \int_{\mathbb{M}} \Gamma(u) d\mu= \int_{\mathbb{M}} (Lu)^2 d\mu + \int_{\mathbb{M}} \frac{Lu}{u} \Gamma(u) d\mu +(r-1)  \left( r(p-2) +1\right)  \int_{\mathbb{M}} \frac{\Gamma(u)^2}{u^2} d\mu
\]
We have from Bochner's inequality,
\[
\Gamma_2 (u^s)\ge \frac{1}{n} ( L u^s)^2 +\rho \Gamma(u^s).
\]
Once again, $s$ is a parameter that will be later decided. Using the chain, to rewrite the previous inequality, leads after tedious computations to
\[
\Gamma_2(u) +(s-1) \frac{1}{u} \Gamma ( u , \Gamma(u)) +(s-1)^2 \frac{\Gamma(u)^2}{u^2} \ge \rho \Gamma(u) +\frac{1}{n} (Lu)^2 +\frac{2}{n} (s-1) \frac{1}{u} Lu \Gamma (u) +\frac{1}{n} (s-1)^2 \frac{1}{u^2} \Gamma(u)^2.
\]
After integration and integration by parts, we see that
\[
\rho \int_{\mathbb{M}} \Gamma(u) d\mu \le \left( 1 -\frac{1}{n} \right) \int_{\mathbb{M}} (Lu)^2 d\mu-s' \left( 1+\frac{2}{n} \right) \int_{\mathbb{M}}  \frac{1}{u} Lu \Gamma (u) d\mu+s'\left( 1+s'\left( 1-\frac{1}{n} \right) \right) \int_{\mathbb{M}} \frac{\Gamma(u)^2}{u^2} d\mu,
\]
where $s'=s-1$. Combining the previous inequalities we can eliminate the term $ \int_{\mathbb{M}}  \frac{1}{u} Lu \Gamma (u) d\mu$. Chosing
\[
\frac{s'}{r}=(p-1) \frac{n-1}{n+2},
\]
we see that the coefficient in front of $ \int_{\mathbb{M}} (Lu)^2 d\mu$ is zero and we are left with
\[
\left( C_p \frac{(p-2)(n-1)}{n} -\rho \right)\int_\mathbb{M} \Gamma(u) d\mu \ge K(s',r) \int_{\mathbb{M}} \frac{\Gamma(u)^2}{u^2} d\mu,
\]
for some constant $K(s',r)$ which is seen to be non-negative as soon as $2< p \le \frac{2n}{n-2}$. We conclude
\[
 C_p \frac{(p-2)(n-1)}{n} -\rho \ge 0.
\]

\subsection{The Sobolev inequality proof of the  Myer's diameter theorem}

It is a well-known result  that if $\mathbb{M}$ is a complete $n$-dimensional Riemannian manifold with $\mathbf{Ricci} \ge \rho$, for some $\rho >0$, then $\mathbb{M}$ has to be compact with diameter less than $\pi \sqrt{\frac{n-1}{\rho} }$. The proof of this fact can be found in any graduate book about Riemannian geometry and classically relies on the study of Jacobi fields. We propose here an alternative proof of the diameter theorem that relies on the sharp Sobolev inequality proved in the previous section. The beautiful argument goes back to Bakry and Ledoux. We only sketch the main arguments and refer the readers to the original article \cite{BL}. 

The theorem by Bakry and Ledoux is the following:

\begin{theorem}
Assume that for some $p>2$, we have the inequality,
\[
\| f \|_p^2 \le \| f \|_2^2 + A \int_\mathbb{M} \Gamma(f) d\mu, \quad f \in C_0^\infty(\mathbb{M}),
\]
then $\mathbb{M}$ is compact with diameter less than $\pi \frac{ \sqrt{2pA}}{p-2}$.
\end{theorem}

Combining this with the inequality
\[
\frac{n \rho}{(n-1)(p-2)} \left( \left(  \int_{\mathbb{M}} | f |^p d\mu\right)^{2/p}-\int_\mathbb{M} f^2 d\mu \right) \le \int_\mathbb{M} \Gamma(f) d\mu,
\]
that was proved in the previous section gives $\mathbf{diam} (\mathbb{M}) \le \pi \sqrt{\frac{2p}{p-2} } \sqrt{\frac{n-1}{n\rho}}$. When $n=2$ we conclude then by letting $p \to \infty$ and when $n >2$, we conclude by choosing $p=\frac{2n}{n-2}$.

\

By using a scaling argument it is easy to see that it is enough to prove that if for some $n>2$, 
\[
\| f \|_{\frac{2n}{n-2}}^2 \le \| f \|_2^2 + \frac{4}{n(n-2)} \int_\mathbb{M} \Gamma(f) d\mu,
\]
then $\mathbf{diam} (\mathbb{M}) \le \pi $.

The main idea is to apply the Sobolev inequality to the functions which are the extremals functions on the sphere. Such extremals are solutions of the fully non linear PDE
\[
f^{(n+2)/(n-2)} -f =- \frac{4}{n(n-2)} Lf
\]
and on the spheres the extremals are explicitly given by
\[
f=(1+\lambda \sin d)^{1-n/2}
\]
where $-1 < \lambda <1$ and $d$ is the distance to a fixed point. So, on our manifold $\mathbb{M}$, that satisfies the inequality
\[
\| f \|_{\frac{2n}{n-2}}^2 \le \| f \|_2^2 + \frac{4}{n(n-2)} \int_\mathbb{M} \Gamma(f) d\mu,
\]
we consider the functional
\[
F(\lambda) =\int_\mathbb{M} ( 1+\lambda \sin (f) )^{2-n} d\mu, \quad -1 < \lambda <1, 
\]
where $f$ is a function on $\mathbb{M}$ that satisfies $\| \Gamma(f) \|_\infty \le 1$. The first step is to prove a differential inequality on $F$. For $k>0$, we denote by $D_k$ the differential operator on $(-1,1)$ defined by
\[
D_k=\frac{1}{k} \lambda \frac{\partial}{\partial \lambda} +I.
\]

\begin{lemma}
Denoting $G=D_{n-1} F$, we have
\[
(D_{n-2} G)^{(n-2)/n} +\frac{n-2}{n} (1-\lambda^2) D_{n-2}G \le \left( 1+\frac{n-2}{n}\right)G.
\]
\end{lemma}

\begin{proof}
We denote $\alpha=\frac{n-2}{n}$ and $f_\lambda=(1+\lambda \sin f)^{1-n/2}$, $-1 < \lambda < 1$. By the chain-rule and the hypothesis that $\Gamma(f) \le 1$, we get
\[
\int_\mathbb{M} \Gamma(f_\lambda) d\mu \le \left( \frac{n}{2} -1 \right)^2 \int_\mathbb{M} (1+\lambda \sin f)^{-n} (1-\sin^2 f) d\mu.
\]
From the Sobolev inequality applied to $f_\lambda$, we thus have,
\[
\left( \int_\mathbb{M} ( 1+\lambda \sin (f) )^{-n} d\mu  \right)^\alpha \le \int_\mathbb{M} ( 1+\lambda \sin (f) )^{2-n} d\mu +\alpha \lambda^2 \int_\mathbb{M} (1+\lambda \sin f)^{-n} (1-\sin^2 f) d\mu.
\]
It is then an easy calculus exercise to deduce our claim.
\end{proof}

The next idea is then to use a comparison theorem to bound $F$ in terms of solutions of the equation
\[
(D_{n-2} H)^{(n-2)/n} +\frac{n-2}{n} (1-\lambda^2) D_{n-2}H \le \left( 1+\frac{n-2}{n}\right)H.
\]
Actually, such solutions are given by 
\[
H_c(\lambda)=\frac{1}{1+\alpha} U_c(\lambda) ^{\frac{2\alpha}{1-\alpha}} +\frac{\alpha}{1+\alpha}  (1-\lambda^2) U_c(\lambda)^{\frac{2}{1-\alpha}},
\]
where $c \in \mathbb{R}$, $\alpha=\frac{n-2}{n}$ and
\[
U_c(\lambda)=\frac{ c\lambda +\sqrt{c^2 \lambda^2 +(1-\lambda^2) }}{1-\lambda^2}.
\]
We have then the following comparison result:

\begin{lemma}
Let $G$ be such that
\[
(D_{n-2} G)^{(n-2)/n} +\frac{n-2}{n} (1-\lambda^2) D_{n-2}G \le \left( 1+\frac{n-2}{n}\right)G,
\]
and assume that $G(\lambda_0) < H_c (\lambda_0)$ for some $\lambda_0 \in [0,1)$. Then for every $\lambda_0 \le \lambda <1$,
\[
G(\lambda) \le H_c(\lambda).
\]
\end{lemma}

Using the previous lemma, we see (again we refer to the original article for the details) that $\int_{\mathbb{M}}  \sin f d\mu >0 $ implies that $\int_{\mathbb{M}}  (1+\sin f )^{n-1}d\mu <\infty  $ and $\int_{\mathbb{M}}  \sin f d\mu < 0 $ implies that $\int_{\mathbb{M}}  (1-\sin f )^{n-1}d\mu <\infty  $. Iterating this result on the basis of the Sobolev inequality again, we actually have
\[
\| (1 \pm \sin f)^{-1} \|_{\infty} < \infty,
\]
from which the conclusion easily follows.

\section{The heat semigroup on sub-Riemannian manifolds and its applications}

It turns out that many of methods presented in the last section can be extended to some sub-Riemannian manifolds, provided that a correct generalization of curvature-dimension inequality is used. This generalized curvature dimension condition was first introduced in \cite{BG} in the context of sub-Riemannian manifolds with transverse symetries. In \cite{BG}, it has been shown that this generalized curvature dimension condition implies a Li-Yau inequality and a Gaussian upper bound for the heat kernel. Then, in \cite{BBG}, it was proved that the curvature dimension inequality implies a lower bound for the heat kernel. Combining those results the volume doubling property and the 2-Poincar\'e inequality on balls were deduced. In this Section we sketch the results obtained in those two papers without entering into details. We refer to the survey \cite{B2} and the forthcoming book \cite{BG5} for further applications of the sub-Riemannian curvature dimension condition (like Bonnet-Myers theorems and eigenvalue estimates), see also \cite{GT1,GT2}.

\subsection{Framework} \label{ssec:foliation}

Hereafter in this section, $\bM$ will be a $C^\infty$ connected  manifold endowed with a smooth measure $\mu$ and a second-order diffusion operator $L$ on $\M$ with real coefficients,  locally subelliptic, satisfying $L1=0$ and 
\begin{equation*}
\int_\bM f L g d\mu=\int_\bM g Lf d\mu,\ \ \ \ \ \ \int_\bM f L f d\mu \le 0,
\end{equation*}
for every $f , g \in C^ \infty_0(\bM)$, where $C^ \infty_0(\bM)$ denotes the space of compactly supported functions. A distance $d$ is constructed as follows (see Section 1):
\begin{equation}\label{di}
d(x,y)=\sup \left\{ |f(x) -f(y) | \mid f \in  C^\infty(\bM) , \| \Gamma(f) \|_\infty \le 1 \right\},\ \ \  \ x,y \in \bM,
\end{equation}
where for a function $g$ on $\bM$ we have let $||g||_\infty = \underset{\bM}{\text{ess} \sup} |g|$. This distance will often be referred to as the subelliptic or sub-Riemannian distance. As before, the quadratic functional $\Gamma(f) = \Gamma(f,f)$, where
\begin{equation}\label{gamma}
\Gamma(f,g) =\frac{1}{2}(L(fg)-fLg-gLf), \quad f,g \in C^\infty(\bM),
\end{equation}
is known as \textit{le carr\'e du champ}.  
Given any point $x\in \M$ there exists an open set $x\in U\subset \M$ in which the operator $L$ can be written as 
\begin{equation}\label{locrep}
L = - \sum_{i=1}^m X^*_i X_i,
\end{equation}
where  the vector fields $X_i$ have Lipschitz continuous coefficients in $U$, and $X_i^*$ indicates the formal adjoint of $X_i$ in $L^2(\M,d\mu)$. We remark that such local representation of $L$ is not unique.

In addition to the differential form \eqref{gamma}, we assume that $\M$ be endowed with another smooth bilinear differential form, indicated with $\Gamma^Z$, satisfying for $f,g \in C^\infty(\M)$
\[
\Gamma^Z(fg,h) = f\Gamma^Z(g,h) + g \Gamma^Z(f,h),
\] 
and $\Gamma^Z(f) = \Gamma^Z(f,f) \ge 0$. We assume that given any point $x\in \M$ there exists an open set $x\in U\subset \M$ in which the operator $\Gamma^Z$ can be written as
\[
\Gamma^Z(f,g) = \sum_{i=1}^p (Z_i f) (Z_i g),
\] 
where  the vector fields $Z_i$ have Lipschitz continuous coefficients in $U$.

Given the first-order bilinear forms $\Gamma$ and $\Gamma^Z$ on $\bM$, we now introduce the following second-order differential forms:
\begin{equation}\label{gamma2}
\Gamma_{2}(f,g) = \frac{1}{2}\big[L\Gamma(f,g) - \Gamma(f,
Lg)-\Gamma (g,Lf)\big],
\end{equation}
\begin{equation}\label{gamma2Z}
\Gamma^Z_{2}(f,g) = \frac{1}{2}\big[L\Gamma^Z (f,g) - \Gamma^Z(f,
Lg)-\Gamma^Z (g,Lf)\big].
\end{equation}
As for $\Gamma$ and $\Gamma^Z$, we will use the notations  $\Gamma_2(f) = \Gamma_2(f,f)$, $\Gamma_2^Z(f) = \Gamma^Z_2(f,f)$.

We make the following assumptions:

\begin{itemize}
\item[(H.1)] There exists an increasing
sequence $h_k\in C^\infty_0(\bM)$   such that $h_k\nearrow 1$ on
$\bM$, and \[
||\Gamma (h_k)||_{\infty} +||\Gamma^Z (h_k)||_{\infty}  \to 0,\ \ \text{as} \ k\to \infty.
\]
\item[(H.2)]  
For any $f \in C^\infty(\bM)$ one has
\[
\Gamma(f, \Gamma^Z(f))=\Gamma^Z( f, \Gamma(f)).
\]
\item[(H.3)]  The heat semigroup generated by $L$, which will denoted $P_t$ throughout the paper, is stochastically complete that is, for $t \ge 0$, $P_t 1=1$ and   for every $f \in C_0^\infty(\bM)$ and $T \ge 0$, one has 
\[
\sup_{t \in [0,T]} \| \Gamma(P_t f)  \|_{ \infty}+\| \Gamma^Z(P_t f) \|_{ \infty} < +\infty.
\]
\item[(H.4)] Given any two points $x, y\in \M$, there exist a subunit curve joining them.
\item[(H.5)] The metric space $(\M,d)$ is complete. 
\end{itemize}

A large class of examples where all these assumptions are satisfied arises in the context of  totally geodesic Riemannian foliations and sub-Riemannian manifolds with transverse symmetries (see \cite{B2, BG, BKW} for a detailed proof of these assumptions).

\begin{definition}(see \cite{BG})\label{D:cdi}
We shall say that  $\mathbb{M}$ satisfies the generalized curvature-dimension inequality CD($\rho_{1},\rho_{2},\kappa,n$) if there exist constants $\rho_{1}\in\mathbb{R}, \rho_{2}>0, \kappa\geq 0,$ and $n > 0$ such that the inequality
\begin{equation}
\label{sRCD}
\Gamma_{2}\left(f\right)+\nu\Gamma^{Z}_{2}\left(f\right)\geq \frac{1}{n}\left(Lf\right)^{2}+\left(\rho_{1}-\frac{\kappa}{\nu}\right)\Gamma\left(f\right)+\rho_{2}\Gamma^{Z}\left(f\right)
\end{equation}
holds for every $f\in C^{\infty}\left(\mathbb{M}\right)$ and every $\nu>0$.
\end{definition}

\subsection{Li-Yau inequality and volume doubling properties for the subelliptic distance}

Throughout the section we assume that $\mathbb{M}$ satisfies the generalized curvature-dimension inequality CD($\rho_{1},\rho_{2},\kappa,n$), and we show  how to obtain the Li-Yau estimate for the subelliptic operator $L$.

Henceforth, we will indicate $C_b^\infty(\mathbb M) = C^\infty(\M)\cap L^\infty(\M)$ and by $P_t$ the semigroup generated by $L$. A key lemma is the following.

\begin{lemma}\label{L:derivatives}
Let $f \in C^\infty_b(\mathbb{M})$, $f > 0$ and $T>0$, and consider the functions
\[
\phi_1 (x,t)=(P_{T-t} f) (x)\Gamma (\ln P_{T-t}f)(x),
\]
\[
\phi_2 (x,t)= (P_{T-t} f)(x) \Gamma^{Z} (\ln P_{T-t}f)(x),
\]
which are defined on $\M\times [0,T)$.  We have
\[
L \phi_1+\frac{\partial \phi_1}{\partial t} =2 (P_{T-t} f) \Gamma_2 (\ln P_{T-t}f)
\]
and
\[
L \phi_2+\frac{\partial \phi_2}{\partial t} =2 (P_{T-t} f) \Gamma_2^{Z} (\ln P_{T-t}f).
\]
\end{lemma}

\begin{proof}
This is direct computation without trick. Let us just point out that the formula
\[
L \phi_2+\frac{\partial \phi_2}{\partial t} =2 (P_{T-t} f) \Gamma_2^{Z} (\ln P_{T-t}f).
\]
uses the fact that $\Gamma(g , \Gamma^Z (g))=\Gamma^Z (g , \Gamma (g))$.
\end{proof}

We will need the following  lemma already used before.
\begin{lemma}\label{P:missing_key2}
Let $T>0$. Let $u,v: \mathbb{M}\times [0,T] \to \mathbb{R}$ be  smooth functions such that  for every $T>0$,  $\sup_{t \in [0,T]} \| u(\cdot,t)\|_\infty <\infty$, $\sup_{t \in [0,T]} \| v(\cdot,t)\|_\infty <\infty$; If the  inequality 
\[
L u+\frac{\partial u}{\partial t} \ge v
\]
holds on $\mathbb{M}\times [0,T]$, then we have
\[
P_T(u(\cdot,T))(x) \ge u(x,0) +\int_0^T P_s(v(\cdot,s))(x) ds.
\]

\end{lemma}
We now show how to prove the Li-Yau estimates for the semigroup $P_t$. 

\begin{theorem}\label{T:ge}
Let $\alpha >2$. For  $f \in C_0^\infty(\M)$, $f  \ge 0$, $f \neq 0$, the following inequality holds for $t>0$:
\begin{align*}
 & \Gamma (\ln P_t f) +\frac{2 \rho_2}{\alpha}  t \Gamma^Z (\ln P_t f) \\
  \le & \left(1+\frac{\alpha \kappa}{(\alpha-1)\rho_2}-\frac{2\rho_1}{\alpha} t\right)
\frac{L P_t f}{P_t f} +\frac{n\rho_1^2}{2\alpha} t-\frac{\rho_1 n}{2}\left(
1+\frac{\alpha \kappa}{(\alpha-1)\rho_2}\right) +\frac{n(\alpha-1)^2\left(
1+\frac{\alpha \kappa}{(\alpha-1)\rho_2}\right)^2}{8(\alpha-2)t}.
\end{align*}
\end{theorem}

\begin{proof}
We fix $T>0$ and consider two functions $a,b:[0,T] \to \mathbb{R}_{\ge 0}$ to be chosen later. Let $f \in C^\infty(\mathbb{M})$, $ f \ge 0$.
Consider the function
\[
\phi (x,t)=a(t)(P_{T-t} f) (x)\Gamma (\ln P_{T-t}f)(x)+b(t)(P_{T-t} f) (x) \Gamma^Z (\ln P_{T-t}f)(x).
\]
Applying Lemma \ref{L:derivatives} and the curvature-dimension inequality, we obtain
\begin{align*}
 & L \phi+\frac{\partial \phi}{\partial t} \\
=& a' (P_{T-t} f) \Gamma (\ln P_{T-t}f)+b' (P_{T-t} f) \Gamma^Z (\ln P_{T-t}f) +2a (P_{T-t} f) \Gamma_2 (\ln P_{T-t}f) \\
 & +2b (P_{T-t} f) \Gamma_2^Z (\ln P_{T-t}f) \\
\ge&  \left(a'+2\rho_1 a -2\kappa \frac{a^2}{b}\right)(P_{T-t} f) \Gamma (\ln P_{T-t}f)  +(b'+2\rho_2 a) (P_{T-t} f)  \Gamma^Z (\ln P_{T-t}f) \\
&+\frac{2a}{n}  (P_{T-t} f) (L(\ln P_{T-t} f))^2. 
\end{align*}
But, for any function $\gamma:[0,T]\to \R$
\[
(L(\ln P_{T-t} f))^2 \ge 2\gamma L(\ln P_{T-t}f) -\gamma^2,
\]
and from chain rule
\[
 L(\ln P_{T-t}f)=\frac{L P_{T-t}f}{P_{T-t}f} -\Gamma(\ln P_{T-t} f ).
 \]
Therefore, we obtain
 \begin{align*}
L \phi+\frac{\partial \phi}{\partial t}   \ge & \left(a'+2\rho_1 a -2\kappa \frac{a^2}{b}-\frac{4a\gamma}{n} \right) (P_{T-t} f) \Gamma (\ln P_{T-t}f)
\\
& +(b'+2\rho_2 a)  (P_{T-t} f) \Gamma^Z (\ln P_{T-t}f) +\frac{4a\gamma}{n} L P_{T-t} f - \frac{2a\gamma^2}{n} P_{T-t} f.
\end{align*}
The idea is now to chose $a,b,\gamma$ such that
\begin{align*}
\begin{cases}
a'+2\rho_1 a -2\kappa \frac{a^2}{b}-\frac{4a\gamma}{n} =0 \\
b'+2\rho_2 a=0
\end{cases}
\end{align*}
With this choice we get
\begin{align}\label{estimetr}
L \phi+\frac{\partial \phi}{\partial t}   \ge \frac{4a\gamma}{n} L P_{T-t} f - \frac{2a\gamma^2}{n} P_{T-t} f
\end{align}

We wish to apply Lemma \ref{P:missing_key2}. We take now $f \in C_0^\infty(\M)$ and apply the previous inequality with $f_\varepsilon=f+\varepsilon$ instead of $f$, where $\varepsilon >0$. If moreover $a(T)=b(T)=0$, we end up with the inequality
\begin{align}\label{jkli}
 & a(0)(P_{T} f_\varepsilon) (x)\Gamma (\ln P_{T}f_\varepsilon)(x)+b(0)(P_{T} f) (x) \Gamma^Z (\ln P_{T}f_\varepsilon)(x) \notag \\
 \le &  -\int_0^T \frac{4a\gamma}{n} dt L P_{T} f_\varepsilon (x)  +\int_0^T \frac{2a\gamma^2}{n}dt  P_{T} f_\varepsilon(x)
\end{align}
If we now chose $b(t)=(T-t)^\alpha$ and $b,\gamma$ such that
\begin{align*}
\begin{cases}
a'+2\rho_1 a -2\kappa \frac{a^2}{b}-\frac{4a\gamma}{n} =0 \\
b'+2\rho_2 a=0
\end{cases}
\end{align*}
the result follows by a simple computation and sending then $\varepsilon \to 0$.
\end{proof}

Observe that if $\rho_1 \ge 0$, then we can take $\rho_1 =0$ and the estimate simplifies to
\begin{align*}
 \Gamma (\ln P_t f) +\frac{2 \rho_2}{\alpha}  t \Gamma^Z (\ln P_t f)  \le & \left(1+\frac{\alpha \kappa}{(\alpha-1)\rho_2}\right) \frac{L P_t f}{P_t f}+\frac{n(\alpha-1)^2\left(
1+\frac{\alpha \kappa}{(\alpha-1)\rho_2}\right)^2}{8(\alpha-2)t}.
\end{align*}

By adapting the classical method of Li and Yau and integrating this last inequality on subunit curves leads to a parabolic Harnack inequality  (details are in \cite{BG}). For $\alpha >2$, we denote
\begin{align}\label{D}
D_\alpha=\frac{n(\alpha-1)^2\left(
1+\frac{\alpha \kappa}{(\alpha-1)\rho_2}\right)}{4(\alpha-2)}.
\end{align}
The minimal value of $D_\alpha$ is difficult to compute, depends on $\kappa, \rho_2$ and does not seem relevant because the constants we get are anyhow not optimal. We just point out that the choice $\alpha=3$ turns out to simplify many computations and is actually optimal when $\kappa=4\rho_2$.

\begin{corollary}
Let us assume that $\rho_1 \ge 0$. Let $f\in
L^\infty(\bM)$, $f \ge 0$, and consider $u(x,t) =
P_t f(x)$. For every $(x,s), (y,t)\in \bM\times (0,\infty)$ with
$s<t$ one has with $D_\alpha$ as in \eqref{D}
\begin{equation*}
u(x,s) \le u(y,t) \left(\frac{t}{s}\right)^{\frac{D_\alpha}{2}} \exp\left(
\frac{D_\alpha}{n} \frac{d(x,y)^2}{4(t-s)} \right).
\end{equation*}
Here $d(x,y)$ is the sub-Riemannian distance between $x$ and $y$.
\end{corollary}

It is classical since the work by Li and Yau  and not difficult  to prove that a parabolic Harnack inequality implies a Gaussian upper bound on the heat kernel. With the curvature dimension inequality in hand, it is actually also possible, but much more difficult, to prove a lower bound. The final result proved in \cite{BBG} is:

\begin{theorem} Let us assume that $\rho_1 \ge 0$, then for any $0<\ve <1$
there exists a constant $C(\ve) = C(n,\kappa,\rho_2,\ve)>0$, which tends
to $\infty$ as $\ve \to 0^+$, such that for every $x,y\in \bM$
and $t>0$ one has
\[
\frac{C(\ve)^{-1}}{\mu(B(x,\sqrt
t))} \exp
\left(-\frac{D_\alpha d(x,y)^2}{n(4-\ve)t}\right)\le p_t(x,y)\le \frac{C(\ve)}{\mu(B(x,\sqrt
t))} \exp
\left(-\frac{d(x,y)^2}{(4+\ve)t}\right),
\]
where $p_t(x,y)$ is the heat kernel of $L$.
\end{theorem}

From the equivalence between Gaussian estimates for the heat kernel and volume doubling properties and Poincar\'e inequalities (see \cite{St1,St2}), this theorem implies the following important result:

\begin{theorem}
Let us assume that $\rho_1 \ge 0$. Then, the metric measure space $(\M,d,\mu)$ satisfies the global volume doubling property and supports a scale invariant 2-Poincar\'e inequality on balls.
\end{theorem}


\begin{thebibliography}{10}

\bibitem{BL}  D. Bakry, M. Ledoux,  \emph{Sobolev inequalities and Myers's diameter theorem for an abstract Markov generator}. Duke Math. J. 85 (1996), no. 1, 253-270.

\bibitem{BGL} D. Bakry, I. Gentil,  M. Ledoux, \emph{Analysis and geometry of Markov diffusion operators}. Grundlehren der Mathematischen Wissenschaften [Fundamental Principles of Mathematical Sciences], 348. Springer, Cham, 2014. xx+552 pp
 
  \bibitem{B2} F. Baudoin, \emph{Sub-Laplacians and hypoelliptic operators on totally geodesic Riemannian foliations}.  Geometry, analysis and dynamics on sub-Riemannian manifolds. Vol. 1, 259--321, EMS Ser. Lect. Math., Eur. Math. Soc., Z\"urich, 2016.
  
  \bibitem{Bau6}  F. Baudoin, \emph{Diffusion processes and stochastic calculus}. EMS Textbooks in Mathematics. European Mathematical Society (EMS), Z\"urich, 2014. xii+276 pp. 
   
     \bibitem{BBG} F. Baudoin, M. Bonnefont, N. Garofalo, \emph{A sub-Riemannian curvature-dimension inequality, volume doubling property and the Poincare inequality}. Math.~Ann. 358 (2014), 3-4, 833--860.
  
  \bibitem{BG} F. Baudoin, N. Garofalo, \emph{Curvature-dimension inequalities and Ricci lower bounds for sub-Riemannian manifolds with transverse symmetries}. J. Eur. Math. Soc. (JEMS) 19 (2017), no. 1, 151--219. 
 
 \bibitem{BG5} F. Baudoin, N. Garofalo, \emph{Curvature-dimension inequalities in Riemannian and sub-Riemannian geometry}, In preparation.
 
  \bibitem{BKW} F. Baudoin, B. Kim, J. Wang, \emph{Transverse Weitzenb\"ock formulas and curvature dimension inequalities on Riemannian foliations with totally geodesic leaves}.  Comm. Anal. Geom. 24 (2016), no. 5, 913--937

\bibitem{BV} F. Baudoin, A. Vatamanelu, \emph{A note on lower bounds estimates for the Neumann eigenvalues of manifolds with positive Ricci curvature.} Potential Anal. 37 (2012), no. 1, 91-101. 

\bibitem{Bra}  M. Bramanti, \emph{An invitation to hypoelliptic operators and Hörmander's vector fields}. SpringerBriefs in Mathematics. Springer, Cham, 2014. xii+150 pp.

\bibitem{Chavel}
I.~Chavel, \emph{Riemannian geometry: a modern introduction}. Cambridge tracts
  in mathematics, 108, Cambridge University Press, Cambridge, 1993.



\bibitem{FP} C. Fefferman, D.H. Phong, \emph{Subelliptic eigenvalue problems}. Conference on harmonic analysis in honor of Antoni Zygmund, Vol. I, II (Chicago, Ill., 1981), 590-606, Wadsworth Math. Ser., Wadsworth, Belmont, CA, 1983.

\bibitem{Fol} G. Folland, \emph{Introduction to partial differential equations.} Second edition. Princeton University Press, Princeton, NJ, 1995. xii+324 pp.

\bibitem{Gri} A. Grigor'yan, \emph{Heat kernel and analysis on manifolds}. AMS/IP Studies in Advanced Mathematics, 47. American Mathematical Society, Providence, RI; International Press, Boston, MA, 2009. xviii+482 pp. 

\bibitem{GT1}   E. Grong,  A. Thalmaier, \emph{Curvature-dimension inequalities on sub-Riemannian manifolds obtained from Riemannian foliations: Part I}. Math.~Z. 282 (2016), no. 1-2, 99--130. 

\bibitem{GT2} E. Grong,  A. Thalmaier, \emph{Curvature-dimension inequalities on sub-Riemannian manifolds obtained from Riemannian foliations: Part II}. Math.~Z. 282 (2016), no. 1-2, 131--164.

\bibitem{Hor} L. H\"ormander, \emph{Hypoelliptic second order differential equations. }, Acta Math. 119 1967 147-171.

\bibitem{Led} M. Ledoux, \emph{The geometry of Markov diffusion generators.} Probability theory. Ann. Fac. Sci. Toulouse Math. (6) 9 (2000), no. 2, 305-366.


 
  
  \bibitem{Petersen} P. Petersen, \emph{Riemannian geometry}. Graduate Texts in Mathematics, vol.~171, Springer-Verlag, 2006. 

\bibitem{St1} K.-T.~Sturm, \emph{On the geometry of metric measure spaces I}. Acta Math. 196 (2006), no.~1, 65--131.
 
\bibitem{St2}  K.-T.~Sturm, \emph{On the geometry of metric measure spaces II}. Acta Math. 196 (2006), no.~1, 133--177. 
 
\end{thebibliography}
\end{document}